\documentclass[11pt]{article}

\usepackage{amsmath, amssymb, amsthm}
\usepackage{mathtools}

\usepackage[T1]{fontenc}
\usepackage{libertine}
\usepackage{comment}

\usepackage{tikz}
\usetikzlibrary{knots,backgrounds}
\tikzset{vblack/.style={circle, fill=black, inner sep=0.8pt}}
\tikzset{vblue/.style={circle, fill=blue, inner sep=0.8pt}}
\tikzset{vred/.style={circle, fill=red, inner sep=0.8pt}}
\tikzset{vgrey/.style={circle, fill=gray, inner sep=0.8pt}}
\tikzset{vempty/.style={circle, draw, inner sep=0.7pt}}
\tikzset{rotate around x=270}

\usepackage{caption, subcaption}
\usepackage{authblk}

\usepackage[colorlinks=true, linkcolor=black, urlcolor=black, citecolor=blue]{hyperref}
\usepackage{fullpage}
\newtheorem{thm}{Theorem}%
\newtheorem{cor}{Corollary} 
\newtheorem{lem}{Lemma} 
\newtheorem{prop}{Proposition} 
\newtheorem{conj}{Conjecture}
\theoremstyle{remark}

\numberwithin{equation}{section}

\usepackage[maxbibnames=5,sorting=nyt,giveninits,sortcites=true]{biblatex}
\renewbibmacro{in:}{}
\DeclareFieldFormat[misc]{title}{\mkbibquote{#1}}

\DeclareFieldFormat[article]{volume}{\mkbibbold{#1}}
\DeclareFieldFormat{url}{}
\DeclareFieldFormat{doi}{\textsc{doi}: \href{http://dx.doi.org/#1}{#1}}
\DeclareFieldFormat{eprint:arxiv}{arXiv:\href{https://arxiv.org/abs/#1}{#1}}
\DeclareFieldFormat{issn}{}
\DeclareFieldFormat{isbn}{}
\DeclareMathOperator{\littleo}{o}
\newcommand{\tube}{\mathbb T}
\newcommand{\tubeLM}{\mathbb T_{M_1,M_2}}
\DeclareMathOperator{\trunk}{trunk}
\newcommand{\eps}{\epsilon}

\title{Entanglement statistics of polymers in a lattice tube and unknotting of 4-plats}

\author[a,*]{Nicholas R. Beaton}
\author[b]{Kai Ishihara}
\author[c]{Mahshid Atapour}
\author[d,e]{Jeremy W. Eng}
\author[f,g]{Mariel Vazquez}
\author[h]{Koya Shimokawa}
\author[d]{Christine E. Soteros}

\affil[a]{School of Mathematics and Statistics, University of Melbourne, Melbourne 3010, Australia}
\affil[b]{
Department of Mathematics, Hiroshima University, Hiroshima 739-8526, Japan }
\affil[c]{Mathematics and Statistics, Capilano University, North Vancouver V7J 3H5, Canada}
\affil[d]{Department of Mathematics and Statistics, University of Saskatchewan, Saskatoon S7N 5E6, Canada}
\affil[e]{(currently) Sask Polytech, Saskatoon S7K 0R8, Canada}
\affil[f]{Department of Microbiology and Molecular Genetics, University of California Davis, Davis CA 95616, USA}
\affil[g]{Department of Mathematics, University of California Davis, Davis CA 95616, USA}
\affil[h]{Department of Mathematics, Ochanomizu University, Tokyo 112-8610, Japan}
\affil[*]{Corresponding Author:
\href{mailto:nrbeaton@unimelb.edu.au}{nrbeaton@unimelb.edu.au}}
\begin{document}
	
	\maketitle

\begin{abstract}
The Knot Entropy Conjecture states that the exponential growth rate of the number of $n$-edge lattice polygons with knot-type $K$ is the same as that for unknot polygons. Moreover, the next order growth follows a power law in $n$ with an exponent that
increases by one for each prime knot in the knot decomposition of $K$. We provide the first proof of this conjecture by considering knots and non-split links in tube $\tube^*$, an $\infty \times 2\times 1$ sublattice of the simple cubic lattice. We establish upper and lower bounds relating the asymptotics of the number of $n$-edge polygons with fixed link-type in $\tube^*$ to that of the number of $n$-edge unknots. For the upper bound, we prove that polygons can be unknotted by braid insertions. For the lower bound, we prove a pattern theorem for unknots using information from exact transfer-matrices. This work provides new knot theory results for 4-plats and new combinatorics results for lattice polygons. 
Connections to modelling polymers such as DNA in nanochannels are  highlighted. 
\end{abstract}


\section{Introduction}

{I}nterest in knotting and linking in lattice polygons is motivated by polymer and biopolymer modelling applications. 
For example, in 1962 Frisch and Wasserman~\cite{Fri61} and Delbr\"uck~\cite{Del62} conjectured that sufficiently long ring polymers in dilute solution are knotted with high probability. The Frisch-Wasserman-Delbr\"uck (FWD) conjecture
was first proved in ~\cite{Pip89,Sum88} for a lattice polygon model of ring polymers. Other proofs followed for several off-lattice polymer models~\cite{Dia94, Diao95, Dia01, Jung94} and for knot diagram models~\cite{EZHNL18,witte_phd_thesis,Chapman19}. 

With the FWD conjecture established for all these models, researchers became interested in the statistics for specific knot-types. Initial results for polygons in the simple cubic lattice followed; simple cubic lattice polygons have vertices in $\mathbb{Z}^3$ and edges of unit length. 
The knot statistics
depend on the number $p_n(K)$ of $n$-edge lattice polygons (up to translation) with fixed knot-type $K$, and its dependence on $K$ and $n$, as $n\to\infty$.  
These counts are lattice dependent and have been studied for the simple cubic as well as other lattices \cite{Rechnitzer, rechnitzer_rensburg}. 
Based on polymer scaling theory,
Orlandini, Tesi, Janse van Rensburg and Whittington \cite{Orl96,Orl98} 
made a conjecture about the asymptotic growth of $p_n(K)$ as $n\to \infty$. 
They predicted that the leading order term of $p_n(K)$ has the form $\mu_{0_1}^n$, where the constant $\mu_{0_1}$ depends only on the lattice and is the \emph{same for every $K$,} including the unknot. They also conjectured that the
 next order growth follows a power law in $n$ with an exponent that
increases by one for each prime knot in the knot decomposition of $K$. 
This 
\emph{Knot Entropy (KE) Conjecture} 
for lattice polygons
is stated formally as follows.
\begin{conj}[Knot Entropy (KE) Conjecture \cite{Orl96,Orl98}]
	
 For any given knot-type $K$, as $n\to\infty$
	$p_n(K)$ satisfies the asymptotic equation
	\begin{equation}
		p_n(K) = B_Kn^{f_K}p_n(0_1)(1+\littleo(1)), \label{eqn1}
	\end{equation}
	where $0_1$ is the unknot, $f_K$ denotes the number of prime knot factors in the knot decomposition of $K$, and $f_{0_1}\equiv 0$.
	Moreover,
	\begin{equation}
		p_n(0_1) = A_0n^{\alpha_{0_1}}\mu_{0_1}^n(1+\littleo(1)), \label{eqn2} %
	\end{equation}
	where $\displaystyle{\alpha_{0_1}}$ is 
	called
	the \textbf{unknot power-law exponent (or entropic critical exponent)} and
	$\mu_{0_1}$ is called
	the \textbf{unknot exponential growth constant}. 
	The constants $B_K$, $A_0$ and $\mu_{0_1}$ are potentially lattice-dependent. Furthermore $\kappa_{0_1}=\log\mu_{0_1}$, the \textbf{unknot growth rate}, is the limiting entropy per monomer for the unknot.
	\label{mainconj}
\end{conj}

There is strong numerical evidence in support of the KE Conjecture~\cite{Rechnitzer,Orl98}. However an analytical proof has been lacking.
Let 
$\displaystyle{p_n=\sum_{K} p_n(K)}$
be the number of $n$-edge lattice polygons regardless of knot type and let $\mu$ be its exponential growth constant. It has been shown that the limits defining $\mu_{0_1}$
and $\mu$ exist
and that $\mu_{0_1}<\mu$~\cite{Sum88}.
The existence of exponential growth constants for non-trivial knots $K$
and of the exponent $\alpha_{0_1}$ are open problems. 
A positive answer to the KE Conjecture, together with a corresponding result for 
$p_n$, would imply that the probability $p_n(K)/p_n$ of a random $n$-edge polygon having knot-type $K$ scales with $n$ like $C_K(\frac{\mu_{0_1}}{\mu})^n n^{\sigma +f_K}$, for $\sigma$ independent of $K$. It is believed that $\sigma=0$ for lattice polygons \cite{Rechnitzer}.
A similar scaling form is expected to hold for off-lattice polygons \cite{Uehara2017,Xiong2021,Cantarella2024}.
For fixed $n$ and variable $f_K$, this scaling form for the knot probabilities suggests that the prime factors of a knot are  distributed randomly along an $n$-edge ring polymer according to a Poisson-like distribution, somewhat like random pearls on a string. The latter is consistent with the expectation that knot factors are on average localized in polymers. 

In this paper we consider the applicability of the KE Conjecture to confined lattice polygons. 
The study of knot and link statistics of polymers in confinement is largely motivated by problems in molecular biology, such as DNA packing in viruses and human cells, or DNA transport through nanochannels~\cite{Arsuagacapsid,Plesa_2016, Arsuaga-fmolb2015,micheletti_polymers_2011,Orlandini_2021}. Different types of confinement are of interest. Researchers have modeled polymers under ``full'' confinement as in a sphere or capsid, or by restricting a single direction (slab, e.g.~\cite{micheletti_polymers_2011,Orlandini_2021}), or by restricting only two directions (nanochannel or tube). 
While the KE Conjecture is not expected to hold for full confinement, 
numerical studies of 
lattice tube polygons
have led to the supposition that the conjecture holds for
lattice tubes~\cite{BES19}. %
Here, we prove the KE Conjecture for the smallest  tube that admits non-trivial knots. We have published a rough sketch of the proof in a Letter \cite{beaton_first_2024}. The Letter was directed towards a broad physics audience; here we give the full details of 
the mathematical results and proofs, which are 
primarily
combinatorial in nature.

\begin{figure}[htb]
\begin{center}
 \includegraphics[width=.8\textwidth]{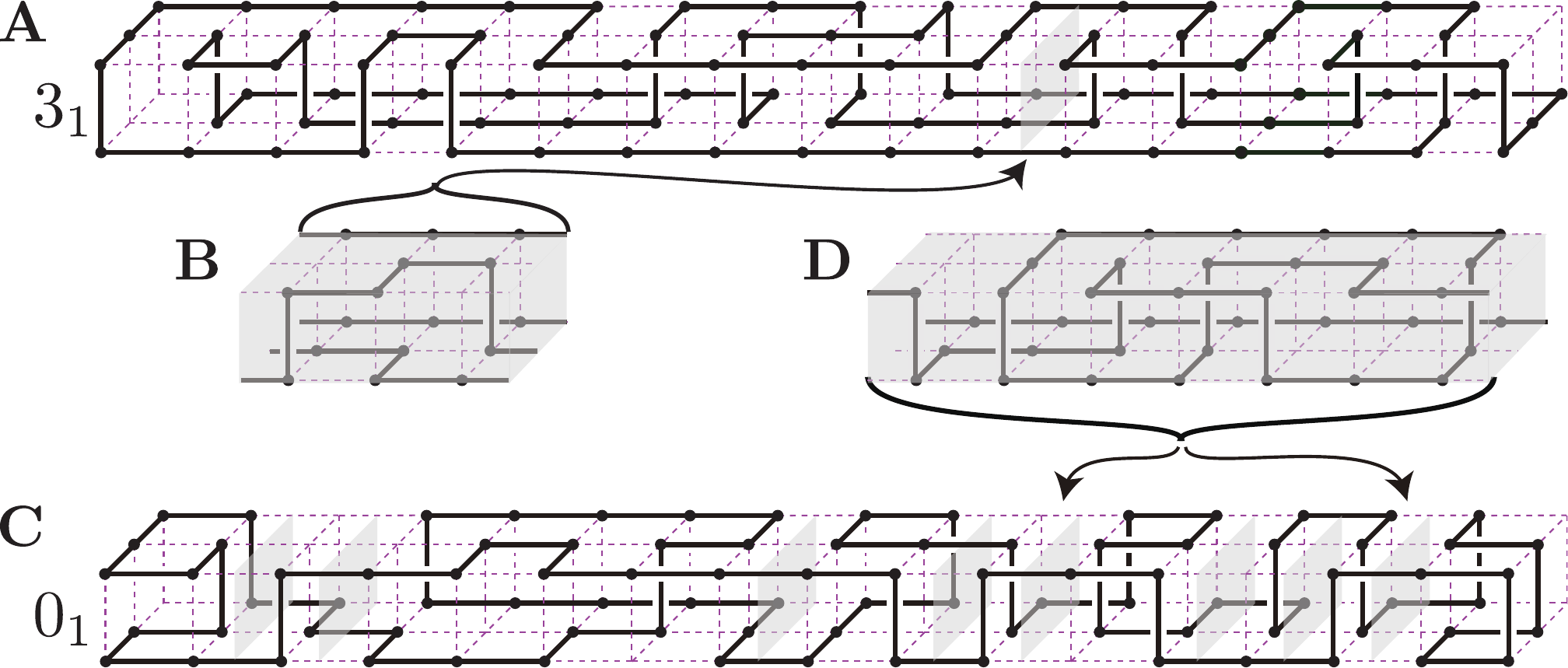}
 \caption{{\bf A.}~\textbf{Example for the upper bound in Theorem \ref{thm-bounds}, equation \ref{mainrelation}.} 
 A 98-edge polygon of knot-type $3_1$ (trefoil) in $\tube^*$; 
 {\bf B.} An insert enclosed in a 3-block. The insert is a 4-braid containing a half twist of two strings. When inserted into (A) at the identified location, this 3-block converts the $3_1$ into an unknot. 
 {\bf C.}~\textbf{Example for the lower bound in Theorem \ref{thm-bounds}, equation \ref{mainrelation}.} The figure shows an unknot lattice polygon with eight highlighted 2-sections. 
{\bf D.}~The shaded image is a trefoil pattern enclosed in a $7$-block. When this pattern is inserted at any of the 2-sections indicated by arrows, the unknotted polygon turns into a trefoil polygon in $\tube^*$.  Note that one can insert the trefoil pattern at any of the other 2-sections of the original polygon by modifying its ends. Figures (A) to (D) have been modified from 
 \cite[Fig. 1]{beaton_first_2024}.}
 \label{fig1intro}
 \end{center}
 \end{figure}

Lattice tube models have been useful for modelling polymers under confinement since the 1970s~\cite{hammersley1985self,WalletalPNAS1978}. 
The smallest tube
of $\mathbb{Z}^3$ that admits non-trivial knots is the $\infty\times 2\times 1$ lattice tube, denoted here by $\tube^*$~\cite{minsteptube}. 
The prime factors of any non-trivial knot admitted in $\tube^*$ must be 
 4-plats (2-bridge knots)~\cite{minsteptube}. Figure~\ref{fig1intro}A illustrates a knotted polygon in $\tube^*$. 
Even though $\tube^*$ is narrow, polygons in $\tube^*$ are physically relevant for modelling 
DNA in nanochannels. 
Importantly all prime knots with 7 or fewer crossings and the majority %
for up to 9 crossings are 4-plats. 
The family of 4-plats includes all twist and torus links. Such links are most likely to be observed in DNA experiments~\cite{sumners_ernst_spengler_cozzarelli_1995}.

The main result of this paper is given in Theorem \ref{thm-bounds}.  This result yields Equation~\eqref{eqn1} in the KE Conjecture, 
up to the leading constant $B_K$, 
for self-avoiding polygons in $\tube^*$.   
That is, we prove  upper and lower bounds on $p_n(K)$, 
 differing only by a constant factor.
We establish similar results for 
multi-component non-split links 
in $\tube^*$.
Since a knot is a 1-component non-split link,
from here on we will only distinguish between knots and multi-component non-split links if necessary, or to connect to the KE Conjecture.

\begin{thm}\label{thm-bounds}
	Let $L$ be any non-split link embeddable in the $2\times 1$ tube $\tube^*$.
	Then, for non-trivial $L$ there exist positive constants $\epsilon\in(0,1), b_{L}\in\mathbb{R},d_L\in\mathbb{Z},e_L\in\mathbb{Z}$ (independent of $n$) and an integer $N_{L,\epsilon}>0$ such that for any $n\geq N_{L,\epsilon}$, 
	there are 
	bounds on $p_{\tube^*,n}(L)$, the number of $n$-edge embeddings of $L$  in $\tube^*$, as follows:
	\begin{equation}\begin{split}
			\frac{1}{2}\binom{\lfloor \epsilon (n-e_L)\rfloor}{f_L} p_{\tube^*,n-e_L}(0_1) &\leq p_{\tube^*,n}(L) %
			\leq b_L \binom{n}{f_L} p_{\tube^*,n+d_L}(0_1),
			\label{mainrelation}
	\end{split}\end{equation}
 where $f_L$ is the number of prime link factors in $L$. There exist 
	$C_1$, $C_2$ such that  
	for all 
	sufficiently large $n$
	\begin{equation}
		C_1n^{f_L} p_{\tube^*,n}(0_1) \leq    p_{\tube^*,n}(L) \leq C_2n^{f_L} p_{\tube^*,n}(0_1). 
		\label{finalform}
	\end{equation}
\end{thm}

Note that if $L$ is a knot $K$ embeddable in $\tube^*$, inequality \eqref{finalform} gives the KE Conjecture Equation~\eqref{eqn1}, up to the term $B_K$.
If $L$ is a non-split link embeddable in $\tube^*$, the bounds of Theorem~\ref{thm-bounds} 
establish these important consequences:
{\it the exponential growth constant for embeddings of $L$ in $\tube^*$ is the same as that for unknot polygons} (as in the KE Conjecture, Equation~\eqref{eqn2}), and 
{\it the power-law
exponent increases by one for each prime factor of $L$}. 
The last statement assumes the existence of a
limit that defines a power-law 
exponent 
for the unknot in $\tube^*$. 
While the exponential growth constant $\mu_{\tube^*, 0_1}$for unknots in $\tube^*$ is known to exist (this relates to ~\eqref{eqn2}), 
the existence of the power-law exponent for unknots in $\tube^*$ 
is an open problem. 

The results presented here about the exponential growth constants and 
power-law
exponents for lattice polygons with fixed knot-type are believed to hold for 
any tube size, as well as in the limit where the tube dimensions go to infinity (see~\cite{Orl98,Rechnitzer}). Here we provide the first proofs, with \emph{$\tube^*$ being the first model for which these results have been confirmed.} 
In the case of any non-split link $L$ with 2 or more components, all of which are unknots, it is known that the exponential growth constant for embeddings of $L$ in $\mathbb{Z}^3$ (i.e.\ the whole lattice) is equal to that of the unknot~\cite{Sot99}. 
The 
tube $\tube^*$ is the first model for which there are results on the growth constants for all multi-component non-split links,
and results on how the power-law (entropic) exponent changes with the link-type.
Recent numerical evidence~\cite{Bonato_2021} for multi-component non-split links 
embedded in $\mathbb{Z}^3$ suggests that Theorem~\ref{thm-bounds} also holds in the unconfined case. 

In this paper we present the complete proof of Theorem~\ref{thm-bounds}.
To obtain the bounds in Theorem~\ref{thm-bounds} Equation~\eqref{mainrelation}, we consider ways to insert polygon  
segments 
that convert an embedding of one link-type into an embedding of another link-type.  The relevant polygon segments (blocks), the insertion process and the mathematical approach to prove the upper bound are different from those for the lower bound. 

The goal for the upper bound is to convert a  non-trivial non-split link $L$ in $\tube^*$ to an unknot polygon. We show an example of such conversion in Figure~\ref{fig1intro}A-B. 
Our approach takes advantage of the fact that all links that fit in $\tube^*$ are connected sums of 2-bridge links  \cite{minsteptube}.
We prove new knot theory results that, in combination with properties of embeddings in $\tube^*$, allow us to find $f_L$ locations in an embedding of $L$, where one can insert $f_L$
polygon segments (blocks) 
and change the embedding into an unknot. 
The binomial term of the upper bound comes from bounding the number of ways the insertion process on different initial embeddings of $L$ could yield the same unknot polygon. 

The goal for the lower bound is to convert an unknot polygon to an embedding of a non-trivial link $L$. This process is illustrated in Figure~\ref{fig1intro}{}{C-D}. 
Our approach takes advantage of the existence of a transfer matrix for polygons (regardless of knot type) in $\tube^*$. We use this, combined with properties of unknots in $\tube^*$, to establish a pattern theorem for unknots in $\tube^*$. From this, we can identify on the order of $\binom{n}{f_L}$ locations where   $f_L$ polygon blocks can be inserted in an $n$-edge unknot polygon to convert it to an embedding of $L$ in $\tube^*$.  

There are interesting connections between the scaling form established by Theorem~\ref{thm-bounds}
and the ``size'' of the linked region in polymers.  Our proof of 
Theorem~\ref{thm-bounds} leads to the proof of general pattern theorems for embeddings with fixed link-type. These general pattern theorems allow us to prove that knotting and linking is typically localized for lattice links in $\tube^*$.   We can also explore different modes of knotting and linking such as the modes identified in \cite{BEISS18,Suma17}. 

%
The paper is structured as follows. 
Section~\ref{sec:BackDef} provides the main background regarding Conjecture~\ref{mainconj}{} and the definitions needed for our results. 
We give details for the upper bound proofs in Sections~\ref{unknotting} and~\ref{upperbound}. 
Notably, Section~\ref{unknotting} includes proofs of the novel knot theory results. In Section~\ref{sec:pattern} we provide details for the lower bound proofs.  
In Section~\ref{sec:Overview} 
we bring together the results from Sections~\ref{sec:newupper} and~\ref{sec:pattern} to prove Theorem~\ref{thm-bounds}. 
In Section~\ref{sec:size} we discuss consequences of Theorem~\ref{thm-bounds}, including 
knot/link 
localization and different modes of linking.
The discussion in Section~\ref{sec:discussion} considers the impact of our results on different disciplines and potential applications to the study of DNA.  

\section{Background and definitions}\label{sec:background}
\label{sec:BackDef}

In this section we introduce some terminology and results that are needed to give a more detailed overview of the main theorems and their proofs. 
We start by introducing definitions for lattice polygons and review the evidence leading to Conjecture~\ref{mainconj}{}.  We then introduce definitions for lattice polygons in tubes and restate  Conjecture~\ref{mainconj}{} for tube models.  Finally, we focus on $\tube^*$ and review  known knot theory results regarding 4-plats. 

Since we are focusing on the simple cubic lattice (a crystallographic lattice) we start by giving a more precise definition of this lattice.  It has \textit{vertices} (0-skeleton) which are the integer points in $\mathbb{R}^3$ and its 1-skeleton is the set of \textit{edges} joining pairs of vertices unit distance apart. (Equivalently the lattice can be thought of as $(\mathbb{R}\times \mathbb{Z}\times \mathbb{Z}) \cup (\mathbb{Z}\times \mathbb{R}\times \mathbb{Z}) \cup (\mathbb{Z}\times \mathbb{Z}\times \mathbb{R})$ with  lattice vertices being the points in $\mathbb{Z}^3$ and  lattice edges being the unit length lines between pairs of vertices.) As is standard, we refer to this lattice by its vertex set $\mathbb{Z}^3$.  Similarly a sublattice of $\mathbb{Z}^3$ is denoted by its vertex set whenever the edge set is induced by the vertices.  

A \textit{self-avoiding polygon} in $\mathbb{Z}^3$ (lattice polygon or polygon, for short) 
is the image of an embedding of one simple closed curve in $\mathbb{Z}^3$.
A \textit{lattice link }
is a disjoint union of lattice polygons.
We will refer to a lattice link with link type $L$ as a lattice embedding of 
$L$, or simply as an \emph{embedding of $L$}.   
The \textit{size} of a lattice embedding of a link is defined to be the number of edges, which is always even.

For polygon enumeration, we consider two polygons in $\mathbb{Z}^3$ to be the same if they are translates of each other. Let $p_n$ be the 
number of distinct $n$-edge polygons in  $\mathbb{Z}^3$.
Hammersley \cite{Ham61} proved that $p_n = e^{ n \log\mu + \littleo(n)}$
where $\mu$ is the \emph{exponential growth constant} of the lattice; that is, Hammersley proved that the following limit that defines $\mu$ exists:
\begin{equation}
	\mu = \lim_{n\to\infty} (p_n)^{1/n}.
\end{equation}
In a similar way it can be shown \cite{Sum88} that the number of 
unknot polygons $p_n(0_1)$ with $n$ edges satisfies the equation
$p_n(0_1)=e^{n \log \mu_{0_1}  +\littleo(n)}$ and pattern theorem arguments \cite{Kesten}
can be used to show that $\mu_{0_1} < \mu$ \cite{Pip89,Sum88}. This establishes the FWD conjecture.  
Similar to Conjecture~\ref{mainconj}{}, there is evidence that 
\begin{equation}
	\label{allpolyasymptotics}
	p_n= A n^{\alpha}\mu^n(1+\littleo(1)),  ~~~~n\to\infty,
\end{equation}
where $A>0$ and $\alpha$ is called the power-law exponent or the entropic critical exponent for lattice polygons.\footnote{In the statistical mechanics literature the exponent of $n$ in~\eqref{allpolyasymptotics} is usually written as $\alpha-3$ instead of $\alpha$; for simplicity we will stick with $\alpha$ here.}  It is also expected that $\alpha$ is not lattice-dependent while $A$ and $\mu$ are. Note that, aside from the existence of the limit defining $\mu$, little is known rigorously regarding~\eqref{allpolyasymptotics}.  In particular, it has yet to be established that the limit that would define $\alpha$, $\displaystyle{\lim_{n\to\infty}  \frac{\log(p_n/\mu^n)}{\log n}}$, exists.

Now consider $p_n(K)$, the number of $n$-edge polygons with fixed knot-type $K$. For $\mathbb{Z}^3$, it has been established that every knot is embeddable \cite[Theorem 2.4]{SSW92} and similar arguments work for multicomponent links. It is then straightforward to 
show that there exists a fixed integer $b$ 
(related to the minimum size of an embedding of $K$) 
such that $p_n(0_1)\leq p_{n+b}(K)$ \cite{SSW92} but little else is known rigorously about $p_n(K)$. There is however strong numerical
evidence from Monte Carlo simulations  that Conjecture~\ref{mainconj}{} holds, i.e.\ that as $n\to\infty$, 
\begin{equation}
	\label{fixedknotasymptotics}
	p_n(K)\sim A_K n^{\alpha_{0_1} + f_K}(\mu_{0_1})^n(1+\littleo(1)).
\end{equation}
Furthermore, 
there is numerical evidence that $\alpha_{0_1}=\alpha$ and that for prime knots $K_1$ and $K_2$, the amplitude ratios $A_{K_1}/A_{K_2}$ are lattice-independent, i.e.\ ``universal'' \cite{Rechnitzer}.
However, except in the case that $K=0_1$, it has yet to be proved that $\lim_{n\to\infty} (p_n(K))^{1/n}$ even exists, although it is known that the corresponding $\liminf$ is bounded below by $\mu_{0_1}$ and the $\limsup$ is strictly less than $\mu$ \cite{SSW92}. 

In general, while it is known that $\mu_{0_1} < \mu$, there is no known approach for exactly determining either $\mu$, $\mu_{0_1}$ or their ratio $ \mu_{0_1}/\mu$, although there are estimates for these quantities based on various numerical approaches \cite{Rechnitzer,Orl98}.

Recently \cite{Bonato_2021} considered embeddings of  $k$-component non-split links in $\mathbb{Z}^3$ for any fixed $k>1$.  In the case that all components are unknots, they proved that the exponential growth constant is equal to $\mu_{0_1}$ and otherwise they obtained that the exponential growth constant, if it exists, is strictly less than $\mu$ and bounded below by  $\mu_{0_1}$.  They provide numerical evidence that is consistent with the exponential growth constant being independent of link type and a scaling form consistent with that of Conjecture~\ref{mainconj}{}.  

\subsection{Knots and links in lattice tubes}
\label{ssec:knotslinkstubes}
In this paper we make progress on proving  
the KE Conjecture, now generalized to any non-split link,  by focusing on the tubular sublattices of the simple cubic lattice.  
These have been studied previously in various contexts \cite{Alm90, atapourphdthesis, Atapour09, Atapour10, Beaton_2016, BEISS18, BES19, Eng2014Thesis, Sot98, Sot06, SoterosWhittington}. 
Unless stated otherwise, the notation and definitions used here are as in~\cite{Beaton_2016}. 

For positive integers $M_1,M_2$,  the semi-infinite sublattice of  $\mathbb{Z}^3$ induced by the vertex set 
\begin{equation}
	\{(x,y,z)\in \mathbb Z^3:x\geq0, 0\leq y\leq M_1, 0 \leq z \leq M_2\}
\end{equation}
is called the $M_1\times M_2$ \emph{tube} and denoted by \(\tubeLM\equiv\tube\subset \mathbb Z^3\). We are interested in %
lattice links in $\tube$ and 
restrict to those which occupy at least one vertex in the plane $x=0$. Let the \emph{span} $s(\pi)$ of a lattice link $\pi$ be the maximal $x$-coordinate reached by any of its vertices.

Counts of polygons by size in lattice tubes have been well-studied.
Let $p_{\tube,n}$ be the number of $n$-edge 
self-avoiding polygons in  $\tube$ which occupy at least one vertex in the plane $x=0$.  
See Figure~\ref{fig1intro}A for an embedding $\pi$ of the knot-type $3_1$ in a $2\times1$ tube with span $s(\pi)=16$ and size $|\pi|=98$.

One major advantage of focusing on  $\tube=\tube_{M_1,M_2}$ is that transfer-matrix arguments have been used to prove 
\cite{Sot98}:
\begin{equation}
	p_{\tube,n}=A_{\tube}\mu_{\tube}^n(1+ \littleo(1)).
\end{equation}
Hence in $\tube$, not only does the limit defining the exponential growth constant $\mu_{\tube}$ for polygons in $\tube$ exist,
\begin{equation}
	\mu_{\tube}=\lim_{n\to\infty} (p_{\tube,n})^{1/n},
\end{equation}
but we can actually prove that the asymptotic form of~\eqref{allpolyasymptotics} holds where, in this case, the 
power-law
exponent $\alpha_{\tube}=0$.
Furthermore, given sufficient computational resources, $A_{\tube}$ and $\mu_{\tube}$ can 
be determined to arbitrary accuracy using the eigenvalues and eigenvectors of the associated transfer-matrix. For the case
of $\tube^*$, i.e.\ $(M_1,M_2)=(2,1)$, the values are respectively 2.330946 $\times 10^{-4}$ and $\log(\mu_{\tube})=\frac{1}{0.498950}$.

Now let $p_{\tube,n}(K)$ be the number of polygons with knot type $K$ counted in  $p_{\tube,n}$. 
It is known that
$p_{\tube,n}(0_1) = (\mu_{\tube,0_1})^n e^{o(n)}$,
where $\mu_{\tube,0_1}$ is the unknot exponential growth constant given by
\begin{equation}
	\label{unknottubegrowthconstant} 
	\log \mu_{\tube,0_1}=\lim_{n\to\infty}\frac{1}{n}\log p_{\tube,n}(0_1).
\end{equation}
Further, in the $1\times 1$ tube  $\mu_{\tube,0_1}=\mu_{\tube}$ and otherwise  $0<\mu_{\tube,0_1}<\mu_{\tube}$ is known  from a pattern theorem argument \cite{Sot98}.
Monte Carlo methods have been used \cite{BES19, eng_phd_thesis} to provide strong evidence for 
the KE Conjecture
with $\alpha_{\tube,0_1}=0$ for polygons in each of the tube sizes: $2\times 1$, $3\times 1$, 
$4\times 1$, $5\times 1$, $2\times 2$ and $3\times 2$.
It is also known (by arguments analogous to those in~\cite{hammersley1985self,SSW2012}) that
\begin{equation}
	\label{unknottubegrowthconstantlimit} 
	\mu_{0_1}=\lim_{M\to\infty}\mu_{\tube_{M,M},0_1},
\end{equation}
so that proving 
the KE Conjecture
for arbitrary tube sizes could lead to a proof for the unconfined case. 

More generally,
we define $p_{\tube,n}(L)$ to be the number of $n$-edge embeddings of a non-split link type $L$ in $\tube$ having
at least one vertex in the plane $x=0$. We do not consider split links $L$,  i.e.\ two or more separable simple closed curves, since their $n$-edge embedding counts are not finite.
For general $\tube$, little is known about $p_{\tube,n}(L)$ when $L\neq 0_1$.

A first question of interest is the determination of which non-split links are embeddable in 
$\tube_{M_1,M_2}$.   
To discuss the answer,
let $h:\mathbb R^3 \to \mathbb R$ be the projection to the $x$-axis, that is $h(x,y,z)=x$.
The link invariant {\em $\trunk(L)$} is defined by
$\trunk(L)=\min_{E}\max_{t\in \mathbb R} |h^{-1}(t)\cap E|$,
where $E$ is an embedding of link $L$ in $\mathbb R^3$ \cite{minsteptube, Ozawa}.
By showing that the trunk of a link limits the size of the smallest tube that can contain it, in \cite{minsteptube}, the link types that can be confined in $\tube_{M_1,M_2}$ are characterized.  In particular the following proposition has been proved:

\begin{prop}[{\cite[Theorem 1]{minsteptube}}]
	A link type $L$ can be confined to $\mathbb{T}_{M_1,M_2}$
	if and only if $\trunk(L) <(M_1+1)(M_2+1)$.
\end{prop}

Note that any non-trivial non-split link has trunk greater than 3 \cite{Ozawa}, so that $\mathbb{T}^*=\mathbb{T}_{2,1}$ is the smallest tube that admits non-trivial non-split links.

Arguments analogous to those of \cite{Bonato_2021} can be used to establish that for $L$ with all components being unknots, $p_{\tube,n}(L)=(\mu_{\tube,0_1})^n e^{o(n)}$.
Based on the numerical evidence available for knots in tubes, we make the following conjecture for any tube size and any non-split link.

\begin{conj}[KE Conjecture for Lattice Tubes \cite{BES19}]\label{mainconjtube}
	For a given $\tube=\tube_{M_1,M_2}$ and any given non-split link $L$ embeddable in $\tube$, there exist constants (independent of $n$ but potentially dependent on $M_1,M_2$ and $L$) $A_{\tube,L}$, $\alpha_{\tube,L}$ and $\mu_{\tube,0_1}$, such that $p_{\tube,n}(L)$  satisfies the following asymptotic equation:
	\begin{equation}
		p_{\tube,n}(L)=A_{\tube,L}n^{\alpha_{\tube,L}}(\mu_{\tube,0_1})^n(1+\littleo(1)),  ~~~n\to\infty,
	\end{equation}
	where it is expected that $\alpha_{\tube,0_1} = 0$ and that: the \textbf{amplitude} $A_{\tube,L}>0$; 
	(a) the \textbf{exponential growth constant} 
	$\displaystyle{ \mu_{\tube,0_1}>0; }$ 
	and
	(b) the \textbf{power-law
 exponent}
	$\displaystyle{\alpha_{\tube,L}=\alpha_{\tube,0_1}+f_L.}$
\end{conj}

To discuss embeddings in tubes further, some additional definitions are needed. 
These definitions are as in~\cite{BEISS18} but generalized to embeddings of non-split links. For simplicity, unless stated otherwise, the term \textit{embedding} will henceforth refer to any embedding of a non-split link in $\tube$.
Given an embedding $\pi$ in $\tube$ and $k\in\mathbb Z$, a \emph{hinge} $H_k$ of $\pi$ is the set of edges and vertices lying in the intersection of $\pi$ and the $y$-$z$ plane defined by $\{(x,y,z):x=k\}$. A \emph{section} $S_k$ is the set of edges in $\pi$, in the $x$ direction, connecting $H_{k-1}$ and $H_k$. A \emph{half-section} of $S_k$ is the set of half-edges in $S_k$ with either $k-1\leq x\leq k-\frac12$ or $k-\frac12\leq x\leq k$. Any section of an embedding which contains exactly $r$ edges is called an {\it $r$-section} of the embedding.

A \emph{1-block} of $\tube$ is any non-empty hinge which can occur in an embedding $\pi$ in $\tube$ together with the half-edges of $\pi$ in the two adjacent half-sections. The \emph{length} of a 1-block is the sum of the lengths of all its {embedding} edges and half-edges. It is thus natural to view a 1-block as the part of an embedding between two half-integer $y$-$z$ planes $x=k\pm\frac12$ for some $k\in\mathbb Z$.

An \emph{$s$-block} is then 
defined to be 
any connected sequence of $s$ 1-blocks, the entirety of which can occur in an embedding in $\tube$. (It is also possible, if the first and last half-sections of the $s$-block are empty, for the $s$-block itself to be an embedding.) The length of an $s$-block is the sum of the lengths of its constituent 1-blocks. 
See Figure~\ref{fig1intro}{}{B} and {D} 
for examples of a 3-block and 7-block respectively.

If $\sigma$ is an $s$-block whose first and last half-sections each contain exactly two half-edges, 
we call $\sigma$ a {\it connected sum pattern}.  Given any connected sum pattern $\sigma$, we can derive an embedding $\pi_{\sigma}$ in $\tube$ from $\sigma$ by first extending each half-edge into a full edge and then joining the two endpoints in the left-most hinge, and then the two endpoints in right-most hinge, using a shortest path in the tube. If the link-type  of $\pi_\sigma$ is  $L\neq 0_1$ (the unknot), then we say that $\sigma$ is a {\it link-pattern} (we may also sometimes say \emph{knot-pattern}) and otherwise it is an {\it unknot-pattern}.  For example, Figure~\ref{fig1intro}{}{D} 
shows  a connected sum pattern $\sigma$ which is a knot-pattern because $\pi_\sigma$ has knot-type $K=3_1$ (the trefoil).  Similarly, if $\sigma$ is an $s$-block whose first (last) half-section is empty and its last (first) half-section contains exactly two half-edges, 
with all other sections containing more than two edges, we call $\sigma$ a {\it start connected sum pattern}  ({\it end connected sum pattern}).  Following the procedure for connected sum patterns,  start and end patterns can also be classified as either link or unknot patterns.  
By dividing at each 2-section, an embedding $\pi$ in $\tube$ with $k\geq 1$ 2-sections can be decomposed into  a start connected sum pattern, a  sequence of $k-1$ connected sum patterns and then an end-connected sum pattern (see for example Figure~\ref{fig1intro}{}{C}, 
which shows an unknot polygon which has $k=8$ 2-sections (indicated by the shaded planes));
hence, using the embeddings derived from each pattern, $\pi$ can be viewed as a topological connected sum of embeddings.   In particular, if $\pi$ is a polygon, then the knot-type of $\pi$ is completely determined by the knot-types of the polygons in this connected sum. More generally, if $\pi$ is an embedding of a non-split link $L$, then at most $f_L$ of the embeddings derived from the connected sum patterns are not unknot polygons. For any $\tube$ it has been established that all prime links embeddable in the tube have a corresponding link-pattern in the tube. 
This gives the following Proposition.

\begin{prop}[{\cite[Result 4]{BEISS18}}]
\label{prop:knotpattern}
For each prime link $L$ embeddable in $\tube$, there is a  link-pattern of $L$.
\end{prop}

 Figure~\ref{fig1intro}{C} and {D} 
 illustrates how a $3_1$ %
 link pattern
 can be inserted at a 2-section of 
 an unknot polygon to convert it into a $3_1$ polygon. An unknot polygon with at least one 2-section can be decomposed 
 into the connected sum of two unknots. Any link pattern can then be ``inserted'' in between, 
 as in Figure~\ref{fig1intro}{D}. 
 In the next subsection we present implications of this for $\tube^*$.

\subsection{4-plats, 4-braids and the tube \texorpdfstring{${\mathbb{T}^*}$}{T*}}
\label{subsec:4-plats}

Our results rely heavily on the following proposition. Here we provide the needed background.

\begin{prop}{\cite[Corollary 2]{minsteptube}}\label{prop:which_knots_fit}
	If a link type $L$ can be embedded in $\mathbb{T}^*$ then each prime factor of $L$ is a $4$-plat. 
	
	Furthermore, there is an embedding of $L$ in  $\mathbb{T}^*$ which consists of a connected sum of $f_L$ link-patterns, one corresponding to each prime factor of $L$.
\end{prop}
\begin{cor}\label{cor:unknot_insert_factors}
	Consider a link $L$ embeddable in $\tube^*$. 
	Any unknot polygon with at least one 2-section can be converted to an embedding of $L$ by $f_L$ insertions of   link-patterns
	at one or more of the 2-sections of the polygon. 
\end{cor}
Thus, for $\tube^*$ we know how to convert an unknot polygon with 2-sections into an embedding of a link $L$. 
For the lower bound in Theorem~\ref{thm-bounds}, we need to know that most unknot polygons contain a sufficient number of 2-sections. In Section~\ref{sec:newlower} we prove this for $\tube^*$.

For the upper bound in Theorem~\ref{thm-bounds}, we obtain new results about 4-plats. Note first that 4-plats are links defined as closures of 4-braids. The family of 4-plats is the same as that of 2-bridge links. 4-plats can be represented by 4-plat diagrams
(reviewed in \cite{BZ13}).

A {\it $4$-braid} can be defined as four disjoint strings in a rectangular cuboid, where the strings start at four points in the left ($x=0$)  face of the cuboid 
and end at four points in the opposite (right) face. Each string is required to run strictly rightwards, i.e.\ for any $c$ the string meets the plane $x=c$ at most once. 
The $4$-braid is studied using a {\it $4$-braid diagram} obtained by projecting the braid onto the $xy$-plane and resolving over and under crossings. 
A $4$-braid diagram with no crossings is said to be {\it trivial}. Any $4$-braid, except the trivial one, can be obtained by joining {\it elementary braids}
$\sigma_1,\sigma_2,\sigma_3, \sigma_1^{-1},\sigma_2^{-1}$ and $\sigma_3^{-1}$ 
(see Figure~\ref{fig:artin_and_4-plat_closure}). 
A sequence of the letters $\sigma_i^{\pm1}$, 
called a {\it $4$-braid word}, represents a $4$-braid (and its corresponding 4-braid diagram). The {\it empty word} with no letters represents the \textit{trivial braid}. The sequence of letters 
used for a braid word is written in simple exponential form, say $\sigma_i^a$ or $\sigma_i^{-a}$. 
Two $4$-braids are equivalent if they are related by ``level preserving'' isotopies, but they may be represented by different $4$-braid diagrams and braid words. In particular, we consider that a reducible word, $\sigma_i^a\sigma_i^{-b}$ or $\sigma_i^{-b}\sigma_i^{a}$ ($a,b\in\mathbb{N}$), is different from the reduced one, $\sigma_i^{a-b}$. 
There are two ways to close a $4$-braid diagram at each end 
to form a $4$-plat diagram; these are 
denoted by $[_1,[_2$ on the left and by $]_1,]_2$ on the right, as in Figure~\ref{fig:artin_and_4-plat_closure}. 
Such a closure of the corresponding braid word $w$ results in a {\em $4$-plat diagram} $[_i w]_j$, where $i,j\in\{1,2\}$.
For example, the $4$-plat diagram $[_1\sigma_1^{-1}\sigma_1\sigma_3\sigma_2\sigma_3^{-2}\sigma_2^{-1}\sigma_3^{-1}\sigma_3\sigma_2^{-1}\sigma_3^{-1}\sigma_3\sigma_1]_1$ depicted in Figure~\ref{fig2}C represents the knot $3_1$.

\begin{figure}[htb]
\begin{center}
 \includegraphics[width=\textwidth]{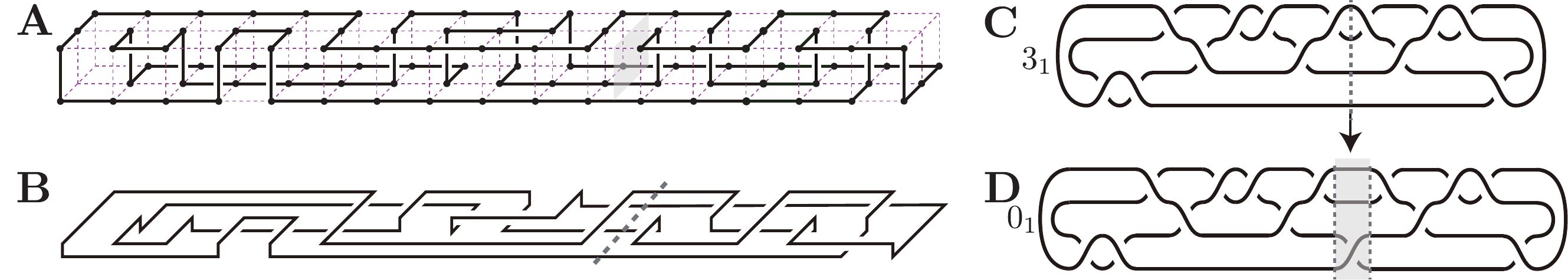}
 \caption{
 {\bf A.} A 98-edge polygon of knot-type $3_1$ (trefoil) in $\tube^*$; 
 {\bf B.}~A shifted knot diagram obtained from a 2-dimensional projection of the polygon in (A); and {\bf C.}~The  4-plat diagram corresponding to the knot in (A) and the knot diagram in (B). 
 {\bf D.}~A 4-plat diagram of the unknot obtained by inserting a half twist at the location indicated with a dotted line in (B) and in (C). Figures (A) to (D) have been modified from Fig. 1 in \cite{beaton_first_2024}. Figures (A) to (D) have been modified from 
 \cite[Fig. 1]{beaton_first_2024}.
}
 \label{fig2}
 \end{center}
 \end{figure}

Formally, a $4$-braid is a 3-dimensional object, a $4$-braid diagram is a 2-D projection of a 4-braid with resolved over- and under-crossings and is described by a $4$-braid word. All three terms are used to describe the same 3-dimensional object and we may use them interchangeably.

\begin{figure}[htb]
\centering
\includegraphics[width=0.48\textwidth]{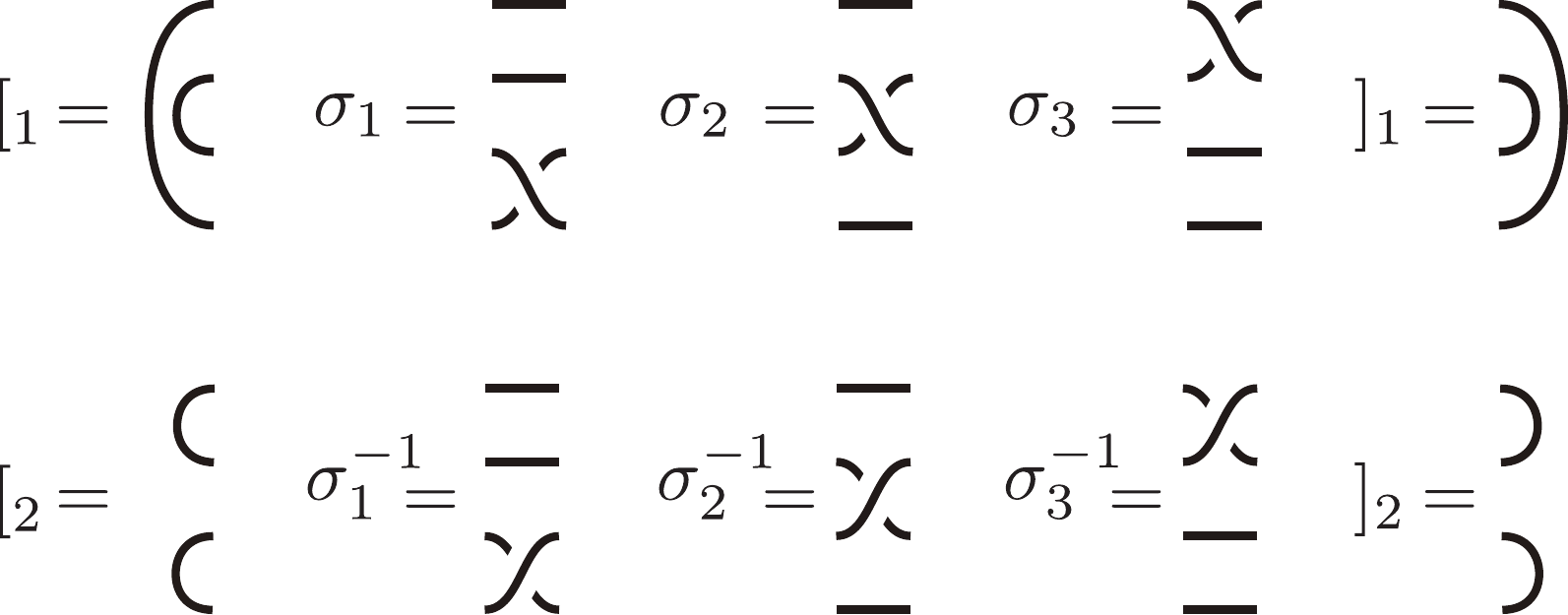}
 \caption{The diagrams in this figure illustrate the elementary $4$-braids $\sigma_1,\sigma_2,\sigma_3, \sigma_1^{-1},\sigma_2^{-1}$ and $\sigma_3^{-1}$ (middle), and the  closures $[_1$, $[_2$, $]_1$,  $]_2$.}
\label{fig:artin_and_4-plat_closure}
\end{figure}

Given a 4-braid word $w$, its  {\em reverse word} $\overline{w}$ is obtained by reversing the order of the elementary braids in $w$. 
The {\it inverse word} $w^{-1}$ of $w$ is obtained by reversing the order of letters and interchanging $\sigma_i$ and $\sigma_i^{-1}$. For example, $(\sigma_1\sigma_2\sigma_3^{-2})^{-1}=\sigma_3^2\sigma_2^{-1}\sigma_1^{-1}$. Note that the words $ww^{-1}$ and $w^{-1}w$ represent braids that are equivalent to the trivial braid. 
By convention, a $4$-braid word that does not contain any $\sigma_3^{\pm1}$ 
is called a {\em $3$-braid word} since one string in the braid diagram has no crossings. 
The {\em flipped word} $\widehat{w}$ of a 3-braid word $w$ is obtained by interchanging the 
1's and 2's 
in the subscripts of the elementary braids in $w$. 

To connect to lattice embeddings, we need two intermediate results.
First, we need a way to go from embeddings in $\tube^*$ to 4-plat diagrams.
We obtain a \emph{shifted diagram}
from the embedding, as described in \cite[Definition~3]{minsteptube} and illustrated in
Figure~\ref{fig2}B.
In the proof of Theorem~\ref{thm-insertion} we establish how to get 4-plat diagrams, one for each factor in the link-type. 

Second, 
we use arguments as in \cite{minsteptube} to construct an
$s$-block 
corresponding to 
any sequence of $\sigma_i$'s and hence any $4$-braid. In general, an $s$-block is a piece of a lattice link between two half-integer planes a distance $s$ apart (Section~\ref{ssec:knotslinkstubes}).
An $s$-block 
corresponding to a 4-braid is called a
\emph{braid $s$-block}, or \emph{braid block}.
Figure~\ref{fig1intro}B 
shows 
a braid $3$-block.

In the next section we present the results needed to obtain the upper bound of Theorem~\ref{thm-bounds}. The main knot theory results are proved in Section~\ref{unknotting}.
%
The connection between the knot theory and lattice embeddings is made in Section~\ref{upperbound}.  

\section{Upper bound: Unknotting lattice links via the unknotting of 4-plat diagrams}

\label{sec:newupper}

Towards establishing the \textit{upper bound} of Theorem~\ref{thm-bounds}, in this section we show how to transition between an embedding of a non-split link $L$ in $\tube^*$
and a corresponding set of $4$-plat diagrams, 
one for each prime factor of $L$. 
The ultimate goal is to ``unknot''  the embedding by a sequence of $f_L$ insertions of braid blocks. 
Specifically, we obtain the following theorem. 
\begin{thm}\label{thm-insertion}
	Any lattice embedding  of  
	a non-split link 
	$L$ in $\tube^*$ can be
	converted to a lattice polygon  of the unknot in $\tube^*$ by  
	$f_L$
	insertions of braid $s$-blocks.
	The span  ($s$) is bounded above by $3c+8$, where $c$ is the maximum crossing number of the prime factors of $L$.
\end{thm}
The \emph{crossing number} of a link $L$ is a topological invariant
given by the minimal number of crossings over all its diagrams.
Figure~\ref{fig1intro}A shows a trefoil polygon in $\tube^*$ with $n=98$ edges, along with an embedding of a braid block (Figure~\ref{fig1intro}B) that, upon insertion at the identified location, converts the trefoil to an unknot polygon. Importantly, the spans of 
the braid blocks are
determined by the crossing numbers of the prime factors of $L$ and do not depend on the size of the lattice embedding of $L$.

To prove Theorem~\ref{thm-insertion} we show that any $4$-plat diagram can be changed into a diagram of the unknot by inserting a specific $4$-braid. The problem of untangling knots via local moves on diagrams is of independent interest in knot theory. 
McCoy \cite{McCoy} proved that unknotting crossings exist in any alternating diagram of an unknotting number one alternating knot.
However, in general, it is not easy to find a small set of crossing changes converting a given diagram into a diagram of the unknot.
Taniyama \cite{Taniyama} showed that for any nontrivial knot and any natural number $N$, there is a diagram of the knot where the unknotting number of the diagram is greater than or equal to $N$.
In the case of 4-plats, in Theorem~\ref{thm:unknottinginsertion} 
we show that any 4-plat diagram of a prime link can be converted to the unknot by inserting a 4-braid whose length is bounded above by the crossing number of the link.

\begin{thm}\label{thm:unknottinginsertion}
For any $4$-plat $L$,  
there exists a $3$-braid word $w_0$ such that 
any $4$-plat diagram of $L$ can be converted into a diagram of the unknot by inserting one of  $w_0,\overline{w_0}, \widehat{w_0}$ and $ \widehat{\overline{w_0}}$.
Moreover, $w_0$ can be taken so that the number of crossings of $w_0$ is at most the crossing number of $L$.
\end{thm}
The definitions of the reverse word $\overline{w_0}$ and the flipped word $\widehat{w_0}$ of $w_0$ are given in Section~\ref{subsec:4-plats}. Converting  a 4-plat diagram $D$ into a diagram $D'$ by inserting $w_0$ means that $D=[_iw_1w_2]_j$ and $D'=[_iw_1w_0w_2]_j$ for some $4$-braid words $w_1,w_2$.
For technical reasons, $w_1$ and $w_2$ are assumed to be non-empty $4$-braid words.

Let us consider some concrete examples. In the case of the knot  
$7_6$, $w_0=\sigma_1^{-2}$ satisfies the condition 
in Theorem~\ref{thm:unknottinginsertion}. Namely, any given $4$-plat diagram of 
$7_6$ can be converted into a diagram of the unknot by inserting either $\sigma_1^{-2}$ ($=w_0=\overline{w_0}$) or $\sigma_2^{-2}$ ($=\widehat{w_0}= \widehat{\overline{w_0}}$). 
Similarly, for $L=6^2_3$ the 3-braid $\sigma_2^{-1}\sigma_1^{-1}$ satisfies the condition of $w_0$,
see Figure~\ref{fig:4-platDiagram}. 
For an arbitrary link type $L$, there can be many options for $w_0$. In fact, not only $\sigma_1^{-2}$ but also $\sigma_1$ and $\sigma_1^{-1}$ satisfy the condition of $w_0$ for $L=3_1$. 

In Section~\ref{unknotting}, we will show that the 3-braid word $w_0$ is found in a minimal-crossing 4-plat diagram of $L$ (Lemma~\ref{lem:minimalunknotting}), and prove Theorem~\ref{thm:unknottinginsertion} from a more general theorem (Theorem~\ref{thm:twist2}). 
Theorem~\ref{thm:twist2} establishes that if a link $L$ is obtained by inserting a 3-braid word $w_0$ into a minimal diagram of a 4-plat, then $L$ can also be obtained by inserting one of $w_0,\overline{w_0}, \widehat{w_0}$ and $ \widehat{\overline{w_0}}$ into any other diagram of the 4-plat. Theorem~\ref{thm:unknottinginsertion} will then follow by taking $L$ to be the unknot. A key element of the proof is the definition of seven types of 4-plat diagram moves, and the proof that they preserve the link type (Lemma~\ref{lem:A1-4}). Further lemmas establish how diagrams related by these moves are affected by insertions of one of $w_0,\overline{w_0}, \widehat{w_0}$ and $ \widehat{\overline{w_0}}$.  
These novel results about 4-plats are presented in full detail in Section~\ref{unknotting}.

 \begin{figure}[htb]
\begin{center}
 \includegraphics[width=.58\textwidth]{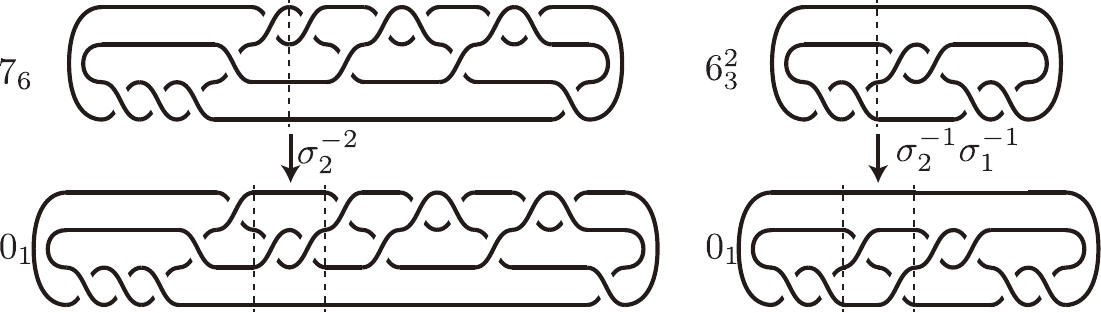}
 \caption{Insertions of $w_0=\sigma_2^{-2}$ and $w_0=\sigma_2^{-1}\sigma_1^{-1}$ on $4$-plat diagrams. These insertions are used to obtain diagrams of the unknot.}
\label{fig:4-platDiagram}
  \end{center}
 \end{figure}

To make the connection to lattice links we establish the following result. %
\begin{prop}
\label{prop:braidblock}
Let $w_0$ be a $3$-braid word with $c$ crossings. 
There is a span-$3c$ braid block in $\tube^*$ representing $w_0$.  Moreover, for each type of $4$-section there is a braid block in $\tube^*$ with span at most $3c+8$ such that it represents $w_0$ and can be inserted into the $4$-section.
\end{prop}
The proof of Proposition~\ref{prop:braidblock} is constructive and uses similar arguments to those of~\cite{minsteptube}; full details are given in Section~\ref{upperbound}. 

This proposition allows us to connect Theorem~\ref{thm:unknottinginsertion} to Theorem~\ref{thm-insertion} provided that we are able to associate 4-plat diagrams to lattice embeddings.
For the latter, we determine sufficient conditions for the shifted diagram of an embedding of link $L$ in $\tube^*$ to be a 4-plat diagram (Lemma~\ref{f_L}). We then show in Lemma~\ref{lem:concat_no_2strings} that we can divide an embedding of $L$ into a sequence of link embeddings, one for each prime factor of $L$, and each with an associated 4-plat diagram.  The resulting diagrams allow for the direct identification of the insertion points for the braid blocks of Proposition~\ref{prop:braidblock} in the original embedding of $L$. 
Full details are given in Section~\ref{upperbound}. 

\subsection{Unknotting of 4-plat diagrams: Insertions of 4-braid words into 4-plat diagrams}

\label{unknotting}

In this section, we consider knots and links which are obtained by inserting a $4$-braid word $w_0$ 
 into a 4-plat diagram, and will prove Theorem~\ref{thm:unknottinginsertion}. %

Before discussing the proof, we define the 
$\mathcal{A}$ and $\mathcal{B}$ moves on 4-plat diagrams in Figure~\ref{fig:ABmoves}. These moves are essential to the proof. 

 \begin{figure}[htb]
\begin{center}
 \includegraphics[width=.8\textwidth]{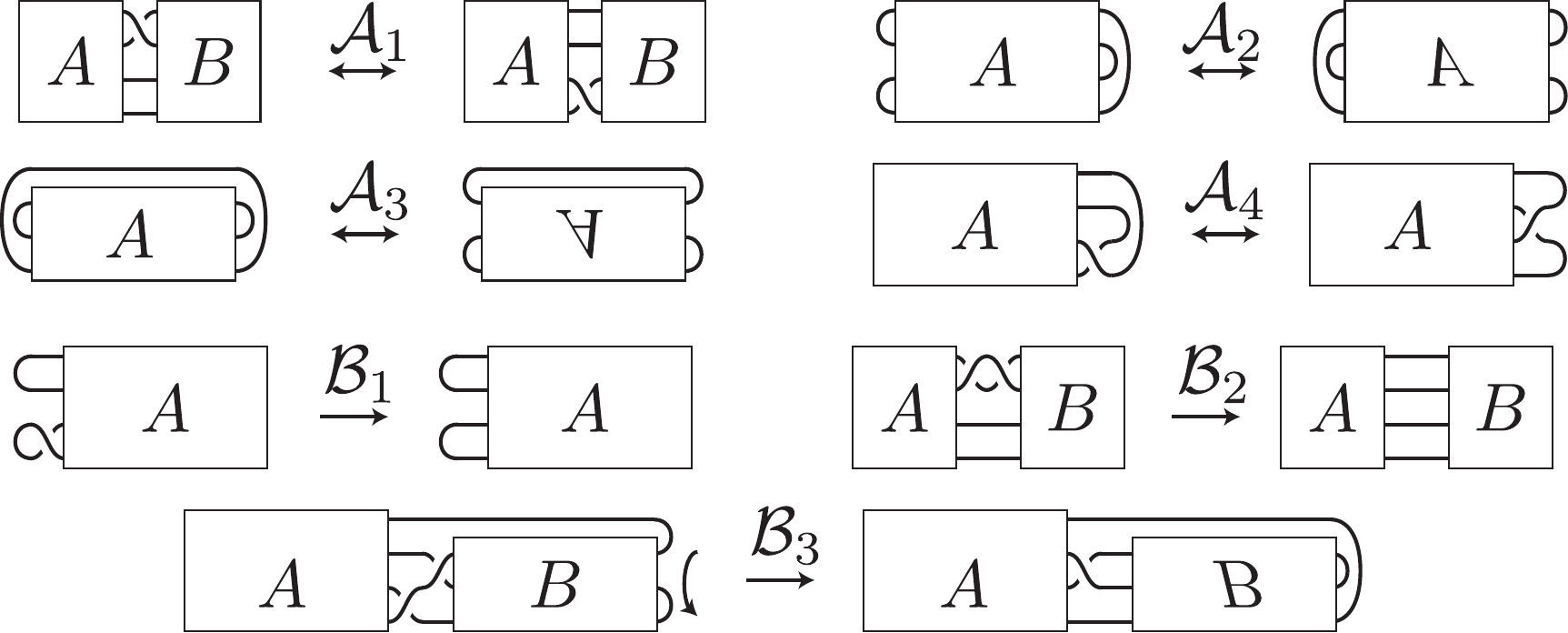}
 \caption{Illustrations of the $\mathcal{A}$ and $\mathcal{B}$ moves which take one 4-plat diagram to another.
 The $A$ in the $\mathcal{A}_3$ move and the $B$ in the $\mathcal{B}_3$ move are 3-braid diagrams. The $A$ in the $\mathcal{A}_2$ move is a 4-braid diagram.
 Every other rectangular block represents a 4-braid diagram along with an arbitrary one-sided closure. 
An $\mathcal A_1$-move relates two diagrams of the form $[_iw_1\sigma_{1}^{\eps}w_2]_j$ and $[_iw_1\sigma_{3}^{\eps}w_2]_j$ for $\eps=\pm1$. 
An $\mathcal A_2$-move relates two diagrams of the form
$[_iw]_j$ and $[_j\overline{w}]_i$, where $\overline{w}$ is the reverse word of $w$ as defined above. 
An $\mathcal A_3$-move relates two diagrams of the form $[_1w]_1$ and $[_2\widehat{w}]_2$, or
$[_1w]_2$ and $[_2\widehat{w}]_1$ for a $3$-braid word $w$, where 
$\widehat{w}$ is the $3$-braid word which is the flipped word defined above.
An $\mathcal A_4$-move relates two diagrams of the form $[_iw\sigma_1^{\eps}]_1$ and $[_iw\sigma_2^{-\eps}]_2$ for $\eps=\pm1$. 
A $\mathcal B_1$ move is a Reidemeister I move on one end of a $4$-plat diagram, that deforms $[_i\sigma_k^{\eps}w]_j$ into $[_iw]_j$ or deforms $[_jw\sigma_k^{\eps}]_i$ into $[_jw]_i$ 
for $\eps=\pm1$ and $(k,i)=(1,2),(2,1)$ or $(3,2)$. 
A $\mathcal B_2$ move is a Reidemeister II move on a reducible $4$-braid diagram, that deforms $[_iw_1\sigma_k^{\eps}\sigma_k^{-\eps}w_2]_j$ into $[_iw_1w_2]_j$ for $\eps=\pm1$ and $i,j,k\in\{1,2,3\}$.
A $\mathcal B_3$ move deforms $[_iw_1\sigma_1^{\eps}\sigma_2^{\eps} w_2]_j$ into $[_iw_1\sigma_2^{-\eps}\widehat{w_2}]_{j'}$, or deforms $[_iw_1\sigma_2^{\eps}\sigma_1^{\eps} w_2]_{j}$ into $[_iw_1\sigma_1^{-\eps}\widehat{w_2}]_{j'}$ for $\eps=\pm1$ and $\{j,j'\}=\{1,2\}$ when $w_2$ is a $3$-braid word,
where $\widehat{w_2}$ is the flipped word of $w_2$.
 }
\label{fig:ABmoves}
  \end{center}
 \end{figure}
 
First, we observe that the $\mathcal{A}$ and $\mathcal{B}$ moves of Figure~\ref{fig:ABmoves}  preserve the link type of any $4$-plat diagram.
\begin{lem}\label{lem:A1-4}
	Suppose that two $4$-plat diagrams $D_1$ and $D_2$ are related by one of the moves $\mathcal A_1$, $\mathcal A_2$, $\mathcal A_3$, $\mathcal A_4$, $\mathcal B_1$, $\mathcal B_2$ and $\mathcal B_3$.
	Then $D_1$ and $D_2$ represent the same link type.
\end{lem}

\begin{proof}
	First, we consider two $4$-plat diagrams $D_1=[_iw_1\sigma_{1}^{\eps}w_2]_j$ and $D_2=[_iw_1\sigma_{3}^{\eps}w_2]_j$ 
	related by 
	an $\mathcal A_1$ move, where $\eps=\pm1$.
	By closing a $4$-braid word $w$ at one end of the braid we obtain a rational tangle, say $w]_j$. 
	Since the rational tangle $w_2]_j$ in $D_1$ (and also in $D_2$) has symmetry, 
	$D_1$ and $D_2$ are transformed 
	into
	each other by turning over the rational tangle (``flype'') , 
	thus they represent the same link type as illustrated in Figure~\ref{fig:S1a}.
	Next, we consider $D_1=[_iw]_j$ and $D_2=[_j\overline{w}]_i$ or $[_i\widehat{w}]_j$ for an $\mathcal A_2$ or $\mathcal A_3$ move.
	We observe that $D_2$ is obtained by rotating all of $D_1$ around a vertical axis or by rotating a part of $D_1$ around a horizontal axis, 
	thus $D_1$ and $D_2$ represent the same link type. 
	Two $4$-plat diagrams related by an $\mathcal A_4$ move are the same link diagram, thus the $\mathcal A_4$ move does not change the link type. 
	$\mathcal B_1$ and $\mathcal B_2$ moves are the Reidemeister moves of types I and II, respectively, thus they do not change the link type. 
	Finally, we observe that a $\mathcal B_3$ move can be considered as a $\pi$-rotation of a $3$-string part of a $4$-plat diagram, and therefore it does not change the link type. 
\end{proof}

		\begin{figure}[htb]
			\centering
			\includegraphics[scale=.5]
			{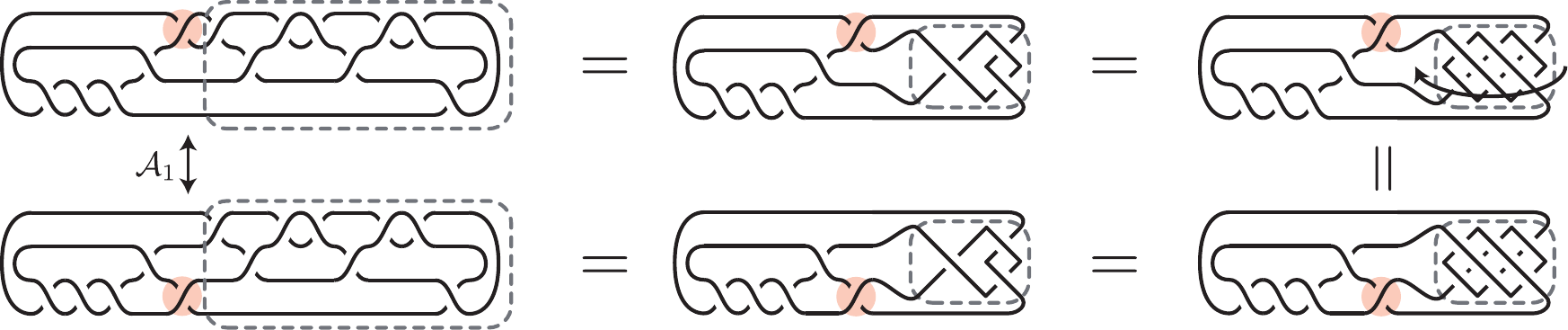}
			\caption{An $\mathcal A_1$ move preserves 
 crossing number and the link type. We illustrate this here with an example where $=$ represents topological equivalence. The portions of each link diagram surrounded by a dotted frame are 2-string rational tangles. Any 2-string rational tangle can be smoothly deformed to obtain one with a symmetric diagram as in the top right image. The top right link diagram can thus be transformed to the diagram on the bottom right by a 180-degree rotation around the vertical axis, going through the middle of the tangle (not shown). This move is indicated by an arrow. }\label{fig:S1a}
		\end{figure}

For a given $4$-braid word $w_0$, we now consider links obtained from a $4$-plat diagram by inserting $w_0$.
Recall that
we say that a $4$-plat diagram {\em $D'$ is obtained from $D$ by inserting $w_0$},  denoted by 
$D\stackrel{w_0}{\longrightarrow}D'$, provided that 
the two $4$-plat diagrams $D$ and $D'$ are represented by non-empty $4$-braid words $w_1,w_2$ as follows:
\begin{align}
	D &= [_iw_1 w_2]_j,\\
	D' &= [_iw_1w_0w_2]_j.
\end{align}
Note that in the insertion above we always assume that $w_1$ and $w_2$ are non-empty words, i.e.\ we do not consider insertions at ends of $4$-plat diagrams.
 We can obtain  a 3-braid word $w_0'$ from $w_0$ by replacing any $\sigma_3^{\pm1}$ by $\sigma_1^{\pm1}$. An insertion of $w_0'$ into a given 4-plat diagram will yield the same link type as the insertion of $w_0$ into the diagram; this is because the two link diagrams are related by $\mathcal A_1$ moves; see Figure~\ref{fig:S1b}.  
Thus we may assume that $w_0$ is a $3$-braid word.

		\begin{figure}[htb]
			\centering
			\includegraphics[scale=.5]{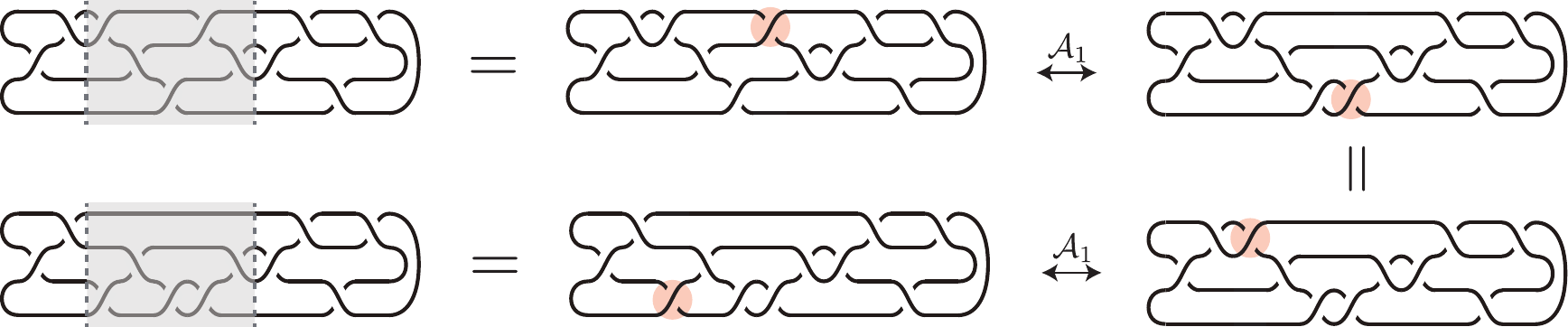}
			\caption{The two $4$-plat diagrams on the left are obtained from the same $4$-plat diagram by insertions of  
  (top left) a 
  $4$-braid word $\sigma_3^{-1}\sigma_2\sigma_1^{-1}\sigma_3^{-1}\sigma_2$ and
  (bottom left) the corresponding $3$-braid word $\sigma_1^{-1}\sigma_2\sigma_1^{-2}\sigma_2$. The 4-plat diagrams on the left are shown to be related to the ones on the right by two $\mathcal A_1$ moves. Thus they represent the same link type (the unknot in this case). }\label{fig:S1b}
		\end{figure}
Suppose that some link $K$ is obtained from a $4$-plat diagram $D_1$ by inserting $w_0$, and that $D_1,D_2$ are related by an $\mathcal A_2$ or $\mathcal A_3$ move. 
Then $K$ is also obtained from $D_2$ by inserting $\overline{w_0}$ or $\widehat{w_0}$. 
Therefore, we consider the insertions of $w_0$, $\overline{w_0}$, $\widehat{w_0}$ and $\widehat{\overline{w_0}}$  all together. 
Let $S(w_0)$ be the set $\{w_0,\overline{w_0},\widehat{w_0},\widehat{\overline{w_0}}\}$ for a $3$-braid word $w_0$. 
Note that $S(w_0)=S(\overline{w_0})=S(\widehat{w_0})=S(\widehat{\overline{w_0}})$ 
since $\overline{\overline{w_0}}=\widehat{\widehat{w_0}}=w_0$ and $\overline{\widehat{w_0}}=\widehat{\overline{w_0}}$. 
Given a 4-plat diagram $D$ and a $3$-braid word $w_0$, 
let $\mathcal{K}_{w_0}(D)$ be the set of all link types of $4$-plat diagrams that are obtained from $D$ by inserting $w_0$, $\overline{w_0}$, $\widehat{w_0}$, or $\widehat{\overline{w_0}}$, that is: 
\begin{equation}
	\mathcal{K}_{w_0}(D):=\{\mbox{the link type of }D'\mid w_0'\in S(w_0),\ D\stackrel{w_0'}{\longrightarrow}D'\}.
\end{equation}

The links in $\mathcal{K}_{w_0}(D)$ are clearly $4$-plats 
since the resulting diagrams $D'$ are $4$-plat diagrams.
\begin{lem}\label{lem:minimalunknotting}
Let $D_0$ be a minimal-crossing $4$-plat diagram of a non-trivial $4$-plat $L$. 
There exists a $3$-braid word $w_0$ such that $D_0$ can be converted into a diagram of the unknot by inserting $w_0$. In other words, $\mathcal{K}_{w_0}(D)$ contains the unknot for some $3$-braid word $w_0$. Moreover, $w_0$ can be taken so that the number of crossings of $w_0$ is at most the crossing number of $L$. 
\end{lem}
\begin{proof}
Since $\mathcal{A}_1$ moves preserve 
crossing number and
link type, we may assume $D_0$ to be a closure of a $3$-braid word $w$; see also Proposition~\ref{prop:conway}. %
Since $L$ is a non-trivial $4$-plat, the $3$-braid word $w$ can be divided into two non-empty $3$-braid words $w_1$ and $w_2$ as $w=w_1w_2$. Using the property of the inverses of braids, all crossings are cancelled after inserting $w_1^{-1}w_2^{-1}$ into $D_0$, and therefore the resulting $4$-plat diagram represents an unlink.
The $4$-plat may be the unknot or the 2-component unlink, depending on how $w$ is closed for $D_0$.
In the former case, 
let $w_0=w_1^{-1}w_2^{-1}$.
In the latter case, 
we can select a $3$-braid $w_0$ that leaves one crossing at an end of $D_0$ to obtain the unknot. 
Thus, there exists a $4$-braid $w_0$ such that the crossing number of $w_0$ is less than or equal to the crossing number of $L$, and
$D_0$ can be converted into a diagram of the unknot by inserting $w_0$. 
\end{proof}

From Lemma \ref{lem:minimalunknotting}, the remaining proof of Theorem~\ref{thm:unknottinginsertion} is obtained by taking $K$ to be the unknot 
in the following theorem.

\begin{thm}%
\label{thm:twist2}
	Suppose that $D$ and $D_0$ are $4$-plat diagrams that represent the same link type, and that $D_0$ is a minimal-crossing diagram. If a link $K$ is obtained by inserting a $3$-braid word $w_0$ into $D_0$, then $K$ can be obtained from $D$ by inserting one of $w_0$, $\overline{w_0}$, $\widehat{w_0}$ and $ \widehat{\overline{w_0}}$. %
	Namely, the following inclusion relation holds for any $3$-braid word $w_0$.
	\begin{equation}
		\mathcal{K}_{w_0}(D)\supseteq \mathcal{K}_{w_0}(D_0).
	\end{equation}
\end{thm}

To prove Theorem \ref{thm:twist2} we deform $D$ to $D_0$ step by step without changing the link type and we show that an inclusion relation such as $\mathcal{K}_{w_0}(D)\supseteq \mathcal{K}_{w_0}(D_0)$ holds for each step. 
First, we observe that 
$\mathcal A$ moves do not change the set $\mathcal{K}_{w_0}(D)$. 

\begin{lem}\label{lem:replace1} 
	Suppose that two $4$-plat diagrams $D_1$ and $D_2$ are related by one of the moves
	$\mathcal A_1$, $\mathcal A_2$, $\mathcal A_3$ or $\mathcal A_4$. 
	Then, for any $3$-braid word $w_0$ we have
	\begin{equation}
		\mathcal{K}_{w_0}(D_1)=\mathcal{K}_{w_0}(D_2).
	\end{equation}
\end{lem}

\begin{proof}
	Except for an $\mathcal A_4$ move, we already observed the equation $\mathcal{K}_{w_0}(D_1)=\mathcal{K}_{w_0}(D_2)$ above.
	Suppose $D_1$ and $D_2$ are related by an $\mathcal A_4$ move so that $D_1=[_iw\sigma_1^{\eps}]_1$ and $D_2=[_iw\sigma_2^{-\eps}]_2$, 
	and $D_1'$ is obtained from $D_1$ by inserting $w_0$.  
	Since we do not consider the insertion at the ends of $D_1$, 
	$D_1$ and $D_1'$ are seen as $D_1=[_i w_1w_2\sigma_1^{\eps}]_1$ and  $D_1'=[_i w_1w_0w_2\sigma_1^{\eps}]_1$ for some $4$-braid words $w_1$ and $w_2$ ($w_1$ is a non-empty word but $w_2$ can be an empty word). 
	Then $D_2'=[_iw_1w_0w_2\sigma_2^{-\eps}]_2$ (obtained from $D_2$) represents the same link type as $D_1'$ since they are related by an $\mathcal A_4$ move. 
	This implies the inclusion relation $\mathcal{K}_{w_{0}}(D_1)\subseteq\mathcal{K}_{w_{0}}(D_2)$. 
	The reverse inclusion relation is shown similarly, so the equation $\mathcal{K}_{w_0}(D_1)= \mathcal{K}_{w_0}(D_2)$ holds.
\end{proof}

By replacing all $\sigma_3^{\pm1}$'s with $\sigma_1^{\pm1}$'s on a $4$-plat diagram by $\mathcal A_1$ moves, 
we obtain a $4$-plat diagram which is a closure of a $3$-braid word. 
In particular, $4$-plat diagrams that are closures of reduced $3$-braid words, 
\begin{equation}
	[_1\sigma_1^{a_1}\sigma_2^{a_2}\sigma_1^{a_3}\cdots\sigma_j^{a_{n}}]_j \qquad\text{and}\qquad 
	[_2\sigma_2^{a_1}\sigma_1^{a_2}\sigma_2^{a_3}\cdots\sigma_{j'}^{a_n}]_{j'},
\end{equation}
are called {\em Conway's normal forms} 
if $a_i\neq0$ for all $i$ and $j=1\ (j'=2)$ or $j=2\ (j'=1)$ according to whether $n$ is odd or even. 
In fact, the above two represent the same link type since they are related by an $\mathcal{A}_3$ move, and are denoted by $C(a_1,-a_2,a_3,\ldots,(-1)^na_n)$.
The following is a special case of Theorem \ref{thm:twist2} where $D$ is also a minimal-crossing $4$-plat diagram.  
\begin{prop}\label{prop:conway}
	Suppose $D$ and $D_0$ are minimal-crossing 
	$4$-plat diagrams of the same link type.
	Then  for any $3$-braid word $w_0$ we have
	\begin{equation}
		\mathcal{K}_{w_0}(D)=\mathcal{K}_{w_0}(D_0).
	\end{equation}
\end{prop}

\begin{proof}
	Since $D$ and $D_0$ are minimal-crossing diagrams, they can be transformed into minimal-crossing Conway's normal forms $D'$ and $D'_0$, respectively, by replacing all $\sigma_3^{\pm1}$'s with $\sigma_1^{\pm1}$'s by $\mathcal A_1$ moves. 
	It is known that two minimal-crossing Conway's normal forms represent the same link if and only if they are related by the combination of $\mathcal A_2,\mathcal A_3$ and $\mathcal A_4$ moves, see \cite{Murasugi}.  
	Then $D$ can be transformed to $D_0$ by a finite sequence of $\mathcal A_1,\mathcal A_2,\mathcal A_3$ and $\mathcal A_4$ moves, via $D'$ and $D'_0$ on the way. 
	Moreover, the equation $\mathcal{K}_{w_0}(D)=\mathcal{K}_{w_0}(D_0)$ holds by Lemma \ref{lem:replace1}.
\end{proof}

Next, we observe an inclusion relation $\mathcal{K}_{w_0}(D_1)\supseteq \mathcal{K}_{w_0}(D_2)$ on $\mathcal B$ moves from $D_1$ to $D_2$.
\begin{lem}\label{lem:reducing}%
	Suppose a $4$-plat diagram $D_1$ is deformed into $D_2$ by one of $\mathcal B_1$, $\mathcal B_2$ and $\mathcal B_3$. %
	Then %
	the following inclusion relation holds for any $3$-braid word $w_0$.
	\begin{equation}
		\mathcal{K}_{w_0}(D_1)\supseteq\mathcal{K}_{w_0}(D_2).
	\end{equation}
\end{lem}

\begin{proof}
	First, we consider a case where $D_1=[_i\sigma_k^{\eps}w]_j$ and $D_2=[_iw]_j$ for a $\mathcal B_1$ move, where $\eps=\pm1$ and $(k,i)=(1,2),(2,1)$ or $(3,2)$.
	For any link type  $K\in \mathcal{K}_{w_0}(D_2)$, there exists a $4$-plat diagram $[_iw']_j$ representing $K$ such that $[_iw]_j\stackrel{w_0'}{\rightarrow}[_iw']_j$, for some $w_0'\in S(w_0)$.
	Then a $4$-plat diagram $[_i\sigma_k^{\eps}w']_j$ that is deformed into $[_iw']_j$ by a $\mathcal B_1$ move, also represents $K$ and  $[_i\sigma_k^{\eps}w]_j\stackrel{w_0'}{\rightarrow}[_i\sigma_k^{\eps}w']_j$. 
	This implies $\mathcal{K}_{w_0}(D_1)\supseteq \mathcal{K}_{w_0}(D_2)$ for the $\mathcal B_1$ move from $D_1$ to $D_2$.
	
	Next, we consider the case where $D_1=[_iw_1\sigma_k^{\eps}\sigma_k^{-\eps}w_2]_j$ and $D_2=[_iw_1w_2]_j$ for a $\mathcal B_2$ move, where $\eps=\pm1$, $i,j\in\{1,2\}$ and $k\in\{1,2,3\}$.
	For any link type $K\in \mathcal{K}_{w_0}(D_2)$, there exists a 4-plat diagram $[_iw'_1w'_2]_j$ (either $w'_1=w_1$ or $w'_2=w_2$)
	representing $K$ such that $[_iw_1w_2]_j\stackrel{w_0'}{\longrightarrow} [_iw'_1w'_2]_j$ for some $w_0'\in S(w_0)$. 
	Then a $4$-plat diagram $[_iw'_1\sigma_k^{\eps}\sigma_k^{-\eps}w'_2]_j$ that is deformed into $[_iw'_1w'_2]_j$ by a $\mathcal B_2$ move,
	also represents $K$ and $[_iw_1\sigma_k^{\eps}\sigma_k^{-\eps}w_2]_j\stackrel{w_0'}{\rightarrow}[_iw'_1\sigma_k^{\eps}\sigma_k^{-\eps}w'_2]_j$. 
	This implies that $\mathcal{K}_{w_0}(D_1)\supseteq \mathcal{K}_{w_0}(D_2)$ for the $\mathcal B_2$ move from $D_1$ to $D_2$.
	
	Finally, we consider only the case where $D_1=[_iw_1\sigma_1^{\eps}\sigma_2^{\eps} w_2]_j$ and $D_2=[_iw_1\sigma_2^{-\eps}\widehat{w_2}]_{j'}$ for a $\mathcal B_3$ move, where $w_2$ is a $3$-braid word, $\eps=\pm1$ and $\{j,j'\}=\{1,2\}$.
	For any link type $K\in \mathcal{K}_{w_0}(D_2)$, there exists a $4$-plat diagram $[_iw'_1\sigma_2^{-\eps} w'_2]_{j'}$  (either $w'_1=w_1$ or $w'_2=\widehat{w_2}$), 
	representing $K$ and $[_iw_1\sigma_2^{-\eps}\widehat{w_2}]_{j'}\stackrel{w_0'}{\longrightarrow}[_iw'_1\sigma_2^{-\eps} w'_2]_{j'}$ for some $w_0'\in S(w_0)$. 
	Then a $4$-plat diagram  $[_iw'_1\sigma_1^{\eps}\sigma_2^{\eps} \widehat{w'_2}]_{j}$, that is deformed into $[_iw'_1\sigma_2^{-\eps} w'_2]_{j'}$ by a $\mathcal B_3$ move, 
	also represents $K$ and $[_iw_1\sigma_1^{\eps}\sigma_2^{\eps}\widehat{w_2}]_{j}\stackrel{w_0'}{\longrightarrow}[_iw'_1\sigma_1^{\eps}\sigma_2^{\eps} \widehat{w'_2}]_{j}$. This implies $\mathcal{K}_{w_0}(D_1)\supseteq\mathcal{K}_{w_0}(D_2)$ for the $\mathcal B_3$ move from $D_1$ to $D_2$. 
\end{proof}

Finally, we prove Theorem \ref{thm:twist2}.  %
It is known that any $4$-plat has an alternating ($4$-plat) diagram. It is also known that a reduced alternating diagram of a $4$-plat (prime alternating link) is a minimal-crossing diagram, while a non-alternating diagram cannot be minimal-crossing~\cite{Kauffman, Murasugi2, Thistlethwaite}. 

\begin{proof}[Proof of Theorem \ref{thm:twist2}] 
	First, we change the $4$-plat diagram $D$ to $D_1$ so that $D_1$ is a closure of a $3$-braid word by replacing all $\sigma_3^{\pm1}$'s with $\sigma_1^{\pm1}$'s using $\mathcal A_1$ moves. 
	By Lemma \ref{lem:replace1}, we have $\mathcal{K}_{w_0}(D)=
	\mathcal{K}_{w_0}(D_1)$. 
	Next, we obtain a $4$-plat diagram
	$D_2$ by applying $\mathcal B_1,\mathcal B_2$ and $\mathcal B_3$ moves to reduce the crossing number of $D_1$ as much as possible. By Lemma 
	\ref{lem:reducing},  we have $\mathcal{K}_{w_0}(D_1)\supseteq\mathcal{K}_{w_0}(D_2)$. 
	$D_2$ is also a closure of a $3$-braid word 
	and we cannot apply $\mathcal B_1$, $\mathcal B_2$ and $\mathcal B_3$ moves. 
	This implies that $D_2$ is a reduced alternating Conway's normal form that is a minimal-crossing $4$-plat diagram.
	Then, by Proposition  \ref{prop:conway}, we have $\mathcal{K}_{w_0}(D_2)=\mathcal{K}_{w_0}(D_0)$. 
	See Figure~\ref{fig:S1c}
	for an example.
	Therefore, 
	\begin{equation}
		\mathcal{K}_{w_0}(D)=%
		\mathcal{K}_{w_0}(D_1)\supseteq\mathcal{K}_{w_0}(D_2)=\mathcal{K}_{w_0}(D_0).
	\end{equation}
\end{proof}

		\begin{figure}[htb]
        \centering
			\includegraphics[scale=.5]
			{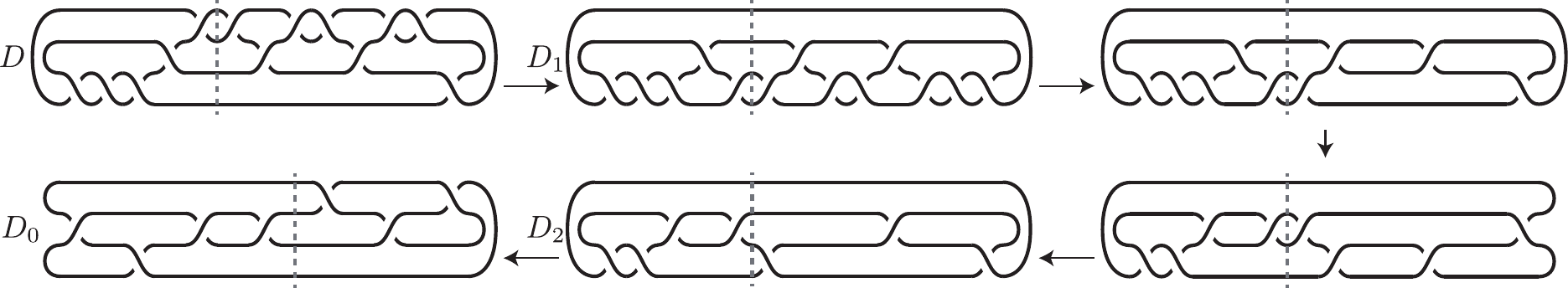}
			\caption{Theorem~\ref{thm:twist2} example: A diagram (3-braid word closure) $D_1$ is obtained from $D$ by $\mathcal A_1$ moves. $D_2$ is obtained from $D_1$ by applying $\mathcal B_1$, $\mathcal B_2$ and $\mathcal B_3$ moves as much as possible, to obtain a minimal-crossing Conway's form. Then $D_2$ and minimal crossing $D_0$ are related by $\mathcal A_1$, $\mathcal A_2$, $\mathcal A_3$ and $\mathcal A_4$ moves.}\label{fig:S1c}
		\end{figure}
\subsection{Unknotting lattice links}
\label{upperbound}

In this section we will prove 
that any embedding in $\tube^*$ of a knot or a non-split link $L$ can be converted to a polygon of the unknot by the insertion of at most $f_L$ blocks (one for each prime factor of $L$), which correspond to $4$-braid insertions.
For the proof, we will establish the following:
(i) $4$-braid insertions as defined in the previous section can be realized by the insertion of blocks at midplanes of sections of tube embeddings of $L$;
(ii) the theory from the previous section can be used to identify the required block insertion locations (midplanes of sections). 
The result is Theorem~\ref{thm-insertion}.

First, we focus on (ii). Given  an embedding $P$ of link-type $L=L_1\#L_2\#\cdots\#L_{f_L}$, first we find a set of associated embeddings $P_{i}$ for each factor $L_i$.  The $P_i$'s will have the properties that: a) their shifted diagrams (as described in \cite[Definition~3]{minsteptube} and illustrated in Figure~\ref{fig2}{}{B}
) are 4-plat diagrams; and b) any block insertion that unknots $P_i$ can be realized by a block insertion at a corresponding section of $P$, so that $L_i$ is removed from the factor decomposition of $P$.  Thus the resulting embedding $P'$ has link-type $L=L_1\#\cdots\#L_{i-1}\#L_{i+1}\#\cdots\#L_{f_L}$.

	\begin{figure}[htb]
		\centering
        \includegraphics[width=.7\textwidth]{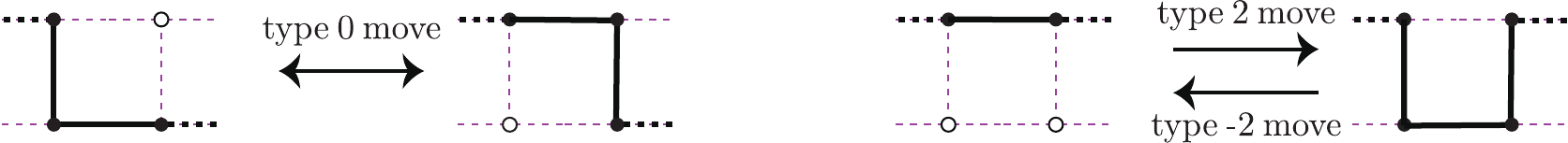}
			
			
			
			
				
			
		\caption{Illustrations of Type $0$, $2$ and $-2$ moves from the BFACF algorithm. A vertex in the figure, marked by an open circle, indicates a lattice point not occupied by the lattice polygon.}
		\label{BFACF}
	\end{figure}
For this we need the following definitions.  We use the standard definition of Type $0$, $2$ and $-2$ moves from the BFACF algorithm as illustrated in Figure~\ref{BFACF}.
It has been established that these moves preserve knot and link-type \cite{vanrensburg_whittington_1990}.
Then, a 4-section is called a {\em  hidden 2-section}  if it 
can be turned into a 2-section after applying 
one type $-2$-move.
See Figure~\ref{Fig:hidden2section}.
Note that if an embedding $P$ has a hidden 2-section, we can divide $P$ into two smaller embeddings in a similar way as in the case of a 2-section.
See Figure~\ref{Fig:hidden2sectiondecomp}.
Note further that a $-2$-move involves three edges  which form a U-shape; we say that the {\em direction} of the $-2$-move is the lattice direction in which the edge at the bottom of the U moves when a $-2$-move is performed.   Here the {\em bottom} of the U is the edge which is perpendicular to the other two edges forming the U.
	\begin{figure}[htb]
		\centering
		\begin{tikzpicture}[scale=0.7]
			\begin{knot}[consider self intersections=true, clip width=3, clip radius=1.5pt, end tolerance=1pt]
				\strand[thick] (0,0,0) node [vblack] {} -- (1,0,0) node [vblack] {} -- (1,0,1) node [vblack] {} -- (2,0,1) node [vblack] {} -- (3,0,1) node [vblack] {} -- (4,0,1) node [vblack] {} -- (5,0,1) node [vblack] {} -- (6,0,1) node [vblack] {} -- (6,0,0) node [vblack] {} -- (7,0,0) node [vblack] {};
				\strand[thick] (0,1,0) node [vblack] {} -- (1,1,0) node [vblack] {} -- (2,1,0) node [vblack] {} -- (2,0,0) node [vblack] {} -- (3,0,0) node [vblack] {} -- (4,0,0) node [vblack] {} -- (5,0,0) node [vblack] {} -- (5,1,0) node [vblack] {} -- (6,1,0) node [vblack] {} -- (7,1,0) node [vblack] {};
				\strand[thick] (7,2,1) node [vblack] {} -- (6,2,1) node [vblack] {} -- (6,1,1) node [vblack] {} -- (5,1,1) node [vblack] {} -- (5,2,1) node [vblack] {} -- (4,2,1) node [vblack] {} --  (4,1,1) node [vblack] {} -- (3,1,1) node [vblack] {} -- (3,1,0) node [vblack] {} -- (4,1,0) node [vblack] {} -- (4,2,0) node [vblack] {} -- (5,2,0) node [vblack] {} -- (6,2,0) node [vblack] {} -- (7,2,0) node [vblack] {};
				\strand[thick] (0,1,1) node [vblack] {} -- (1,1,1) node [vblack] {} -- (2,1,1) node [vblack] {} -- (2,2,1) node [vblack] {} -- (3,2,1) node [vblack] {} -- (3,2,0) node [vblack] {} -- (2,2,0) node [vblack] {} -- (1,2,0) node [vblack] {} -- (1,2,1) node [vblack] {} -- (0,2,1) node [vblack] {};
				\draw (0,0,1) node [vgrey] {};
				\draw (0,2,0) node [vgrey] {};
				\draw (7,0,1) node [vgrey] {};
				\draw (7,1,1) node [vgrey] {};
			\end{knot}
			
			\tikzset{xshift=5.8cm}
			\draw[line width=1.5pt, ->] (2.5,0,0.75) -- (3.5,0,0.75);
			
			\tikzset{xshift=4cm}
			\begin{knot}[consider self intersections=true, clip width=3, clip radius=3pt, end tolerance=1pt]
				\strand[thick] (0,0,0) node [vblack] {} -- (1,0,0) node [vblack] {} -- (1,0,1) node [vblack] {} -- (2,0,1) node [vblack] {} -- (3,0,1) node [vblack] {} -- (4,0,1) node [vblack] {} -- (5,0,1) node [vblack] {} -- (6,0,1) node [vblack] {} -- (6,0,0) node [vblack] {} -- (7,0,0) node [vblack] {};
				\strand[thick] (0,1,0) node [vblack] {} -- (1,1,0) node [vblack] {} -- (2,1,0) node [vblack] {} -- (2,0,0) node [vblack] {} -- (3,0,0) node [vblack] {} -- (4,0,0) node [vblack] {} -- (5,0,0) node [vblack] {} -- (5,1,0) node [vblack] {} -- (6,1,0) node [vblack] {} -- (7,1,0) node [vblack] {};
				\strand[thick, black] (7,2,1) node [vblack] {} -- (6,2,1) node [vblack] {} -- (6,1,1) node [vblack] {} -- (5,1,1) node [vblack] {} -- (5,2,1) node [vblack] {} -- (4,2,1) node [vblack] {} --  (4,1,1) node [vblack] {} -- (4,1,0) node [vblack] {} -- (4,2,0) node [vblack] {} -- (5,2,0) node [vblack] {} -- (6,2,0) node [vblack] {} -- (7,2,0) node [vblack] {};
				\strand[thick] (0,1,1) node [vblack] {} -- (1,1,1) node [vblack] {} -- (2,1,1) node [vblack] {} -- (2,2,1) node [vblack] {} -- (3,2,1) node [vblack] {} -- (3,2,0) node [vblack] {} -- (2,2,0) node [vblack] {} -- (1,2,0) node [vblack] {} -- (1,2,1) node [vblack] {} -- (0,2,1) node [vblack] {};
				\draw (0,0,1) node [vgrey] {};
				\draw (0,2,0) node [vgrey] {};
				\draw (7,0,1) node [vgrey] {};
				\draw (7,1,1) node [vgrey] {};
			\end{knot}
		\end{tikzpicture}
		\caption{An example of a hidden 2-section. A 2-section appears after applying one $-2$-move.}
		\label{Fig:hidden2section}
	\end{figure}
    \begin{figure}[htb]
\centering
		\begin{tikzpicture}[scale=0.6]
			
			\begin{knot}[consider self intersections=true, clip width=3, clip radius=1.5pt, end tolerance=1pt]
				\strand[thick] (0,0,0) node [vblack] {} -- (1,0,0) node [vblack] {} -- (1,0,1) node [vblack] {} -- (2,0,1) node [vblack] {} -- (3,0,1) node [vblack] {} -- (4,0,1) node [vblack] {} -- (5,0,1) node [vblack] {} -- (6,0,1) node [vblack] {} -- (6,0,0) node [vblack] {} -- (7,0,0) node [vblack] {};
				\strand[thick] (0,1,0) node [vblack] {} -- (1,1,0) node [vblack] {} -- (2,1,0) node [vblack] {} -- (2,0,0) node [vblack] {} -- (3,0,0) node [vblack] {} -- (4,0,0) node [vblack] {} -- (5,0,0) node [vblack] {} -- (5,1,0) node [vblack] {} -- (6,1,0) node [vblack] {} -- (7,1,0) node [vblack] {};
				\strand[thick] (7,2,1) node [vblack] {} -- (6,2,1) node [vblack] {} -- (6,1,1) node [vblack] {} -- (5,1,1) node [vblack] {} -- (5,2,1) node [vblack] {} -- (4,2,1) node [vblack] {} --  (4,1,1) node [vblack] {} -- (3,1,1) node [vblack] {} -- (3,1,0) node [vblack] {} -- (4,1,0) node [vblack] {} -- (4,2,0) node [vblack] {} -- (5,2,0) node [vblack] {} -- (6,2,0) node [vblack] {} -- (7,2,0) node [vblack] {};
				\strand[thick] (0,1,1) node [vblack] {} -- (1,1,1) node [vblack] {} -- (2,1,1) node [vblack] {} -- (2,2,1) node [vblack] {} -- (3,2,1) node [vblack] {} -- (3,2,0) node [vblack] {} -- (2,2,0) node [vblack] {} -- (1,2,0) node [vblack] {} -- (1,2,1) node [vblack] {} -- (0,2,1) node [vblack] {};
				\draw (0,0,1) node [vgrey] {};
				\draw (0,2,0) node [vgrey] {};
				\draw (7,0,1) node [vgrey] {};
				\draw (7,1,1) node [vgrey] {};
			\end{knot}
			
			\tikzset{xshift=5.8cm}
			\draw[line width=1.5pt, ->] (2.5,0,0.75) -- (3.5,0,0.75);
			
			\tikzset{xshift=4cm}
			\begin{knot}[consider self intersections=true, clip width=3, clip radius=3pt, end tolerance=1pt]
				\strand[thick] (0,0,0) node [vblack] {} -- (1,0,0) node [vblack] {} -- (1,0,1) node [vblack] {} -- (2,0,1) node [vblack] {} -- (3,0,1) node [vblack] {} -- (3,0,0) node [vblack] {} -- (2,0,0) node [vblack] {} -- (2,1,0) node [vblack] {} -- (1,1,0) node [vblack] {} -- (0,1,0) node [vblack] {};
				\strand[thick] (7,0,0) node [vblack] {} -- (6,0,0) node [vblack] {} -- (6,0,1) node [vblack] {} -- (5,0,1) node [vblack] {}  -- (4,0,1) node [vblack] {} -- (4,0,0) node [vblack] {} -- (5,0,0) node [vblack] {} -- (5,1,0) node [vblack] {} -- (6,1,0) node [vblack] {} -- (7,1,0) node [vblack] {};
				\strand[thick] (7,2,1) node [vblack] {} -- (6,2,1) node [vblack] {} -- (6,1,1) node [vblack] {} -- (5,1,1) node [vblack] {} -- (5,2,1) node [vblack] {} -- (4,2,1) node [vblack] {} --  (4,1,1) node [vblack] {} -- (3,1,1) node [vblack] {} -- (3,1,0) node [vblack] {} -- (4,1,0) node [vblack] {} -- (4,2,0) node [vblack] {} -- (5,2,0) node [vblack] {} -- (6,2,0) node [vblack] {} -- (7,2,0) node [vblack] {};
				\strand[thick] (0,1,1) node [vblack] {} -- (1,1,1) node [vblack] {} -- (2,1,1) node [vblack] {} -- (2,2,1) node [vblack] {} -- (3,2,1) node [vblack] {} -- (3,2,0) node [vblack] {} -- (2,2,0) node [vblack] {} -- (1,2,0) node [vblack] {} -- (1,2,1) node [vblack] {} -- (0,2,1) node [vblack] {};
				\draw (0,0,1) node [vgrey] {};
				\draw (0,2,0) node [vgrey] {};
				\draw (7,0,1) node [vgrey] {};
				\draw (7,1,1) node [vgrey] {};
			\end{knot}
		\end{tikzpicture}
		\caption{If a polygon has a hidden 2-section, we can decompose it into two closed polygons. }
		\label{Fig:hidden2sectiondecomp}
    \end{figure}
We classify U-shapes and their corresponding $-2$-moves into three types:
Type I) U-shapes in the $\pm x$ direction; Type II) U-shapes that lie entirely in a hinge; Type III) U-shapes in the $\pm y$ or $\pm z$ direction with the bottom edge in a section.
Note that the only $-2$-move that removes edges from a section are Type I, hence hidden 2-sections only involve a Type I $-2$-move.

Related to establishing property a) for the $P_i$'s, we first prove the following result:

\begin{lem}\label{lem:shifted_4plat}
	A sufficient condition for an embedding of a link in $\tube^*$ to have a shifted diagram which is a 4-plat diagram is that the embedding only has 4-sections and, except possibly for the first and last section, no sections are hidden 2-sections. For embeddings with span greater than 1, a hidden 2-section in the first section can only involve a $+x$  direction U-shape and one in the last section can only involve a $-x$ direction U-shape.
\end{lem}

\begin{proof} 
	In $\tube^*$, a link embedding with span 0 or 1 and having only 4-sections, is either the unknot or the 2-component unlink and thus its shifted diagram is a 4-plat diagram. Otherwise, suppose $P$ is an embedding of a link in $\tube^*$ with span $m\geq 2$, having only 4-sections, having no $-x$  U-shape at $x=1/2$ and no $+x$ U-shape at $x=m-1/2$, and for $m>2$, having no hidden 2-sections at the half-integer planes $x=j+1/2$, $j=1,...,m-2$.
	$P$ has 4 ``strings'' (one for each edge in its first section) leaving from the plane $x=0$ and 4 strings ending in the plane $x=m$. If no string achieves a local $x$-maximum (i.e.\ locally highest $x$-value) or $x$-minimum at any integer plane $x=j$, $j=1,...,m-1$, then there are exactly 4 strings that extend from $x=0$ to $x=m$ and they can be represented by a 4-braid. Thus the associated shifted diagram must be a 4-plat diagram. 
 
	Hence we only need to establish that $P$ contains no local $x$-maxima or minima. Note that  $P$ has a local $x$-maximum (minimum) at $x=j$ only if there is a subwalk of $P$ in the plane $x=j$ whose endpoints are both joined to edges in the previous (next) section, i.e.\ to edges in the section at $x=j-1/2$ ($x=j+1/2$). Suppose to the contrary that $P$ contains a local $x$-maximum (minimum) at the plane $x=j$, $j\in\{1,...,m-1\}$. By the definition of a local $x$-maximum (minimum),  the $x$-maximum (minimum) occupies at least 2 vertices and one edge of the 6-vertex hinge at $x=j$, there are two edges entering (leaving) these vertices from the left (right) and there are no edges leaving these vertices to the right (left). Thus, for there to be a 4-section at $x=j+1/2$ ($x=j-1/2$), there must be four edges leaving the hinge at $x=j$ to the right (left) and hence there must be 4 vertices in that hinge that are not involved in the $x$-maximum (minimum). Thus the $x$-maximum (minimum) must occupy exactly one edge of the hinge and it must be part of a $-x$ ($+x$) Type I U-shape.   A Type I $-2$-move can thus be performed and hence $P$ has a hidden 2-section. This is a contradiction. 
\end{proof}

In order to determine the $P_i$'s, the first step is to ensure that there are no 6-sections.  
For an embedding $P$ with 6-sections, let $P'$ denote an embedding without 6-sections obtained from $P$ by ``collapsing'' all boxes containing 6-sections onto a leftmost plane, as illustrated in Figure~\ref{Fig:6-string}. 
Since the strings in the 6-sections are all parallel, this operation does not change the knot or link-type.
\begin{figure}
\centering
		\includegraphics[width=.7\textwidth]{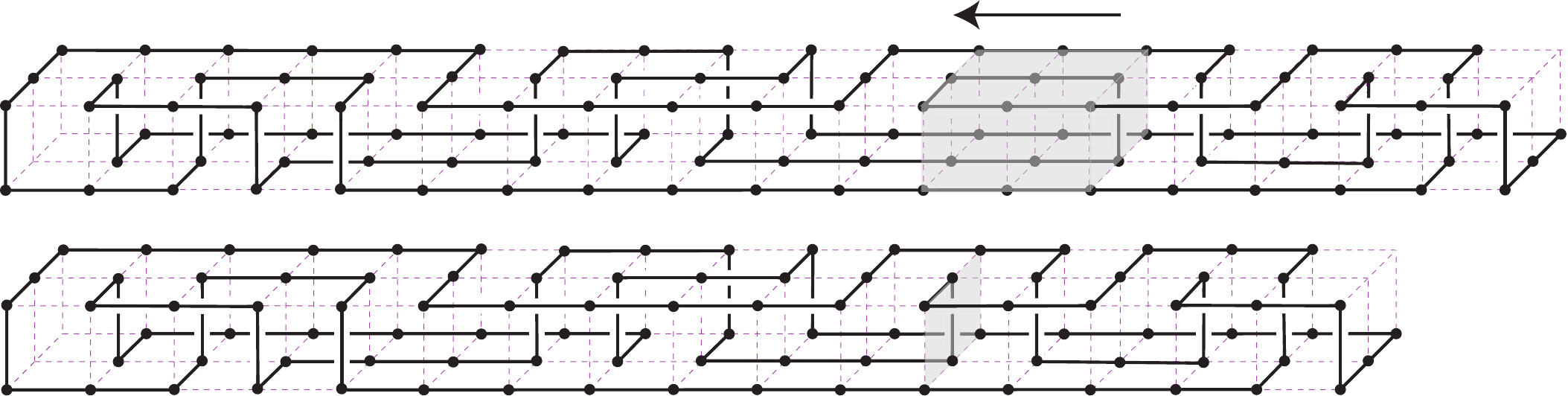}
		\caption{In case a polygon $P$ has sections of 6-strings, we change it into a polygon $P'$ with the same knot type by squashing those parts. }
		\label{Fig:6-string}
\end{figure}
\begin{prop}\label{6-string}
	If $P'$ is unknotted by inserting a finite number of blocks, then $P$ is also unknotted by the same number of insertions.
	The locations of insertions in $P'$ determine the locations for $P$.
\end{prop}

\begin{proof}
	Note that $P$ and $P'$ are isotopic in $\mathbb R^3$.
	For insertions of blocks in $P'$, we consider corresponding insertions of the same set of blocks in $P$.
	For example, suppose $P'$ is obtained by collapsing one box containing 6-sections with span $w$.
	Let $k$ be the $x$-coordinate of the leftmost plane of the box and $\frac{j}{2}$ the $x$-coordinate of the insertion of an $s$-block.
	Then we consider an insertion of the same $s$-block at $x=\frac{j}{2}$ in $P$ if $\frac{j}{2}<k$ and $x=\frac{j}{2}+w$ if $\frac{j}{2}>k$.
	Then the obtained embeddings are also isotopic in $\mathbb R^3$.
\end{proof}

Without loss of generality we assume now that $P$ has no 6-sections.
For this, first consider the case $f_L=1$.   
If $P$ satisfies the conditions of Lemma~\ref{lem:shifted_4plat}
then its shifted diagram is a 4-plat diagram.
Thus, for a more general $P$ we need to first find a sub-embedding $P'$ without 2-sections or hidden 2-sections.   We do this next in two lemmas.
For simplicity we refer to sections of an embedding that are neither first nor last 
as {\it interior} sections.

\begin{lem}
	Let $P$ be a non-split 4-plat 
	embedding. If every section of $P$ is either a 2-section or hidden 2-section then $P$ is an unknot.
	\label{lem:splitP}
\end{lem}

\begin{proof}
	Since every section is either a 2-section or can be reduced to a 2-section by a  $-2$-move, then applying all $-2$-moves makes every section that remains a 2-section and hence $P$ is isotopic to the unknot.  
\end{proof}

Therefore, if an embedding is non-split and non-trivial without 6-sections then it must contain at least one section which is neither a 2-section nor a hidden 2-section.

\begin{lem}
	\label{f_L}
	Let $P$ be an embedding of a  non-trivial non-split  link $L=L_1\#L_2\#\cdots\#L_{f_L}$ in $\tube^*$ without 6-sections. A set of embeddings $P_1,P_2, \dots,P_{f_L}$ can be determined with no interior 2-sections or hidden 2-sections, such that:
	$P_i$ has link-type  $L_i$ and $P_i$ minus its first and last hinge is a sub-block of $P$. 
\end{lem}

\begin{proof}
	We consider first the case $f_L=1$.
	Since $P$ is non-trivial and non-split, it must contain a section which is neither a 2-section nor a hidden 2-section.  
	A sub-block of $P$ will be called {\it suitable} if it contains no 2-sections or hidden 2-sections.  $P$ must have at least one suitable sub-block.  Let $B_1$, \dots, $B_m$ be the maximal (in span) suitable sub-blocks of $P$ ordered from left-to-right according to their occurrence in $P$.   We argue next, by induction on $m$, that we can obtain an embedding $P'$ of $L$ from exactly one of these sub-blocks and it is only different from the original block of $P$ in its left-most and right-most hinges.  
	
	Suppose $m=1$ and let $s$ be the span of $P$. Let $l$ and $r$ be the section numbers of the first and last sections of $B_1$. 
	If $l=1$  and $r=s$  then $P_1=P=B_1$.   
	If $l>1$ ($r<s$),  then the $(l-1)$st ($(r+1)$st) section must be a 2-section or hidden 2-section and hence either it is a 2-section already or there is a type~I $-2$ move  that turns the $(l-1)$st ($(r+1)$st) into a 2-section.  In either case there are two edges $l_1$ and $l_2$ ($r_1$ and $r_2$) in the $(l-1)$st ($(r+1)$st) section that form either the existing 2-section or the resulting 2-section.   
	The endpoints of $l_1$ and $l_2$ ($r_1$ and $r_2$) can then be joined in the leftmost (rightmost) hinge of $B_1$ to form $P_1$.
	If one of $l=1$  or $r=s$ (the  span of $P$)  then $P_1$ is obtained as above using only one section.
	$P_1$ necessarily has link-type $L$ (since the non-suitable blocks on either side of it are closed off into unknots) and its sections are identical to $B_1$ except on the left-most and right-most hinges.  Note that the first and last section of $B_1$ could become  hidden 2-sections in $P_1$ but all other sections are unchanged.  
	
	Let $m>1$ and assume that any embedding of a non-trivial $4$-plat 
	without 6-sections that has fewer than $m$ maximal suitable sub-blocks has a block that yields the appropriate $P_1$. 
	Now consider $P$ a non-trivial $4$-plat 
	without 6 sections with maximal suitable sub-blocks $B_1, \dots, B_m$.  Because the 2-sections or hidden 2-sections between the $B_i$'s are equivalent to connected sum operations and since $L$ is prime, only one of the $B_i$'s will yield the required $P_1$.  Close off $B_1$ into an embedding $P'$ as in the $m=1$ case. Either $P'$ has link-type $L$ or the closed off sequence of blocks (both suitable and unsuitable) to the right of $B_1$ has link-type $L$.  In the former case $P_1=P'$ and otherwise the inductive assumption yields the required $P_1$ from one of the remaining blocks.  Thus by induction on $m$ the result holds for $f_L=1$.
	
	For $f_L>1$, suppose $P$ is an embedding of a  non-split  link $L=L_1\#L_2\#\cdots\#L_{f_L}$ in $\tube^*$ without 6-sections.
	Since $P$ is non-trivial, then it must have maximal suitable sub-blocks $B_1$, $B_2$, \dots,$B_m$.  Since the sub-blocks are either located at the start or end of $P$, or are preceded and followed by 2-sections or hidden 2-sections, by Lemma~\ref{lem:shifted_4plat} they can each be closed off (as in the ${f_L=1}$ case) into embeddings of non-split  $4$-plats (because they have a 4-plat diagram). 
	Thus $m\geq f_L$ and for each prime $L_i$ there must be a corresponding $B_{j_i}$ which when closed off gives $P_i$ with link-type $L_i$. The resulting $P_i$'s have the required properties.
\end{proof}

Since
by the arguments of \cite{minsteptube} we can construct a braid block for each $3$-braid word in Proposition~\ref{prop:braidblock}{}, we can now prove the proposition. 
\begin{proof}[Proof of Proposition \ref{prop:braidblock}]
	There are $15$ ($=\binom{6}{4}$) types of $4$-sections in Figure~\ref{fig:braidblocks} (left). 
	The elementary braids $\sigma_1,\sigma_1^{-1},\sigma_2$ and $\sigma_2^{-1}$ are constructed as span-$3$ braid blocks in $\tube^*$ so that they have the same $4$-section at the ends of the block, as shown in Figure~\ref{fig:braidblocks} (right). By arranging such braid blocks in the order of the letters that appear in $w_0$, a span-$3c$ braid block can be obtained. Moreover, each type of $4$-section can be connected to the $4$-section in the ends of the block by a braid block with span at most $4$ that represents the trivial $4$-braid as shown in Figure~\ref{fig:braidblocks} (left), thus the span-$3c$ braid block can be modified into one with span at most $3c+8$ so that it can be inserted into the $4$-section.
\end{proof}
\begin{figure}[htb]
\centering
		\includegraphics[width=.9\textwidth]{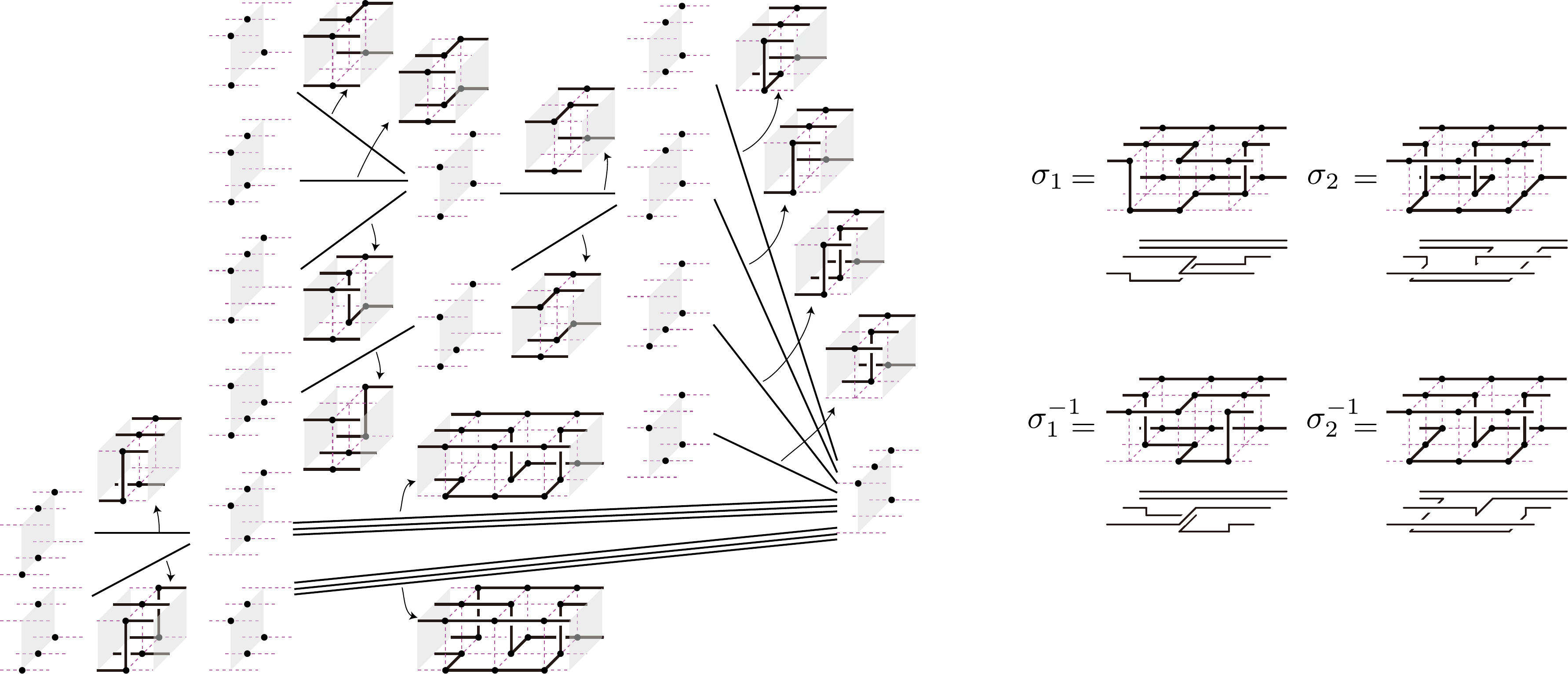}
		\caption{$15$ types of $4$-sections and span-$3$ braid blocks representing the elementary braids.}
		\label{fig:braidblocks}
\end{figure}		

\begin{proof}[Proof of Theorem~\ref{thm-insertion}]
	By Proposition \ref{6-string} we assume, without loss of generality, that an embedding $P$ of non-split link $L$  is without 6-sections.  From Lemma~\ref{lem:splitP}, for each prime factor $L_i$ of $L$ we can find a sub-block $B_i$ of $P$ which is associated with an embedding $P_i$ of  $L_i$ with no interior 2-sections or hidden 2-sections. Here $B_i$ only differs from $P_i$ in its first and last hinge.  
	Due to this, by Lemma~\ref{lem:shifted_4plat},  the shifted diagram of $P_i$ is a 
	4-plat diagram. 
	By 
	Theorem~\ref{thm:unknottinginsertion}{}, we can change this diagram into a diagram of the unknot by 
	inserting a suitable $4$-braid in one place. 
	By Proposition \ref{prop:braidblock}{}, we can construct an $s$-block corresponding to any sequence of $\sigma_i$'s and hence any $4$-braid. 
	The corresponding change in $P_i$ can be realized by an insertion of such an $s$-block.  The insertion is not in the first or last hinge and hence can be realized as an insertion in $B_i$.  
\end{proof}

\section{Lower bound: A pattern theorem for unknots in \texorpdfstring{$\tube^*$}{T}%
}
\label{sec:pattern} 

\tikzset{vblack/.style={circle, fill=black, inner sep=1.5pt}}
\tikzset{vblue/.style={circle, fill=blue, inner sep=1.5pt}}
\tikzset{vred/.style={circle, fill=red, inner sep=1.5pt}}
\tikzset{vgrey/.style={circle, fill=gray, inner sep=1.5pt}}
\tikzset{vorange/.style={circle, fill=orange, inner sep=1.5pt}}
\tikzset{vempty/.style={circle, draw, inner sep=1.4pt}}


%

\label{sec:newlower}
Towards proving the \textit{lower bound}, in this section we prove a \emph{pattern theorem} for unknot polygons using information from exact transfer-matrix calculations for all polygons (regardless of knot-type) in the tube $\tube^*$. 
In essence, a pattern theorem  states that for some particular type of lattice object (polygon, path, tree, etc.), a typical sample of sufficiently large size contains many copies of a small piece (a \emph{pattern}). 

Pattern theorems have been used previously, for example, to prove the FWD conjecture for polygons in tubes \cite{Atapour09, Sot98} and to study linking probabilities for the case of two polygons which span a tube \cite{Atapour10}.  Here we present the first proof of a pattern theorem for unknot polygons.  
In particular we show that for $n$ sufficiently large, all but exponentially few $n$-edge unknot polygons (unknots in $\tube^*$) contain a density ($\epsilon n$) of sections with exactly two edges (called 2-sections).

\begin{thm}\label{thm:density_of_2strings-nonHam}
Let $p_{\tube^*,n}(0_1,\leq\! k)$ be the number of unknots of length $n$ in $\tube^*$ which contain at most $k$ 2-sections. Then there exists an $\epsilon > 0$ such that
\begin{equation}\label{eqn:unknot_density_2section}\begin{split}
\limsup_{n\to\infty} \frac1n \log p_{\tube^*,n}(0_1,\leq\! \epsilon n) &< 
\log\mu_{\tube^*,0_1},
\end{split}\end{equation}
where $n$ is taken through multiples of 2 and $\mu_{\tube^*,0_1}$ is the unknot exponential growth constant as in~\eqref{unknottubegrowthconstant}.
\end{thm}

Theorem~\ref{thm:density_of_2strings-nonHam} leads to  a  general pattern theorem for unknot polygons (Corollary~\ref{thm:density_of_connectsumpatterns}). 
The method of proof 
follows that of \cite[Theorem 2.1]{Madras99} and we give the details later in Section~\ref{ssec:generalpatternthm}. 

\begin{cor}[Corollary of Theorem~\ref{thm:density_of_2strings-nonHam}.]
	\label{thm:density_of_connectsumpatterns}
	Let $P$ be an unknot connected sum pattern (defined in Section~\ref{sec:background}) in $\tube^*$.
	Let $p_{\tube^*,n}(0_1,P,\leq\! k)$ be the number of unknots of length $n$ in $\tube^*$ which contain at most $k$ $x$-translates of $P$. Then there exists an $\epsilon_P > 0$ such that
	\begin{equation}\label{eqn:unknot_density_P}\begin{split}
			\limsup_{n\to\infty} \frac1n \log p_{\tube^*,n}(0_1,P, \leq\! \epsilon_P n) &< \log\mu_{\tube^*,0_1}
	\end{split}\end{equation}
	where $n$ is taken through multiples of 2.
\end{cor}

The proof of Theorem~\ref{thm:density_of_2strings-nonHam} has two separate parts, which we summarize below. 
It will be convenient to refer to logarithms of growth constants. Hence,  given any growth constant such as $\mu_{\tube^*}$ or $\mu_{\tube^*,0_1}$,  we define a corresponding growth {\it rate},
$\kappa_{\tube^*}\equiv\log \mu_{\tube^*}$ or $\kappa_{\tube^*}(0_1)\equiv\log \mu_{\tube^*,0_1}$.  

\paragraph{Unknots with no 2-sections are exponentially rare:}

The first part involves showing that the growth rate of unknots with \emph{no 2-sections} is strictly less than that of all unknots. This means that unknots with no 2-sections are exponentially rare. 

Note that for lattice models characterized by \emph{finite transfer matrices}, 
showing that a pattern's non-occurrence  is exponentially rare 
follows from the fact that the dominant eigenvalue(s) of the transfer-matrix must decrease when the pattern in question is forbidden. (A similar idea can sometimes be applied in the absence of a finite transfer-matrix, when one knows something about the critical behavior of the generating function~\cite{Sum88}.) However, unknots in a lattice tube \emph{do not} have a finite transfer-matrix, so we are forced to take a quite different approach.

This part of the proof has two steps. In the first step, detailed in Section~\ref{ssec:no_2_sections}, we compute an upper bound on $\hat\kappa_{\tube^*}$, the growth rate of polygons with no 2-sections.

\begin{lem}\label{lem:tube*_no2sections_upperbound}
    \begin{equation}
        \hat\kappa_{\tube^*} < 0.446287.
    \end{equation}
\end{lem}

Lemma~\ref{lem:tube*_no2sections_upperbound} is proved in Section~\ref{ssec:no_2_sections} using a standard upper bound for the spectral radius of a matrix, 
\begin{equation}
    \tau_M \leq \lVert M^k\rVert^{1/k}
\end{equation}
for any $k\geq1$, where $\lVert\cdot\rVert$ is any consistent matrix norm. We use $\lVert \cdot \rVert_\infty$, which is the maximum absolute row sum. Note that by inclusion, Lemma~\ref{lem:tube*_no2sections_upperbound} also provides an upper bound for $\hat\kappa_{\tube^*}(0_1)$,
the growth rate of unknots with {no 2-sections}.

For the second step, we establish a lower bound on $\kappa_{\tube^*}(0_1)$.
\begin{lem}\label{lem:tube*_unknots_lowerbound}
\begin{equation}
    \kappa_{\tube^*}(0_1) \geq 0.620044.
\end{equation}
\end{lem}
The idea of the proof of Lemma~\ref{lem:tube*_unknots_lowerbound} is as follows. Because unknots in $\tube^*$ can be concatenated to form bigger unknots, we have
\begin{equation}
p_{\tube^*,m}(0_1)p_{\tube^*,n}(0_1) \leq p_{\tube^*,m+n+6}(0_1),
\end{equation}
where the $(+6)$ corresponds to the number of edges that must be inserted at the concatenation point. So $\log p_{\tube^*,n-6}(0_1)$ is a \emph{superadditive} sequence, and it follows that
\begin{equation}
\kappa_{\tube^*}(0_1) = \lim_{n\to\infty} \frac1n \log p_{\tube^*,n-6}(0_1) = \sup_{n\geq 0}\frac1n \log p_{\tube^*,n-6}(0_1).
\end{equation}
Brute-force enumeration yields $p_{\tube^*,24}(0_1) = 119796593$, from which the lower bound in the lemma follows. (It is known~\cite{minsteptube} that every polygon of length $\leq34$ in $\tube^*$ is an unknot, so in fact $p_{\tube^*,24}(0_1) = p_{\tube^*,24}$.)

Combining Lemmas~\ref{lem:tube*_no2sections_upperbound} and~\ref{lem:tube*_unknots_lowerbound}, we have
\begin{equation}\label{eqn:unknots_no_2strings_bounds}
\hat{\kappa}_{\tube^*}(0_1) \leq \hat{\kappa}_{\tube^*} < 0.446287 < 0.620044 \leq \kappa_{\tube^*}(0_1).
\end{equation}

\paragraph{Unknots have a positive density of 2-sections:}

The second part of the proof
relies crucially on the next lemma. It demonstrates that we can remove all the 2-sections from an unknot while controlling the number of new edges added.
\begin{lem}\label{lem:unknots_tube*_2strings_remove}
\begin{equation}\label{eqn:remove_2strings}
p_{{\tube^*},n}(0_1,\leq\! k) = \sum_{t=0}^k p_{{\tube^*},n}(0_1,t) \leq \sum_{t=0}^k 2^t \binom{\frac{n}{2}}{t}p_{\tube^*,n+Et}(0_1,0)
\end{equation}
for a constant $E$.
\end{lem}
The proof of Lemma~\ref{lem:unknots_tube*_2strings_remove} is given in Section~\ref{ssec:eliminating_2secs}. We  break the proof into two parts (sub-lemmas).
We first show how to take a polygon of length $n$ with $t$ 2-sections and break it apart into $t+1$ polygons with no 2-sections, with total length $n+2Dt$ for a constant $D$. We then show 
how to join those polygons back together into one large polygon of length $n+(2C+2D+2)t$ with no 2-sections, for another constant $C$. This process is reversible, with knowledge of where the splits and joins occurred ($\binom{\frac{n}{2}}{t}$ is an upper bound on the number of possibilities) and knowledge of which of a pair of possible choices at each cut formed the original 2-section (at most two choices). Here $E=2C+2D+2$.

\begin{proof}[Proof of Theorem~\ref{thm:density_of_2strings-nonHam}]
Take~\eqref{eqn:remove_2strings} with $k=\epsilon n$ and $\epsilon<\frac14$. Then the last summand on the RHS is the largest, so
\begin{equation}
p_{\tube^*,n}(0_1,\leq\! \epsilon n) \leq (\epsilon n+1)2^{\epsilon n} \binom{\frac{n}{2}}{\epsilon n} p_{\tube^*,n(1+\epsilon E)}(0_1,0).
\end{equation}
Take logs, divide by $n$ and take the $\limsup$:
\begin{equation}
\limsup_{n\to\infty}\frac1n \log p_{\tube^*,n}(0_1,\leq\! \epsilon n) \\ \leq -\epsilon\log\epsilon-\left(\frac12-\epsilon\right)\log(1-2\epsilon)+(1+\epsilon E)\hat{\kappa}_{\tube^*}(0_1).
\end{equation}
As $\epsilon\to0$, the RHS approaches $\hat{\kappa}_{\tube^*}$ and by~\eqref{eqn:unknots_no_2strings_bounds} we know that $\hat{\kappa}_{\tube^*}(0_1) \leq \hat{\kappa}_{\tube^*}< \kappa_{\tube^*}(0_1)$, so there must be an $\epsilon_T \in (0,1/4)$ which satisfies~\eqref{eqn:unknot_density_2section}. Here we have used the fact that 
	$\lim_{n\to\infty} n^{-1}{\log \binom{an}{bn}}=a\log a - b\log b - (a-b)\log (a-b)$ provided $0<b<a$. 
\end{proof}




Our approach to
adding the details to the proof of
Theorem~\ref{thm:density_of_2strings-nonHam} will be as follows. In Section~\ref{ssec:TMs} we will define the transfer matrix for polygons in (general) $\tube$ and explain its connection with the generating function for $p_{\tube^*,n}$ and its exponential growth rate. In Section~\ref{ssec:no_2_sections} we show that these concepts also apply to polygons in $\tube^*$ with no 2-sections. In Section~\ref{ssec:rigorous_bounds} we calculate two rigorous bounds: an upper bound on the growth rate of polygons in $\tube^*$ with no 2-sections (using the transfer matrix), and a lower bound on the growth rate of unknots in $\tube^*$ (using exact enumeration data). Importantly, these bounds establish that polygons with no 2-sections have a smaller growth rate than unknots. Finally in Section~\ref{ssec:eliminating_2secs}, we show that the 2-sections of \emph{any} polygon can be removed in a systematic (and reversible) way while controlling the number of edges that must be added. This is then used to prove that polygons with a sufficiently small (but positive) density of 2-sections have the same growth rate as those with none at all, and using the bounds from Section~\ref{ssec:rigorous_bounds}, the theorem follows.

In this section, it will be useful to label the 6 possible $(y,z)$ coordinates within the tube by $\mathcal V = \{a,b,c,d,e,f\}$ as in Figure~\ref{fig:abcdef}.

\begin{figure}%
	\centering
	\begin{subfigure}{\textwidth}
		\centering
		\begin{tikzpicture}[scale=1.2]
			\tikzset{vert/.style={circle, fill, inner sep=2.5pt}}
			\draw (0,0,0) node [vert, label=right:$a$] {};
			\draw (0,0,1) node [vert, label=right:$b$] {};
			\draw (0,1,0) node [vert, label=right:$c$] {};
			\draw (0,1,1) node [vert, label=right:$d$] {};
			\draw (0,2,0) node [vert, label=right:$e$] {};
			\draw (0,2,1) node [vert, label=right:$f$] {};
			
			\draw[ultra thick, <->] (2,1,1) node [label=above:$z$] {} to (2,1,0) -- (3,1,0) node [label=right:$x$] {};
			\draw[ultra thick, ->] (2,1,0) -- (2,2.5,0) node [label=above right:$y$] {};
		\end{tikzpicture}
		\caption{}
		\label{fig:abcdef}
	\end{subfigure}
	
	\vspace{5ex}
	
	\begin{subfigure}{\textwidth}
		\centering
		\begin{tikzpicture}
			\begin{knot}[consider self intersections=true, clip width=3, clip radius=3pt, end tolerance=1pt]
				\strand[very thick, black] (0,0,0) node [vblack] {} -- (1,0,0) node [vblack] {} -- (2,0,0) node [vblack] {} -- (2.5,0,0) {};
				\strand[very thick, blue] (2.5,0,0) -- (3,0,0) node [vblue] {} -- (3.5,0,0);
				\strand[very thick, black] (3.5,0,0) -- (4,0,0) node [vblack] {} -- (4,0,1) node [vblack] {} -- (5,0,1) node [vblack] {} -- (5,0,0) node [vblack] {} -- (5,1,0) node [vblack] {} -- (4,1,0) node [vblack] {} -- (3.5,1,0);
				\strand[very thick, blue] (3.5,1,0) -- (3,1,0) node [vblue] {} -- (2.5,1,0);
				\strand[very thick, black] (2.5,1,0) -- (2,1,0) node [vblack] {} -- (2,2,0) node [vblack] {} -- (1,2,0) node [vblack] {} -- (1,1,0) node [vblack] {} -- (0,1,0) node [vblack] {} -- (0,2,0) node [vblack] {} -- (0,2,1) node [vblack] {} -- (1,2,1) node [vblack] {} -- (1,1,1) node [vblack] {} -- (2,1,1) node [vblack] {} -- (2,2,1) node [vblack] {} -- (2.5,2,1);
				\strand[very thick, blue] (2.5,2,1) -- (3,2,1) node [vblue] {} -- (3,2,0) node [vblue] {} -- (3.5,2,0);
				\strand[very thick, black] (3.5,2,0) -- (4,2,0) node [vblack] {} -- (5,2,0) node [vblack] {} -- (5,2,1) node [vblack] {} -- (4,2,1) node [vblack] {} -- (4,1,1) node [vblack] {} -- (3.5,1,1);
				\strand[very thick, blue] (3.5,1,1) -- (3,1,1) node [vblue] {} -- (3,0,1) node [vblue] {} -- (2.5,0,1);
				\strand[very thick, black] (2.5,0,1) -- (2,0,1) node [vblack] {} -- (1,0,1) node [vblack] {} -- (0,0,1) node [vblack] {} -- (0,0,0);
				\flipcrossings{5,6}
			\end{knot}
			\begin{scope}[xshift=5cm]
				\begin{knot}[consider self intersections=true, clip width=3, clip radius=3pt, end tolerance=1pt]
					\strand[thick, red, only when rendering/.style={dashed}](2.5,0,0)
					.. controls  (1.75,0,0.25) and  (1.75,0,0.75) .. 
					(2.5,0,1);
					\strand[thick, red, only when rendering/.style={dashed}](2.5,1,0)
					.. controls  (1.65,1.5,0.5) and  (1.65,1.75,0.75) .. 
					(2.5,2,1);
					\strand[very thick, blue] (2.5,0,0) node [left] {\color{black} $a$} -- (3,0,0) node [vblue] {} -- (3.5,0,0) node [right] {\color{black} $a$};
					\strand[very thick, blue] (3.5,1,0) node [right] {\color{black} $c$} -- (3,1,0) node [vblue] {} -- (2.5,1,0) node [left] {\color{black} $c$};
					\strand[very thick, blue] (2.5,2,1) node [left] {\color{black} $f$} -- (3,2,1) node [vblue] {} -- (3,2,0) node [vblue] {} -- (3.5,2,0) node [right] {\color{black} $e$};
					\strand[very thick, blue] (3.5,1,1) node [right] {\color{black} $d$} -- (3,1,1) node [vblue] {} -- (3,0,1) node [vblue] {} -- (2.5,0,1) node [left] {\color{black} $b$};
					\flipcrossings{1}
				\end{knot}
			\end{scope}
		\end{tikzpicture}
		\caption{}
		\label{fig:1pattern_pairings}
	\end{subfigure}
	\caption{\textbf{(a)} The six vertices in a hinge of the $2\times1$ tube. \textbf{(b)} A polygon $P$ in $\tube^*$ and the 1-block $\beta$ (in blue) occurring between $x=3\pm\frac12$. Here $\mathcal{E}_\beta=\{a,b,c,f\}$ and $\mathcal{E}'_\beta=\{a,c,d,e\}$. The pairing $\rho$ on $\mathcal{E}_\beta$ induced by $P_{\rm{left}}$ is $\{\{a,b\},\{c,f\}\}$ (indicated by the dashed red lines), and this induces the pairing $\rho'$=$\{\{a,d\},\{c,e\}\}$ on $\mathcal{E}'_\beta$. The pair $(\beta,\rho$) forms a 1-pattern.}
	
\end{figure}

\subsection{Transfer matrices for polygons in \texorpdfstring{$\tube$}{T}}\label{ssec:TMs}

We begin by introducing 1-patterns in $\tube$. The definitions presented here are the same as those appearing in~\cite{BES19,BEISS18}.

Recall from Section~\ref{sec:background} that a 1-block is any (nonempty) collection of vertices, edges and half-edges that can comprise the part of a polygon in $\tube$ between planes $x=k\pm\frac12$ for some $k\in\mathbb{Z}$. For a 1-block $\beta$, let $\mathcal{E}_\beta=\{e_1,\dots,e_{2m}\}$ be the set of half-edges on the left (it may also be the case that $m=0$). Similarly let $\mathcal{E}'_\beta$ be the set of half-edges on the right. If $m>0$ and $P$ is a polygon in $\tube$ containing an occurrence of $\beta$, then the part of $P$ to the left of $\beta$ ($P_{\rm{left}}$) induces a partition of $\mathcal{E}_\beta$ into pairs (for any $e_i$, follow the edges of $P$ on the left of $\beta$ from $e_i$ until eventually arriving back at some $e_j$; then $e_i$ and $e_j$ get paired). See Figure~\ref{fig:1pattern_pairings}. Then for a given 1-block $\beta$ we  define $\Pi(\mathcal{E}_\beta)$ to be the set of all pair-partitions of $\mathcal{E}_\beta$ which are induced by some polygon containing $\beta$ in $\tube$. If $m=0$ then $\Pi(\mathcal{E}_\beta) = \{\emptyset\}$.

We define a \emph{1-pattern} $\pi$ to be a pair $(\beta,\rho)$ where $\beta$ is a 1-block and $\rho\in\Pi(\mathcal{E}_\beta)$. Note that a non-empty $\rho$  induces a pairing  on $\mathcal{E}'_\beta$ -- call this pairing $\rho'$. (If $\rho=\emptyset$, then the edges in $\beta$'s hinge induce the pairing $\rho'$.) See Figure~\ref{fig:1pattern_pairings}. If $\rho$ and $\rho'$ are both nonempty then we say $\pi$ is a \emph{proper 1-pattern}; if only $\rho$ is empty then $\pi$ is a \emph{starting 1-pattern}; and if only $\rho'$ is empty then $\pi$ is an \emph{ending 1-pattern}. We denote the sets of starting, proper and ending 1-patterns in $\tube$ by $\mathcal{A}_\mathrm{S},\mathcal{A}_\mathrm{P}$ and $\mathcal{A}_\mathrm{E}$ respectively. We also define ${\cal A}_0$ to be the set of all polygons in $\tube$ with span $s=0$; namely the 1-patterns for which both $\rho$ and $\rho'$ are empty. Note that the \emph{connected sum} patterns defined in Section~\ref{sec:background} with span 1 are indeed a subset of the 1-patterns defined here -- because connected sum patterns start and end with 2-sections, there is only one possible pairing on the leftmost (and rightmost) half-edges.

Given two 1-patterns $\pi_1 = (\beta_1,\rho_1)$ and $\pi_2 = (\beta_2,\rho_2)$, we say $\pi_2$ can \emph{follow} $\pi_1$ if $\rho'_1 = \rho_2$.

With this definition of patterns, we can follow the approaches used in \cite{Eng2014Thesis}  to obtain
transfer matrices.     We assign a labelling to the elements of ${\cal A}_\mathrm{P}$ and denote them as $\pi_1,\pi_2, \ldots, \pi_{N}$.   Then we obtain the $N\times N$ transfer matrix $T(x)$ for proper 1-patterns as follows:
\begin{equation}
	\left[T(x)\right]_{i,j} = \begin{cases} x^{n_{\pi_j}} & \text{if $\pi_j$ can follow $\pi_i$} \\ 0 & \text{otherwise,}\end{cases}
\end{equation}
where $n_{\pi}$ is the length of the 1-block from which the 1-pattern $\pi$ was derived.  We also need two other matrices.
We assign a labelling to the elements of ${\cal A}_\mathrm{S}$ and denote them as $\alpha_1,\alpha_2, \ldots, \alpha_{N_\mathrm{S}}$.   Then we obtain the $N_\mathrm{S}\times N$ start matrix $A(x)$ for starting 1-patterns as follows:
\begin{equation}
	\left[A(x)\right]_{i,j} = \begin{cases} x^{n_{\alpha_i} + n_{\pi_j}} & \text{if $\pi_j\in {\cal A}_\mathrm{P}$ can follow $\alpha_i\in{\cal A}_\mathrm{S}$} \\ 0 & \text{otherwise}\end{cases} .
\end{equation}
We also assign a labelling to the elements of ${\cal A}_\mathrm{E}$ and denote them as $\gamma_1,\gamma_2, \ldots, \gamma_{N_\mathrm{E}}$.   Then we obtain the $N\times N_\mathrm{E}$ end matrix $B(x)$ for ending 1-patterns as follows:
\begin{equation}
	\left[B(x)\right]_{i,j} = \begin{cases} x^{n_{\gamma_j}} & \text{if $\gamma_j\in {\cal A}_\mathrm{E}$ can follow $\pi_i\in{\cal A}_\mathrm{P}$} \\ 0 & \text{otherwise}\end{cases} .
\end{equation}

Define  $p_{\tube,n,s}$ to be the number of polygons counted in $p_{\tube,n}$ that have span $s$. The generating function for polygon counts can be expressed in terms of these  matrices as follows: 
\begin{align}
	G(x)=\sum_{n\geq 4} p_{\tube,n} x^{n} &=\sum_{s=0}^1\sum_{n\geq 4}  p_{\tube,n,s} x^{n}+\sum_{i,j}\left[\sum_{t\geq 0} A(x) \left(T(x)\right)^tB(x)\right]_{i,j} \notag\\
	&=\sum_{s=0}^1\sum_{n\geq 4}  p_{\tube,n,s} x^{n}+\sum_{i,j}\left[  A(x)(I-T(x))^{-1}B(x)\right]_{i,j}. \label{eqn:generating_function_matrices}
\end{align}
The radius of convergence of $G(x)$ is given by $\mu_{\tube}^{-1}=e^{-\kappa_{\tube}}$,  and since  for any $\tube_{M_1,M_2}$ it is known that $T(x)$ is aperiodic and irreducible \cite{Sot98}, the radius of convergence can also be determined by the smallest value of $x>0$ which satisfies $\det(I-T(x))=0$.  For the case that $\tube=\tube^*$, $T(x)$ has been determined (it has dimensions $1829\times1829$) and from that $\log\mu_{\tube^*}=\kappa_{\tube^*}=0.82694822250681$, with numerical error expected to be confined to the last digit.
Note that $\mu_{\tube^*}=e^{\kappa_{\tube^*}}$ must be an algebraic number, but we have not attempted to compute its minimal polynomial.

\subsection{Polygons with no 2-sections}\label{ssec:no_2_sections}

We next turn our attention to polygons in $\tube$  which have no 2-sections. For these polygons, our goal is to establish a relationship between their exponential growth rate and the radius of convergence of their generating function.
We accomplish this here by defining an appropriate transfer matrix and establishing that it is aperiodic and irreducible.   Note that we are assuming the dimensions of $\tube$ are large enough so that a polygon in $\tube$ can contain a 4-section. 

Define  $p_{\tube,n,s}(0)$ to be the number of polygons counted in $p_{\tube,n,s}$ that have no 2-sections and $p_{\tube,n,s}(K,0)$ to be the number of polygons counted in $p_{\tube,n,s}(0)$ that have knot type $K$. Also define $p_{\tube,n}(0)=\sum_s p_{\tube,n,s}(0)$  and $p_{\tube,n}(K,0)=\sum_s p_{\tube,n,s}(K,0)$. We can obtain a transfer matrix associated with $p_{\tube,n}(0)$ by removing all  1-patterns which cannot occur in a polygon without 2-sections from the transfer matrices above.  Specifically, 
we obtain new matrices $\hat{A}(x)$, $\hat{T}(x)$, and $\hat{B}(x)$ respectively from $A(x), T(x), B(x)$ by deleting all rows and columns associated with 1-patterns that cannot occur in a polygon without 2-sections. Analogously to~\eqref{eqn:generating_function_matrices}, we have that 
\begin{align}
	\hat{G}(x)&=\sum_{n\geq 4}  p_{\tube,n}(0) x^{n}= \sum_{n\geq 4} \sum_{s\geq 0} p_{\tube,n,s}(0) x^{n} \\ 
	&=\sum_{s=0}^1\sum_{n\geq 4}  p_{\tube,n,s}(0) x^{n}+\sum_{i,j}\left[\hat{A}(x)(I-\hat{T}(x))^{-1}\hat{B}(x)\right]_{i,j}.
\end{align}
We denote the set of proper 1-patterns that remain by $\hat{\mathcal{A}}_\mathrm{P}$ and note that it is not empty since there are proper 1-patterns consisting of exactly 4 edges in the $x$-direction which can occur in a polygon with no 2-sections.   $\hat{\mathcal{A}}_\mathrm{P}$ is necessarily a subset of the set of 1-patterns of $\mathcal{A}_\mathrm{P}$ that contain no 2-sections; we show in Section~\ref{ssec:eliminating_2secs} via Lemma~\ref{lem:split_at_first_2string} that in fact these two sets of 1-patterns are the same.  

The radius of convergence of $\hat{G}(x)$ is given by 
$\hat{\mu}_{\tube}^{-1}=e^{-\hat{\kappa}_\tube}$ where 
\begin{equation}
	\label{mu_no_2section}
	\hat{\mu}_{\tube}= \limsup_{n\to\infty}  ~(p_{\tube,n}(0))^{1/n} ,
\end{equation}
with the limit  taken through even values of $n$.
If the limit superior can be replaced by a limit then $\hat{\mu}_{\tube}$ is the exponential growth constant for polygons in $\tube$ which have no 2-sections. In fact this is the case for $\tube^*$, as follows from Lemma~\ref{lem:concat_no_2strings} below.

The matrix $\hat{T}(x)$ is much smaller than $T(x)$; indeed for $\tube^*$, $\hat{T}(x)$ has dimensions $553\times553$. $\hat{T}(x)$ is aperiodic, since there are many 1-patterns which can occur consecutively in a polygon without 2-sections (e.g. any 1-pattern in $\hat{\mathcal{A}}_\mathrm{P}$ consisting of exactly 4 edges in the $x$-direction). 
Further, if the matrix $\hat{T}(x)$ is irreducible, then it will follow that the radius of convergence of $\hat{G}(x)$, $e^{-\hat{\kappa}_\tube}$, is the smallest value of $x>0$ satisfying $\det(I-\hat{T}(x)) = 0$. 
Thus it remains to establish that $\hat{T}(x)$ is irreducible to achieve the main goals of this section.

$\hat{T}(x)$ is irreducible if, given any two 1-patterns $i$ and $j$ in $\hat{\mathcal{A}}_\mathrm{P}$, there exists an integer $k\geq 1$ such that $\left(\hat{T}(x)\right)^k_{ij} > 0$, or equivalently there is a sequence of 1-patterns, starting with $i$ and ending in $j$, which can occur consecutively in a polygon without 2-sections in $\tube$. If $\hat{T}(x)$ is explicitly known then this can be established by finding a power of $\hat{T}(x)$ for which all the entries are non-zero for $x>0$.  Otherwise a standard approach is to use a ``concatenation'' argument.  Here, even though we have $\hat{T}(x)$ explicitly for $\tube^*$, we use the latter approach to prove irreducibility, since the same concatenation argument  establishes that the $\limsup$ in~\eqref{mu_no_2section} is a limit  and it is also useful for the arguments in the next sections.

For the concatenation argument, we need to establish that it is possible to join (i.e.\ concatenate) any two polygons without 2-sections to create a new polygon, also without 2-sections, in such a way that minimal changes are made to the original two polygons.  In particular, if the concatenation leaves unchanged any proper 1-patterns that occur in the polygons then irreducibility follows.  This is because if $i$ and $j$ are two 1-patterns in $\hat{\mathcal{A}}_\mathrm{P}$, then by definition they can both occur in polygons without 2-sections. Let $P_i$ and $P_j$ be such polygons containing $i$ and $j$ respectively. Then $P_i$ and $P_j$ can be concatenated to give a new polygon without 2-sections that contains both $i$ and $j$.  This will give the needed sequence of consecutive 1-patterns from $i$ to $j$ to establish that $\left(\hat{T}(x)\right)^k_{ij} > 0$ for some $k$.  Thus it remains to show that there is a way to concatenate two such polygons.  Because we are also interested in controlling knot-types, we  want the concatenation process to be equivalent to a connected sum.   The required concatenation restrictions are given in Lemma~\ref{lem:concat_no_2strings} below and we prove this lemma using two others, Lemmas~\ref{lem:stretch_to_right} and~\ref{lem:stretch_to_left}.  The lemmas are proved only for $\tube^*$, however, the arguments can be extended in a straightforward manner to larger tubes. See \cite{Atapour10, atapourphdthesis} for a similar concatenation argument for a model of 2 component links, called 2-SAPs,  in $\tube$.

\begin{lem}\label{lem:concat_no_2strings}
	Let $P_1$ and $P_2$ be two polygons with no 2-sections in $\tube^*$, of lengths $n_1$ and $n_2$ and spans $s_1$ and $s_2$ respectively. Then there exists a polygon $P^*$ of length $n_1+n_2+2C+2$ (with $C$ being the same as in Lemmas~\ref{lem:stretch_to_right} and~\ref{lem:stretch_to_left}) and span $s_1+s_2+9$ such that
	\begin{itemize}
		\item the first (starting from the left) $s_1$ sections and $s_1$ hinges of $P_1$ and $P^*$ are the same;
		\item the last  $s_2$ sections and $s_2$ hinges of $P_2$ and $P^*$ are the same;
		\item $P^*$ contains no 2-sections; and
		\item $K(P^*) = K(P_1) \# K(P_2)$, where $K(P)$ denotes the knot-type of polygon $P$.
	\end{itemize}
\end{lem}

To prove Lemma~\ref{lem:concat_no_2strings}, two simple technical lemmas will help. The strategy is to add edges to one end of a polygon so that the polygon either starts (Lemma~\ref{lem:stretch_to_left}) or ends (Lemma~\ref{lem:stretch_to_right}) in a specific way in order to make the concatenation easy. 

\begin{lem}\label{lem:stretch_to_right}
	Let $P$ be an $n$-edge polygon with span $s\geq0$ and no 2-sections in $\tube^*$. Then there exists a constant $C$ and a polygon $P'$ of length $n+C$ such that
	\begin{itemize}
		\item the first (starting from the left-most hinge) $s$ sections and $s$ hinges of $P$ and $P'$ are the same;
		\item $P'$ contains no 2-sections;
		\item $P$ and $P'$ have the same knot type; and
		\item the rightmost hinge of $P'$ contains edges $ab$ and $df$ (recall Figure~\ref{fig:abcdef}), with no other edges occupied.
	\end{itemize}
\end{lem}

\begin{proof}
	The rightmost hinge of $P$ must necessarily have at least two edges which do not share a vertex (even if $P$ has span 0)
	since the section of the polygon immediately to the left of this hinge is not a 2-section.
	Choose two such edges and call them $e_1$ and $e_2$. Then there must be a sequence of $C/2=10$ BFACF $+2$ moves (see Figure~\ref{BFACF})  which ``bump out'' $e_1$ and $e_2$ to the right, moving them around if necessary, until $e_1$ has been ``stretched'' to $ab$ and $e_2$ has been ``stretched'' to $df$. The resulting polygon $P$ will necessarily have the same knot type as $P$, and will not contain any 2-sections.  $P'$ will have length $n+20$ and span $s+4$.   See Figure~\ref{fig:metacase2} for two of the relevant cases; the top picture shows the best case scenario where $e_1$ and $e_2$ are already $ab$ and $df$ while the bottom picture shows a worst case scenario, that is one requiring the full additional span in order to end at edges $ab$ and $df$ alone in the rightmost hinge.  
\end{proof}

\begin{figure}
	\centering
	\begin{subfigure}{\textwidth}
		\centering
		\resizebox{!}{10ex}{
			\begin{tikzpicture}[scale=0.8]
				\begin{knot}[consider self intersections=true, clip width=3, clip radius=3pt, end tolerance=1pt]
					\strand[very thick] (0,1,1) node [vblack] {} -- (0,2,1) node [vblack] {};
					\strand[very thick, blue] (0,0,0) node [vblue] {} -- (0,0,1) node [vblue] {};
					\draw (0,2,0) node [vgrey] {};
					\draw (0,1,0) node [vgrey] {};
				\end{knot}
				
				\draw[line width=3pt, ->] (2.5,0,0.75) -- (3.5,0,0.75);
				
				\tikzset{xshift=4.5cm}
				\begin{knot}[consider self intersections=true, clip width=3, clip radius=3pt, end tolerance=1pt]
					\strand[very thick] (0,1,1) node [vblack] {} -- (1,1,1) node [vblack] {}  -- (2,1,1) node [vblack] {} -- (3,1,1) node [vblack] {} -- (4,1,1) node [vblack] {}
					-- (4,2,1) node [vblack] {} -- (3,2,1) node [vblack] {} -- (2,2,1) node [vblack] {} -- (2,2,0) node [vblack] {} -- (1,2,0) node [vblack] {} -- (1,2,1) node [vblack] {} -- (0,2,1) node [vblack] {};
					\strand[very thick, red] (0,0,0) node [vred] {} -- (1,0,0) node [vred] {} -- (1,1,0) node [vred] {} -- (2,1,0) node [vred] {} -- (2,0,0) node [vred] {} -- (3,0,0) node [vred] {} -- (4,0,0) node [vred] {} 
					-- (4,0,1) node [vred] {} -- (3,0,1) node [vred] {} -- (2,0,1) node [vred] {} -- (1,0,1) node [vred] {} -- (0,0,1) node [vred] {};
					\draw (0,2,0) node [vgrey] {};
					\draw (0,1,0) node [vgrey] {};
					\foreach \x in {3,...,4} {
						\draw (\x,1,0) node [vempty] {};
						\draw (\x,2,0) node [vempty] {};
					}
					\flipcrossings{3,4}
				\end{knot}
			\end{tikzpicture}
		}
		\label{fig:meta_2v}
	\end{subfigure}
	
	\vspace{3ex}
	
	\begin{subfigure}{\textwidth}
		\centering
		\resizebox{!}{10ex}{
			\begin{tikzpicture}[scale=0.8]
				\begin{knot}[consider self intersections=true, clip width=3, clip radius=3pt, end tolerance=1pt]
					\strand[very thick] (0,2,0) node [vblack] {} -- (0,2,1) node [vblack] {};
					\strand[very thick, blue] (0,1,1) node [vblue] {} -- (0,0,1) node [vblue] {};
					\draw (0,0,0) node [vgrey] {};
					\draw (0,1,0) node [vgrey] {};
				\end{knot}
				
				\draw[line width=3pt, ->] (2.5,0,0.75) -- (3.5,0,0.75);
				
				\tikzset{xshift=4.5cm}
				\begin{knot}[consider self intersections=true, clip width=3, clip radius=3pt, end tolerance=1pt]
					\strand[very thick] (0,2,0) node [vblack] {} -- (1,2,0) node [vblack] {} -- (1,1,0) node [vblack] {} -- (2,1,0) node [vblack] {} -- (3,1,0) node [vblack] {} -- (3,1,1) node [vblack] {} -- (4,1,1) node [vblack] {} -- (4,2,1);
					\strand[very thick] (0,2,1) node [vblack] {} -- (1,2,1) node [vblack] {} -- (2,2,1) node [vblack] {} -- (3,2,1) node [vblack] {} -- (4,2,1) node [vblack] {};
					\strand[very thick, red] (0,0,1) node [vred] {} -- (1,0,1) node [vred] {} -- (1,0,0) node [vred] {} -- (2,0,0) node [vred] {} -- (3,0,0) node [vred] {} -- (4,0,0) node [vred] {} -- (4,0,1);
					\strand[very thick, red] (0,1,1) node [vred] {} -- (1,1,1) node [vred] {} -- (2,1,1) node [vred] {} -- (2,0,1) node [vred] {} -- (3,0,1) node [vred] {} -- (4,0,1) node [vred] {};
					\draw (0,1,0) node [vgrey] {};
					\draw (0,0,0) node [vgrey] {};
					\foreach \x in {1,...,4} \draw (\x,2,0) node [vempty] {};
					\draw (4,1,0) node [vempty] {};
					\flipcrossings{1,2}
				\end{knot}
			\end{tikzpicture}
		}
		\label{fig:meta_2vi}
	\end{subfigure}
	\caption{Bumping out edges.}
	\label{fig:metacase2}
\end{figure}

By symmetry we immediately have the following.
\begin{lem}\label{lem:stretch_to_left}
	Let $P$ be a polygon with no 2-sections in $\tube^*$, of length $n$ and span $s\geq0$. Then there exists a constant $C$ (the same constant as in Lemma~\ref{lem:stretch_to_right}) and a polygon $P''$ of length $n+C$ such that
	\begin{itemize}
		\item the rightmost $s$ sections and $s$ hinges of $P$ and $P''$ are the same;
		\item $P''$ contains no 2-sections;
		\item $P$ and $P''$ have the same knot type; and
		\item the leftmost hinge of $P''$ contains edges $ab$ and $ce$, with no other edges occupied.
	\end{itemize}
\end{lem}

The proof of Lemma~\ref{lem:concat_no_2strings} is now straightforward.
\begin{proof}[Proof of Lemma~\ref{lem:concat_no_2strings}]
	Let $P'_1$ and $P''_2$ be the polygons obtained by applying Lemma~\ref{lem:stretch_to_right} to $P_1$ and Lemma~\ref{lem:stretch_to_left} to $P_2$ respectively. Shift $P''_2$ so that its leftmost hinge is one unit to the right of $P'_1$'s rightmost hinge. Remove the edges $ab$ from each of those two hinges, and add the $aa$ and $bb$ edges between the two. Finally, apply a BFACF $+2$ move in the $+x$ direction to the $df$ edge in (what was) the rightmost hinge of $P'_1$.
\end{proof}
We remark at this point that, given $P^*$ and either $n_1$ or $s_1$ (and knowledge of the constant $C$), it is possible to unambiguously recover $P_1$ and $P_2$ from $P^*$.

Since $\hat{T}(x)$ is aperiodic and irreducible, it follows that the radius of convergence of $\hat{G}(x)$ is the smallest value of $x>0$ satisfying $\det(I-\hat{T}(x)) = 0$. 
For $\tube^*$, the matrix $\hat{T}(x)$ can be computed exactly, and from this it can be found that $\hat{\kappa}_{\tube^*} = 0.43623880228124$, with the numerical error expected to be confined to the last digit. 
Note that we do not need to prove the numerical accuracy of this estimate since in the next subsection we obtain a rigorous upper bound on $\hat{\kappa}_{\tube^*}$ (rather than just a numerical estimate) that is sufficient to prove Theorem~\ref{thm:density_of_2strings-nonHam}. Full details are given in the next section but first we note two more consequences of Lemma~\ref{lem:concat_no_2strings}.

From Lemma~\ref{lem:concat_no_2strings} we also have that:
\[p_{\tube^*,n_1}(0)p_{\tube^*,n_2}(0)\leq p_{\tube^*,n_1+n_2+42}(0).\]
Hence we have for $n_1$ and $n_2$ sufficiently large that:
\begin{equation}
	\log p_{\tube^*,n_1-42}(0)+\log p_{\tube^*,n_2-42}(0)\leq \log p_{\tube^*,n_1+n_2-42}(0),
\end{equation}
so that $\log p_{\tube^*,n-42}(0)$ is a superadditive sequence and it follows (see for example~\cite[\S1.2]{MadrasSlade}) that the $\limsup$ of~\eqref{mu_no_2section} is a limit with:
\begin{equation}
	\label{mu_no_2section_limit}
	\hat{\mu}_{\tube^*}= \lim_{n\to\infty}  ~(p_{\tube^*,n}(0))^{1/n} =\sup_{n\geq 0} ~(p_{\tube^*,n-42}(0))^{1/n}.
\end{equation}
Furthermore, because the concatenation corresponds to a connected sum and since the connected sum of two unknots yields an unknot, Lemma~\ref{lem:concat_no_2strings}  also implies the existence of the limit that defines the exponential growth constant, $\hat{\mu}_{\tube^*,0_1}$, for unknot polygons in $\tube$ without 2-sections.

\subsection{Rigorous numerical bounds}\label{ssec:rigorous_bounds}

The next task is to establish two rigorous bounds: an upper bound on $\hat{\kappa}_{\tube^*}$, the growth rate of polygons which have no 2-sections; and a lower bound on $\kappa_{\tube^*}(0_1)$, the growth rate of unknots. The first will be derived by analysis of the transfer matrix $\hat{T}(x)$; the second will follow from analysis of the first few terms of the unknot enumeration series.

Recall that $\hat{\kappa}_\tube = -\log\hat{x}_0$, where $\hat{x}_0$ is the smallest value of $x>0$ which satisfies $\det(I-\hat{T}(x)) = 0$. Note that, for real $x>0$, because $\hat{T}(x)$ is an irreducible, aperiodic matrix with non-negative entries, it has a unique dominant eigenvalue. This eigenvalue is real and positive, and is an increasing function of $x$. Therefore an equivalent definition of $\hat{x}_0$ is the unique positive $x$ such that $\hat{T}(x)$ has dominant eigenvalue 1. So we can find a lower bound for $\hat{x}_0$, and thus an upper bound for $\hat{\kappa}_\tube$, by finding any value of $x$ which yields a dominant eigenvalue (and thus spectral radius) smaller than~1.

The spectral radius $r_M$ of a matrix $M$ satisfies
\begin{equation}
	r_M \leq \lVert M^k\rVert^{1/k}
\end{equation}
for any $k\geq 1$, where $\lVert\cdot\rVert$ is any consistent matrix norm. We can, for example, use $\lVert \cdot \rVert_\infty$, which is the maximum absolute row sum (see for example~\cite{DerzkoPfeffer1965}).

In this case, when $x=0.64$, using $k=10$ establishes that the spectral radius is smaller than $0.99485$. So $\hat{x}_0>0.64$, and hence
\begin{equation}
	\hat{\kappa}_{\tube^*} < -\log0.64 = 0.446287,
\end{equation}
which establishes Lemma~\ref{lem:tube*_no2sections_upperbound}.

Next, we require a lower bound on $\kappa_{\tube^*}(0_1)$. Since any two polygons in $\tube^*$ can be concatenated with the addition of exactly 6 edges,  giving their topological connected sum, we have
\begin{equation}
	p_{\tube^*,m}(0_1)p_{\tube^*,n}(0_1) \leq p_{\tube^*,m+n+6}(0_1)
\end{equation}
or equivalently
\begin{equation}
	\log p_{\tube^*,m-6}(0_1) + \log p_{\tube^*,n-6}(0_1) \leq \log p_{\tube^*,m+n-6}(0_1).
\end{equation}
So $\log p_{\tube^*,n-6}(0_1)$ is a \emph{superadditive} sequence, and it follows that
\begin{equation}
	\kappa_{\tube^*}(0_1) = \lim_{n\to\infty} \frac1n \log p_{\tube^*,n-6}(0_1) = \sup_{n\geq 0}\frac1n \log p_{\tube^*,n-6}(0_1).
\end{equation}

For any $n\geq0$, $\frac1n \log p_{\tube^*,n-6}(0_1)$ is thus a lower bound on $\kappa_{\tube^*}(0_1)$. We have computed the first few terms in this sequence by exhaustively generating all polygons and checking their knot type. Various simplifications make this task somewhat less onerous than it may seem at first. The data is presented in Table~\ref{table:unknot_counts}. We see that $n=30$ gives $\kappa_{\tube^*}(0_1) \geq 0.620044$, establishing Lemma~\ref{lem:tube*_unknots_lowerbound}.

We thus have that 
\begin{equation}
	\hat \kappa_{\tube^*} (0_1) \leq \hat{\kappa}_{\tube^*} < -\log0.64 = 0.446287< 0.620044 \leq \kappa_{\tube^*}(0_1).
\end{equation}
Since unknot polygons without 2-sections are a subset of all polygons without 2-sections, this establishes that unknot polygons with no 2-sections are exponentially rare in the set of all unknot polygons in $\tube^*$.  So all but exponentially few unknot polygons in $\tube^*$ must contain  2-sections, but this is not sufficient, yet,  to prove Theorem~\ref{thm:density_of_2strings-nonHam}.  We must also establish that they contain a non-zero density of 2-sections.    We establish more lemmas in the next subsection which allow us to obtain this stronger result.

\begin{table}%
	\caption{The first few values in the sequences $p_{\tube^*,n}(0_1)$ and $\frac1n\log p_{\tube^*,n-6}(0_1)$ for $\tube^*$, the $2\times1$ tube.}
	\label{table:unknot_counts}
	\centering
	\begin{tabular}{l | l | l}
		$n$ & $p_{\tube^*,n}(0_1)$ & $\frac1n \log p_{\tube^*,n-6}(0_1)$ \\
		\hline
		4 & 9 & $-\infty$ \\
		6 & 42 & $-\infty$ \\
		8 & 209 & $-\infty$ \\
		10 & 1,113 & 0.219722 \\
		12 & 5,835 & 0.311472 \\
		14 & 30,561 & 0.381595 \\
		16 & 160,119 & 0.438426 \\
		18 & 838,043 & 0.481757 \\
		20 & 4,383,657 & 0.516374 \\
		22 & 22,917,673 & 0.544712 \\
		24 & 119,796,593 & 0.568284 \\
		26 & ? & 0.588207 \\
		28 & ? & 0.605265 \\
		30 & ? & 0.620044
	\end{tabular}
\end{table}

\subsection{Eliminating 2-sections}\label{ssec:eliminating_2secs}

Two more lemmas for this section relate polygons with 2-sections to those without. The first shows that a single 2-section can be removed with the addition of a fixed number of edges.
Note that these arguments are for $\tube^*$ but we expect that they can be modified in a straightforward manner to hold for larger $\tube$.

\begin{lem}\label{lem:split_at_first_2string}
	Let $P$ be a polygon of span $s$ with $t\geq1$ 2-sections. Say the leftmost 2-section is between $x=k$ and $x=k+1$, and that there are $n_1$ edges of $P$ in the half-space $x\leq k$ and $n_2$ edges in the half-space $x\geq k+1$. Then there exists a constant $D$ and polygons $P_1$ and $P_2$, of lengths $n_1+D$ and $n_2+D$ respectively, such that
	\begin{itemize}
		\item the leftmost $k$ sections and $k$ hinges of $P$ and $P_1$ are the same;
		\item the rightmost $s-k-1$ sections and $s-k-1$ hinges of $P$ and $P_2$ are the same;
		\item $P_1$ has no 2-sections;
		\item $P_2$ has $t-1$ 2-sections; and
		\item $K(P) = K(P_1) \# K(P_2)$.
	\end{itemize}
\end{lem}

\begin{proof}
	Let $\mathcal C$ denote the set of six edges in an arbitrary section of $\tube^*$. 
	Say the 2-section in question contains edges $u,v\in\mathcal C$. We break $P$ at $x=k+1/2$, and first consider only the left piece, call it $P_{\rm{left}}$. Begin by extending the half-edges of $P_{\rm{left}}$ to full edges, i.e.\ make $u$ and $v$ full edges again which end in the hinge $x=k+1$. Then at least one of the following must be true:
	\begin{enumerate}
		\item[(1)] $P$ has an edge $\eta$ in the hinge $x=k$ which is incident on neither $u$ nor $v$; or
		\item[(2)] there are adjacent vertices $p,q \neq u,v$ in the hinge $x=k$, with $P$ containing neither $p$ nor $q$.
	\end{enumerate}
	Note that if $k>0$ then (1) must be true, since between $x=k-1$ and $x=k$ there is either a 4-section or a 6-section. If $k=0$ then this can easily be checked by hand (there are not very many cases to consider).
	See Figure~\ref{fig:removing_2string} for an example for case (1). 
	
	In case (1), perform a single BFACF $+2$-move in the $+x$ direction to $\eta$ (extending it to the hinge $x=k+1$). Now in the hinge $x=k+1$, if there is a simple path joining the endpoints of $u$ and $v$ that does not intersect $\eta$, add such a path to create the polygon $P_1'$.  Otherwise $\eta$ must necessarily be the edge $dc$ and only one of $u$, $v$ has its endpoint in $\{a,b\}$; without loss of generality suppose that $u$ has $a$ as its endpoint. Add 2 edges in the $+x$ direction to each of $u$ and $v$ (extending to the hinge $x=k+3$). Then use a sequence of 3 BFACF $+2$-moves to bump $\eta$ to the edge $bd$ in the hinge $x=k+3$ (avoiding the extended $u$ and $v$).  Now there will be a  simple path between $u$ and $v$ which avoids $bd$ in the hinge $x=k+3$, add one to create $P_1'$. 
	Thus the right-most plane of $P_1'$ is $x=k+j+1$ with the  value of $j$ required depending on $u,v$ and $\eta$; it may be 0 or 2 (we will return to this below).
	
	In case (2), join $u$ and $v$ by a simple path in the hinge $x=k+1$ which contains the edge $pq$. Then apply a BFACF $+2$ move in the $-x$ direction from that edge to create $P_1'$.
	
	The number of edges added in the above two cases will depend on $j$, as well as the specific values of $u,v$, etc. However, there will always be at least one 4-section (in the section between $x=k$ and $x=k+1$, for example), whose edges can be duplicated if necessary (increasing the number of edges added by 4), and where a BFACF $+2$-move can be performed (increasing the number of edges added by 2). It follows that there is a $D$ sufficiently large such that the above operations can be performed on \emph{any} polygon $P$, so that the piece to the left of the first 2-section is closed off with a total length of $n_1+D$ to create $P_1$ from $P_1'$.
	
	The mirror image of  these arguments allows us to close off the right piece with total length $n_2+D$ and create the required $P_2$.
\end{proof}

\begin{figure}%
	\centering
        \scalebox{0.9}{
	\begin{tikzpicture}
		\begin{knot}[consider self intersections=true, clip width=3, clip radius=3pt, end tolerance=1pt]
			\draw[very thick, red] (2,0,0) -- (3,0,0);
			\draw[very thick, red] (2,2,1) -- (3,2,1);
			\strand[very thick, black] (0,0,0) node [vblack] {} -- (1,0,0) node [vblack] {} -- (2,0,0) node [vblack] {};
			\strand[very thick, black] (3,0,0) node [vblack] {} -- (4,0,0) node [vblack] {} -- (5,0,0) node [vblack] {} -- (5,1,0) node [vblack] {} -- (4,1,0) node [vblack] {} -- (3,1,0) node [vblack] {} -- (3,2,0) node [vblack] {} -- (4,2,0) node [vblack] {} -- (4,2,1) node [vblack] {} -- (5,2,1) node [vblack] {} -- (5,1,1) node [vblack] {} -- (4,1,1) node [vblack] {} -- (4,0,1) node [vblack] {} -- (3,0,1) node [vblack] {} -- (3,1,1) node [vblack] {} -- (3,2,1) node [vblack] {};
			\strand[very thick, black] (2,2,1) node [vblack] {} -- (1,2,1) node [vblack] {} -- (0,2,1) node [vblack] {} -- (0,1,1) node [vblack] {} -- (0,1,0) node [vblack] {} -- (1,1,0) node [vblack] {} -- (2,1,0) node [vblack] {} -- (2,1,1) node [vblack] {} -- (1,1,1) node [vblack] {} -- (1,0,1) node [vblack] {} -- (0,0,1) node [vblack] {} -- (0,0,0);
			\flipcrossings{1,2}
		\end{knot}
		
		\begin{scope}[xshift=-3.5cm,yshift=-3cm]
			\begin{knot}[consider self intersections=true, clip width=3, clip radius=3pt, end tolerance=1pt]
				\strand[very thick, orange] (2,0,0) -- (3,0,0) node [vorange] {} -- (4,0,0) node [vorange] {} -- (5,0,0) node [vorange] {} -- (5,1,0) node [vorange] {} -- (5,2,0) node [vorange] {} -- (5,2,1) node [vorange] {} -- (4,2,1) node [vorange] {} -- (3,2,1) node [vorange] {} -- (2,2,1);
				\strand[very thick, blue] (2,1,0) -- (3,1,0) node [vblue] {} -- (4,1,0) node [vblue] {} -- (4,1,1) node [vblue] {} -- (5,1,1) node [vblue] {} -- (5,0,1) node [vblue] {} -- (4,0,1) node [vblue] {} -- (3,0,1) node [vblue] {} -- (3,1,1) node [vblue] {} -- (2,1,1);
				\strand[very thick, black] (2,0,0) -- (1,0,0) node [vblack] {} -- (0,0,0) node [vblack] {} -- (0,0,1) node [vblack] {} -- (1,0,1) node [vblack] {} -- (1,1,1) node [vblack] {} -- (2,1,1);
				\strand[very thick, black] (2,1,0) -- (1,1,0) node [vblack] {} -- (0,1,0) node [vblack] {} -- (0,1,1) node [vblack] {} -- (0,2,1) node [vblack] {} -- (1,2,1) node [vblack] {} -- (2,2,1);
				\flipcrossings{1}
			\end{knot}
			\draw (2,0,0) node [vblack] {};
			\draw (2,1,1) node [vblack] {};
			\draw (2,1,0) node [vblack] {};
			\draw (2,2,1) node [vblack] {};
		\end{scope}
		
		\begin{scope}[xshift=3.5cm,yshift=-3cm]
			\begin{knot}[consider self intersections=true, clip width=3, clip radius=3pt, end tolerance=1pt]
				\strand[very thick, orange] (3,0,0) -- (2,0,0) node [vorange] {} -- (1,0,0) node [vorange] {} -- (0,0,0) node [vorange] {} -- (0,0,1) node [vorange] {} -- (0,1,1) node [vorange] {} -- (0,2,1) node [vorange] {} -- (1,2,1) node [vorange] {} -- (1,1,1) node [vorange] {} -- (2,1,1) node [vorange] {} -- (2,2,1) node [vorange] {} -- (3,2,1);
				\strand[very thick, blue] (3,1,0) -- (2,1,0) node [vblue] {} -- (1,1,0) node [vblue] {} -- (0,1,0) node [vblue] {} -- (0,2,0) node [vblue] {} -- (1,2,0) node [vblue] {} -- (2,2,0) node [vblue] {} -- (3,2,0);
				\strand[very thick, black] (3,0,0) -- (4,0,0) node [vblack] {} -- (5,0,0) node [vblack] {} -- (5,1,0) node [vblack] {} -- (4,1,0) node [vblack] {} -- (3,1,0);
				\strand[very thick, black] (3,2,0) -- (4,2,0) node [vblack] {} -- (4,2,1) node [vblack] {} -- (5,2,1) node [vblack] {} -- (5,1,1) node [vblack] {} -- (4,1,1) node [vblack] {} -- (4,0,1) node [vblack] {} -- (3,0,1) node [vblack] {} -- (3,1,1) node [vblack] {} -- (3,2,1); 
				\flipcrossings{1}
			\end{knot}
			\draw (3,0,0) node [vblack] {};
			\draw (3,1,0) node [vblack] {};
			\draw (3,2,0) node [vblack] {};
			\draw (3,2,1) node [vblack] {};
		\end{scope}
		
		\draw[line width=2.5pt, black, ->] (1,-0.7,0) -- (1,-2.3,0);
		\draw[line width=2.5pt, black, ->] (5,-0.7,0) -- (6.1,-2.3,0);
		
	\end{tikzpicture}
        }
	\caption{An illustration of the construction used in Lemma~\ref{lem:split_at_first_2string}. Above is a polygon in $\tube^*$ with a 2-section between $x=2$ and $x=3$ (indicated in red). Below are the two polygons $P_1$ and $P_2$. In this case the endpoints of the edges $u$ and $v$ are in $\{a,f\}$, and both sides are in case (1) -- on the left $\eta=cd$ and on the right $\eta=ce$ (we could also have used $\eta=bd$). The number of edges added to each side is $D=17$ (note that an extra BFACF $+2$ move was used on $P_2$ to achieve this).}
	\label{fig:removing_2string}
\end{figure}

We remark at this point that, given $P_1$ and $P_2$ (and knowledge of the constant $D$), it may not be possible to unambiguously recover the original polygon $P$. However, there will be at most two possibilities for $P$. 

We get two important consequences from Lemma \ref{lem:split_at_first_2string}.  The first is that it allows us to establish that the set of 1-patterns that can occur in a polygon in $\tube^*$ without 2-sections, 
$\hat{\mathcal{A}}_\mathrm{P}$, is equal to the set of all proper 1-patterns from ${\mathcal{A}}_\mathrm{P}$ with no 2-sections.  That is, we show that any proper 1-pattern in ${\mathcal{A}}_\mathrm{P}$ with no 2-sections can occur in some polygon without 2-sections.  (We note that we have already used this result in the previous section.)  The argument is as follows.  Let $\beta$ be a proper 1-pattern with no 2-sections and let $P$ be any polygon containing $\beta$. If $P$ has no 2-sections then we are done, otherwise suppose it has $t>1$ 2-sections.  Lemma~\ref{lem:split_at_first_2string} shows how to take a polygon with $t$ 2-sections and break it apart into $t+1$ polygons with no 2-sections. Since $\beta$ has no 2-sections it will not have been altered in this process and hence one of the resulting polygons without 2-sections must contain $\beta$.

The second consequence of Lemma~\ref{lem:split_at_first_2string} is Lemma~\ref{lem:unknots_tube*_2strings_remove}, which is the final result needed for the proof of Theorem~\ref{thm:density_of_2strings-nonHam}.



\begin{proof}[Proof of Lemma~\ref{lem:unknots_tube*_2strings_remove}]
    Let the constant $E$ in the lemma be $E=2C+2D+2$, where $C$ and $D$ are the constants used in Lemmas~\ref{lem:concat_no_2strings} and~\ref{lem:split_at_first_2string}.
	The method outlined in Lemma~\ref{lem:split_at_first_2string} shows how to take a polygon of length $n$ with $t$ 2-sections and break it apart into $t+1$ polygons with no 2-sections, with total length $n+2Dt$. The method outlined in Lemma~\ref{lem:concat_no_2strings} then shows how to join those polygons back together into one large polygon of length $n+(2C+2D+2)t$ with no 2-sections. This process is reversible, with knowledge of where the splits and joins occurred ($\binom{\frac{n}{2}}{t}$ is an upper bound on the number of possibilities) and knowledge of which of a pair of possible choices at each cut formed the original 2-section (at most two choices).
\end{proof}


\subsection{General pattern theorem for unknot polygons and ratio limit} \label{ssec:generalpatternthm}

Theorem~\ref{thm:density_of_2strings-nonHam} allows us to prove  a more general pattern theorem for unknot polygons (Corollary~\ref{thm:density_of_connectsumpatterns}) and then  a ratio limit theorem.  The method of proof for the general pattern theorem for unknots follows that of \cite[Theorem 2.1]{Madras99} and the ratio limit theorem follows directly from \cite[Theorem 2.2]{Madras99}.  The ratio limit theorem is useful for going from the first pair of bounds to the second pair of bounds in Theorem~\ref{thm-bounds}{}.


\begin{proof}[Proof of Corollary~\ref{thm:density_of_connectsumpatterns}]
	
	Fix any $\epsilon>0$ that satisfies  Theorem~\ref{thm:density_of_2strings-nonHam}. Define ${\cal P}_n(0_1,\epsilon)$ to be the set of $n$-edge unknot polygons with more than $\epsilon n$ 2-sections. 
	Let $P$ be an unknot connected sum pattern.
	
	Take any positive $\epsilon_P$ and define $p_{\tube^*,n}(0_1, > \epsilon n , P, \leq \epsilon_P n)$ to be the number of $n$ edge unknot polygons in $\tube^*$ with more than $\epsilon n$ 2-sections and at most $\epsilon_P n$ $x$-translates of $P$. 
	Note that 
	\begin{equation}
		p_{\tube^*,n}(0_1, P, \leq \epsilon_P n) \leq    p_{\tube^*,n}(0_1, > \epsilon n , P, \leq \epsilon_P n) + p_{\tube^*,n}(0_1, \leq \epsilon n ),
		\label{eqS39}
	\end{equation} 
	so if $p_{\tube^*,n}(0_1, > \epsilon n , P, \leq \epsilon_P n)\leq p_{\tube^*,n}(0_1, \leq \epsilon n )$ for all $n$ sufficiently large, then we have the required result from Theorem~\ref{thm:density_of_2strings-nonHam}.  Otherwise, 
    let ${\mathcal N(\epsilon_P)}$ be the (infinite) sequence of $n$'s  such that for $n \in \mathcal N(\epsilon_P)$, we have $p_{\tube^*,n}(0_1, > \epsilon n , P, \leq \epsilon_P n)>p_{\tube^*,n}(0_1, \leq \epsilon n )>0$.
	For any $n\in \mathcal N(\epsilon_P)$, take $\pi \in {\cal P}_n(0_1,\epsilon)$ having at most  $\lfloor \epsilon_P n\rfloor$ translates of $P$. 
	The strategy is to create a new polygon from $\pi$ that has more copies of $P$ and show that there must be exponentially more of such polygons.  
	
	Take  positive $\tau$ such that $\lfloor \epsilon n\rfloor  > \lfloor \tau n \rfloor > 1 $.  Choose any $\lfloor \tau n \rfloor$ of the first $\lfloor \epsilon n\rfloor$ 2-sections and ``insert'' the pattern $P$ at these 2-sections. Inserting $P$ at a 2-section can involve inserting a ``pad'' of edges on each end of $P$ to connect it to the 2-section. This can always be done so that the total number of edges added by an insertion is the same; let $k_P$ be this number. 
	The result is a new size $n+k_P\lfloor \tau n \rfloor$ unknot polygon with at most  $\lfloor \epsilon_P n\rfloor  + \lfloor \tau n \rfloor$ translates of $P$.  
	
	This gives for any $n\in \mathcal N(\epsilon_P)$:
	\begin{equation}
		\binom{\lfloor \epsilon n\rfloor}{\lfloor \tau n \rfloor} p_{\tube^*,n}(0_1, > \epsilon n , P, \leq \epsilon_P n) \leq    \binom{\lfloor \epsilon_P n\rfloor + \lfloor \tau n \rfloor}{\lfloor \tau n \rfloor} p_{\tube^*,n+k_P\lfloor \tau n \rfloor}(0_1).
	\end{equation}
	
	Taking logs on both sides, dividing by $n$ and then taking the $\limsup$ as $n$ goes to infinity gives 
	\begin{multline}
		\limsup_{n\to\infty;\,n\in \mathcal N(\epsilon_P)} \frac{1}{n} \log p_{\tube^*,n}(0_1, > \epsilon n , P, \leq \epsilon_P n) \\ \leq    \log\mu_{\tube^*,0_1} + \tau\left[\log\frac{(\mu_{\tube^*,0_1})^{k_P}(\epsilon_P+\tau)}{ (\epsilon-\tau)}\right] +\epsilon_P\left[\log\frac{(\epsilon_P+\tau)}{\epsilon_P}\right] +\epsilon\log\frac{(\epsilon -\tau)}{\epsilon},
	\end{multline}
	where the $\limsup$ is taken through the sequence $\mathcal N(\epsilon_P)$.
	
	Define $\displaystyle{Q=\frac{(\mu_{\tube^*,0_1})^{k_P}3^{3/2}}{2}}$. Then consider any $t\in (0,1)$ and set $\tau=t\epsilon$,  $\epsilon_P=\tau/2=t\epsilon/2$ to obtain
	\begin{equation}
		\limsup_{n\to\infty;\, n\in \mathcal N(\epsilon_P)} \frac{1}{n} \log p_{\tube^*,n}(0_1, > \epsilon n , P, \leq \epsilon_P n) \leq    \log\mu_{\tube^*,0_1} + \epsilon\log [t^t(1-t)^{1-t}Q^t]  .
	\end{equation}
	Taking $t=1/(Q+1)$ then gives $\epsilon_P=\epsilon/(2Q+2)$ and we obtain
	\begin{equation}
		\limsup_{n\to\infty;\, n\in \mathcal N(\epsilon_P)} \frac{1}{n} \log p_{\tube^*,n}(0_1, > \epsilon n , P, \leq \epsilon_P n) \leq    \log\mu_{\tube^*,0_1} + \epsilon\log \left[\frac{Q}{Q+1}\right] < \log\mu_{\tube^*,0_1}.
	\end{equation}
	Hence using~\eqref{eqS39} with $\epsilon_P=\epsilon/(2Q+2)$ along with the definition of $\mathcal N(\epsilon_P)$ gives
	\begin{align}
		\limsup_{n\to\infty}  \frac{1}{n} &\log  p_{\tube^*,n}(0_1, P, \leq \epsilon_P n)\notag \\
		&\leq   \max\left\{  \limsup_{n\to\infty} \frac{1}{n} \log p_{\tube^*,n}(0_1, > \epsilon n , P, \leq \epsilon_P n) ,  \limsup_{n\to\infty} \frac{1}{n} \log p_{\tube^*,n}(0_1, \leq \epsilon n ) \right\} \notag \\ &< \log\mu_{\tube^*,0_1},
	\end{align} 
	where here the $\limsup$ is taken through all even $n$. 
\end{proof}

Because we know the limit defining $\mu_{\tube^*,0_1}$ exists, when taken through even values of $n$, and because there are two connected sum unknot patterns $U$ and $V$ (see Figure~\ref{fig:UV}) which can be interchanged to change the number of edges in an unknot polygon from $n+2$ to $n$ or vice versa, then Corollary~\ref{thm:density_of_connectsumpatterns} combined with Madras' Theorem 2.2~\cite{Madras99} gives the following ratio limit result.

\begin{cor}\label{cor:ratiolimitcor}
	\begin{equation}
		\label{ratiolimit2}
		\lim_{n\to\infty} \frac{p_{\tube^*,n+2}(0_1)}{p_{\tube^*,n}(0_1)} =(\mu_{\tube^*,0_1})^2 ,
	\end{equation} 
	where the limit is taken through even values of $n$. More generally,
    \begin{equation}
\lim_{n\to\infty}\frac{p_{\tube^*,n+m}(0_1)}{p_{\tube^*,n}(0_1)}=(\mu_{\tube^*,0_1})^m, 
\label{ratiolimit}
\end{equation}
	again with $n$ even. Thus for any given even $m$ and for even $n$ sufficiently large, there exist constants $C_1$ and $C_2$ such that:
	\begin{equation}
		\label{ratiobounds}
		C_1 \,p_{\tube^*,n}(0_1)\leq  p_{\tube^*,n+m}(0_1) \leq C_2\, p_{\tube^*,n}(0_1).
	\end{equation} 
\end{cor}

%
%

\begin{figure}
	\centering
	\begin{subfigure}{\textwidth}
		\centering
		\resizebox{!}{10ex}{
			\begin{tikzpicture}[scale=0.8]
				\begin{knot}[consider self intersections=true, clip width=3, clip radius=3pt, end tolerance=1pt]
					\strand[very thick] (0.5,1,1) 
					{} -- (1,1,1) node [vblack] {}  -- (2,1,1) node [vblack] {} -- (2.5,1,1) 
					{} 
					(2.5,2,1) 
					{} -- (2,2,1) node [vblack] {} 
					-- (1,2,1) node [vblack] {} -- (0.5,2,1) 
					{};
					\draw (1,1,0) node [vempty] {}; 
					\draw (2,1,0) node [vempty] {}; 
					\draw (1,2,0) node [vempty] {}; 
					\draw (2,2,0) node [vempty] {}; 
					\flipcrossings{3,4}
				\end{knot}
				
				\draw[line width=3pt, <->] (3.5,1,0.75) -- (4.5,1,0.75);
				
				\tikzset{xshift=4.5cm}
				\begin{knot}[consider self intersections=true, clip width=3, clip radius=3pt, end tolerance=1pt]
					\strand[very thick] (0.5,1,1) 
					{} -- (1,1,1) node [vblack] {}  -- (2,1,1) node [vblack] {} -- (2.5,1,1) 
					{} 
					(2.5,2,1) 
					{} -- (2,2,1) node [vblack] {} -- (2,2,0) node [vblack] {} -- (1,2,0) node [vblack] {} -- (1,2,1) node [vblack] {} -- (0.5,2,1) 
					{};
					\draw (1,1,0) node [vempty] {}; 
					\draw (2,1,0) node [vempty] {}; 
					\flipcrossings{3,4}
				\end{knot}
			\end{tikzpicture}
		}
	\end{subfigure}

	\caption{Unknot connected sum pattern $V$ (on the left) can be interchanged with unknot pattern $U$ (on the right) to change an unknot polygon size by $\pm 2$ edges. }
	\label{fig:UV}
\end{figure}

\section{Proof of the main theorem}
\label{sec:Overview}
In this section we combine the arguments to prove Theorem~\ref{thm-bounds}. 

For the upper bound, we obtained new knot theory results about 4-plats.
In Theorem~\ref{thm:unknottinginsertion} we establish that any 4-plat diagram can be unknotted by the insertion of a braid word, where appropriate braid words can be determined from a minimal diagram of the 4-plat. Combining these results with properties of $\tube^*$, in Theorem~\ref{thm-insertion}  we then established that any lattice embedding of a non-split link in $\tube^*$ can be converted to an unknot polygon by the insertion of braid blocks. %
Theorems~\ref{thm-insertion} and~\ref{thm:unknottinginsertion} and 
their proofs were presented 
in Section~\ref{sec:newupper}.  

For the lower bound, we proved a pattern theorem for unknot polygons in $\tube^*$ by proving that unknot polygons have a density of 2-sections, 
Theorem~\ref{thm:density_of_2strings-nonHam}.  The theorem along with 
its proof were given in Section~\ref{sec:newlower}. 

\begin{proof}[Proof of Theorem~\ref{thm-bounds} from Theorems~\ref{thm-insertion} and~\ref{thm:density_of_2strings-nonHam}] 
By Theorem~\ref{thm-insertion}, we can insert a finite number of $s$-blocks to change an embedding of
  $L$ into an unknot polygon. Then
 $p_{\tube^*,n}(L)\le 
 \binom{n+d_L}{
 f_L}p_{\tube^*,n+d_L}(0_1)$, where $d_L$ is a fixed number (the total number of edges in the inserted parts) which is bounded above by a finite number times 
 $f_L$.
 The binomial term in the upper bound accounts for the number of ways different embeddings of $L$ could lead to the same unknot polygon; this is bounded above by the number of places (polygon edges) the inserted $f_L$ braid blocks could be located in the unknot polygon, namely $\binom{n+d_L}{
 f_L}$. 
Then $\binom{n+d_L}{f_L} \leq d_L \binom{n}{f_L}$ for $n$ sufficiently large.


For the lower bound, consider $L$ a prime link. Let $P$ be a connected sum pattern for $L$ and let $S$ be any 2-section. Then 
from Corollary~\ref{cor:unknot_insert_factors}, 
$P$ can be inserted (additional edges may be needed) at $S$ in a polygon 
(see for example Figure~\ref{fig1intro}D).
Note that any connected sum pattern can be inserted at any $S$ by adding a fixed number of edges. The options can be determined in a way similar (but simpler) to that used for inserting a braid block at a 4-section as illustrated in Figure~\ref{Fig:6-string}~(e). 
	The length of the resulting embedding is increased by some constant $\Delta$ depending on $P$ (but not $S$).

Now take any unknot polygon with $n$ edges and at least $\epsilon n$ 2-sections, with $\epsilon$ the same as in Theorem~\ref{thm:density_of_2strings-nonHam}. By the above, the pattern $P$ can be inserted at any one of those 2-sections, 
and the resulting lattice embedding will have link type $L$. We thus have 
\begin{equation}
 {\binom{\epsilon n}{1}} \left[p_{\tube^*,n}(0_1)-p_{\tube^*,n}(0_1, \leq \epsilon n)\right]\le p_{\tube^*,n+\Delta}(L).   
\end{equation} 
However, from Theorem~\ref{thm:density_of_2strings-nonHam} we have 
\begin{equation}
\lim_{n\to\infty} \frac{p_{\tube^*,n}(0_1)-p_{\tube^*,n}(0_1, \leq \epsilon n)}{p_{\tube^*,n}(0_1)}=1
\end{equation}
so that the numerator can be made arbitrarily close to the denominator for sufficiently large $n$. Hence, for example, there exists $N>0$ such that for all $n\geq N$,  $\frac{1}{2}\binom{\epsilon n}{1} p_{\tube^*,n}(0_1)\le p_{\tube^*,n+\Delta}(L)$.

This argument can be extended in a straightforward way to the case where $L$ is composite (see Corollary~\ref{cor:unknot_insert_factors} to Proposition~\ref{prop:which_knots_fit} in Section~\ref{subsec:4-plats}), and the lower bound follows.

The arguments above establish~\eqref{mainrelation}.
Equation~\eqref{finalform} can be obtained from~\eqref{mainrelation} by combining the
fact that ${\displaystyle \lim_{n\to\infty}} \frac{\log \binom{an}{b}}{\log {n }}=b$ with
~\eqref{ratiobounds} from Corollary~\ref{cor:ratiolimitcor}.
Note that the resulting constant $C_1$ in \eqref{finalform} will depend at least on $\epsilon$ and $e_L$ and the constant $C_2$ will depend at least on $d_L$ and $b_L$ so that  $C_1< C_2$.
%
%
\end{proof}

We note that Theorem~\ref{thm-insertion} combined with simpler known lower bounds (obtained by concatenating an arbitrary unknot embedding to an embedding of a given knot or link)  
is sufficient to 
prove that the exponential growth rate of fixed knot-type or link-type embeddings is the same as that for unknots. This establishes Conjecture~\ref{mainconjtube} (a). However, since this result is also a consequence of Theorem~\ref{thm-bounds}{} we do not give the details here. 

The proof of Theorem~\ref{thm-bounds} gives as a corollary %
a general pattern theorem for embeddings of non-split $L$.  The corollary is stated as Corollary~\ref{cor:density_of_connectsumpatterns} and proved in 
the next section. 
These results allow us to investigate the ``size'' of the linked region in a lattice link and 
we also present results related to this in the next section.

\section{The size of the linked region via a pattern theorem for fixed-link-type embeddings 
}\label{sec:size}

A general pattern theorem for fixed-link-type embeddings can be obtained as a corollary to Theorem~\ref{thm-bounds}{} and its proof.

\begin{cor}\label{cor:density_of_connectsumpatterns}
	Given a non-split link $L$ embeddable in $\tube^*$, let $P$ be a 2-section or an unknot connected sum pattern.
	Let $p_{\tube^*,n}(L,P,\leq\! k)$ be the number of embeddings of $L$ of length $n$ in $\tube^*$ which contain at most $k$ translates of $P$. Then there exists an $\epsilon_{L,P} > 0$ such that
	\begin{equation}\label{eqn:link_density_P}\begin{split}
			\limsup_{n\to\infty} \frac1n \log p_{\tube^*,n}(L,P, \leq\! \epsilon_{L,P} n) &< \lim_{n\to\infty} \frac1n \log p_{\tube^*,n}(L) \\ &=  \log\mu_{\tube^*,0_1} = \kappa_{\tube^*}(0_1),
	\end{split}\end{equation}
	where $n$ is taken through multiples of 2.
\end{cor}

We delay the proof of  this Corollary to Section~\ref{ssec:generallinkpatternthm}.
First we establish several connections 
between the scaling form established in 
~\eqref{finalform}, the consequences of Corollary~\ref{cor:density_of_connectsumpatterns} and the \textit{localization} of knots in polymers.  

Here the term localization refers to the size of the ``knotted part'' of a polygon, relative to the length of the whole polygon.
Different approaches have been suggested for defining the size of the knotted part of a polygon (e.g.~\cite{micheletti_polymers_2011}). A natural definition of this size $k(\pi)$ for an $n$-edge embedding $\pi$ of a link $L$ in a tube is as follows (see also \cite{BEISS18}). First divide the embedding at each 2-section to create a set of patterns which can each be closed off into embeddings of links. Up to $f_L$ of the resulting embeddings will be non-trivial links and the sizes of these non-trivial pieces can be summed to give the $k(\pi)$, \emph{size of the linked part of $\pi$}.

With this definition in hand, one can measure the average size of the knotted part $\langle k\rangle_n$ of an embedding of length $n$ and link type $L$, according to whichever probability distribution is appropriate (e.g.\ the uniform distribution). Letting $n$ get large, we can compare $\langle k\rangle_n$ to $n$. If $\langle k\rangle_n = o(n)$, then $L$ is said to be ``localized'' in the embeddings. Otherwise it is said to be ``not localized''.   Further, if $\langle k \rangle_n = O(1)$, then $L$ is said to be ``strongly localized'' in the embeddings.


We make two initial observations.  First, 
%
from Corollary~\ref{cor:density_of_connectsumpatterns},
all but exponentially few sufficiently large $n$-edge embeddings in $\tube^*$ of non-split $L$ contain a density  ($\epsilon n$) of 2-sections.  Dividing an embedding with $\epsilon n$ 2-sections at those 2-sections yields connected sum patterns with average size $1/\epsilon$.  Thus one expects the link patterns to have size $O(1)$, i.e.\ to be strongly localized.  Second, we note that the upper bound proof of Theorem~\ref{thm-bounds} establishes that one can unknot a link by the insertion of $f_L$ $O(1)$-size braid blocks, so that each ``essential'' part of the link can be removed by making $O(1)$ changes, another indication of strong localization. We next discuss further results that expand on these observations. 
Detailed proofs related to this discussion are given in Section~\ref{ssec:localglobalresults}. 

Consider first
 sets of embeddings of a link $L$ in $\tube^*$ 
 such that the size of the linked part 
 is $\alpha n =O(n)$, for any fixed $0<\alpha <1$. Such a linked-part-size limits the number of 2-sections that can occur and thus by Corollary~\ref{cor:density_of_connectsumpatterns}  these embeddings are exponentially rare amongst all embeddings of $L$.
 This is made more precise in Corollary~\ref{cor:corollarys4} of Section~\ref{ssec:localglobalresults}. 
In contrast, sets of embeddings of $L$ where the linked part is strongly localized (bounded in size by a fixed amount) are \emph{not} exponentially rare; instead the embedding counts follow the same scaling form as $p_{\tube^*,n}(L)$ 
(see Section~\ref{ssec:localglobalresults}, Corollary~\ref{cor:corollarys5}).      

Furthermore,  the conclusions of Theorem~\ref{thm-bounds} hold for any subset of embeddings of $L$ where
some element of the set can be obtained from an unknot polygon by inserting any fixed-size link patterns corresponding to the prime factors of $L$.
Thus, aside from the size of the linked part, one can also restrict other geometric features of the linked part and get the  same scaling form. 
In particular, for the case of a prime knot $L$, two modes have been identified for connected sum
knot patterns: the 2-filament mode (also called double filament or non-local) and the 1-filament mode (also called single filament or local) \cite{Sharma19, Suma17,BEISS18}. See Figure~\ref{fig2tubedef} for an example of each type of trefoil pattern in $\tube^*$; see \cite{BEISS18} for a full definition. Restricting to knot patterns in either one of these modes yields the same types of bounds as in  Theorem~\ref{thm-bounds}. Hence, for a prime knot, the number of embeddings  with the knot occurring in the 1-filament mode follows the same scaling form as that for the 2-filament mode, except for the constant factor 
(Corollary~\ref{cor:corollarys6}).  
Thus the likelihood of one type of knot mode occurring over another is determined  by the constants in the scaling form. For $\tube^*$, numerical evidence indicates that the 2-filament mode is more likely \cite{BEISS18}.  This could be related to the fact \cite[Result 5]{BEISS18} that the span of a smallest 2-filament mode knot pattern is smaller than that for a 1-filament mode knot pattern in $\tube^*$. (Figure ~\ref{fig2tubedef} shows smallest trefoil patterns of each type in $\tube^*$.)
Recent numerical results based on data from \cite{eng_phd_thesis} have shown that the difference between the occurrence probabilities of the two modes decreases with tube size although the 2-filament mode still dominates at tube size $3\times2$.  
Note that if one measures knot-size by the arc length measure of~\cite{micheletti_orlandini_2012} instead of by the knot pattern size defined here, then the 1-filament mode patterns yield comparable knot-sizes by either measure but the 2-filament mode patterns could have a substantially larger knot-size by the arc-length measure.  
Based on this, the off-lattice model results of~\cite{micheletti_orlandini_2012} at small tube sizes are consistent with knot localization and with the 2-filament mode being more likely than the 1-filament mode \cite{BEISS18}.

\begin{figure}[h!]
\centering
\begin{subfigure}
{0.45\textwidth}
\centering
\resizebox{0.833\textwidth}{!}{
\begin{tikzpicture}[rotate around x=0, scale=1.7]
\tikzset{vblack/.style={circle, draw=black, line width=1.pt, fill=black, inner sep=2pt}}
\tikzset{vblack/.style={circle, draw=black, line width=1.pt, fill=black, inner sep=2pt}}
\tikzset{vblack/.style={circle, draw=black, line width=1.pt, fill=black, inner sep=2pt}}
\node[vblack] at (1,0,0) {}; \node[vblack] at (1,1,0) {}; \node[vblack] at (1,2,0) {}; \node[vblack] at (1,0,1) {}; \node[vblack] at (1,1,1) {}; \node[vblack] at (1,2,1) {};
\node[vblack] at (2,0,0) {}; \node[vblack] at (2,1,0) {}; \node[vblack] at (2,2,0) {}; \node[vblack] at (2,0,1) {}; \node[vblack] at (2,2,1) {};
\node[vblack] at (3,0,0) {}; \node[vblack] at (3,1,0) {}; \node[vblack] at (3,2,0) {}; \node[vblack] at (3,0,1) {}; \node[vblack] at (3,2,1) {};
\node[vblack] at (4,0,0) {}; \node[vblack] at (4,1,0) {}; \node[vblack] at (4,2,0) {}; \node[vblack] at (4,0,1) {}; \node[vblack] at (4,1,1) {}; \node[vblack] at (4,2,1) {};
\node[vblack] at (5,0,0) {}; \node[vblack] at (5,1,0) {}; \node[vblack] at (5,2,0) {}; \node[vblack] at (5,0,1) {}; \node[vblack] at (5,1,1) {}; \node[vblack] at (5,2,1) {};
\begin{scope}[on background layer]
\begin{knot}[consider self intersections=true, clip width=3, clip radius=5.5pt, end tolerance=1pt]
\strand[line width=2pt, blue] (0.5,0,0) -- (1,0,0) -- (2,0,0) -- (2,1,0) -- (3,1,0) -- (4,1,0) -- (4,1,1) -- (4,0,1) -- (5,0,1) -- (5.5,0,1);
\strand[line width=2pt, red] %
(5.5,2,0) -- (5,2,0) -- (4,2,0) -- (3,2,0) -- (2,2,0) -- (1,2,0) -- (1,1,0) -- (1,1,1) -- (1,0,1) -- (2,0,1) -- (3,0,1) -- (3,0,0) -- (4,0,0) -- (5,0,0) -- (5,1,0) -- (5,1,1) -- (5,2,1) -- (4,2,1) -- (3,2,1) -- (2,2,1) -- (1,2,1) -- (0.5,2,1) ;
\flipcrossings{1,2,5,6,7}
\end{knot}
\end{scope}
\end{tikzpicture}
}
\caption{}
\label{fig:polygon_2x1_2strings}
\end{subfigure}
\begin{subfigure}
{0.45\textwidth}
\resizebox{\textwidth}{!}{
\begin{tikzpicture}[rotate around x=0, scale=1.7]
\tikzset{vblack/.style={circle, draw=black, line width=1.pt, fill=black, inner sep=2pt}}
\tikzset{vblack/.style={circle, draw=black, line width=1.pt, fill=black, inner sep=2pt}}
\tikzset{vblack/.style={circle, draw=black, line width=1.pt, fill=black, inner sep=2pt}}
\node[vblack] at (1,0,0) {}; \node[vblack] at (1,1,0) {}; \node[vblack] at (1,2,0) {}; \node[vblack] at (1,0,1) {}; \node[vblack] at (1,1,1) {}; \node[vblack] at (1,2,1) {};
\node[vblack] at (2,0,0) {}; \node[vblack] at (2,1,0) {}; \node[vblack] at (2,2,0) {}; \node[vblack] at (2,0,1) {}; \node[vblack] at (2,2,1) {};
\node[vblack] at (3,0,0) {}; \node[vblack] at (3,1,0) {}; \node[vblack] at (3,2,0) {}; \node[vblack] at (3,0,1) {}; \node[vblack] at (3,2,1) {};
\node[vblack] at (4,0,0) {}; \node[vblack] at (4,1,0) {}; \node[vblack] at (4,2,0) {}; \node[vblack] at (4,0,1) {}; \node[vblack] at (4,1,1) {}; \node[vblack] at (4,2,1) {};
\node[vblack] at (5,0,0) {}; \node[vblack] at (5,1,0) {}; \node[vblack] at (5,2,0) {}; \node[vblack] at (5,0,1) {}; \node[vblack] at (5,1,1) {}; \node[vblack] at (5,2,1) {};
\node[vblack] at (6,2,0) {}; \node[vblack] at (6,0,1) {}; \node[vblack] at (6,1,1) {}; \node[vblack] at (6,2,1) {};  \node[vblack] at (6,0,0) {}; \node[vblack] at (6,1,0) {};
\begin{scope}[on background layer]
\begin{knot}[consider self intersections=true, clip width=3, clip radius=5.5pt, end tolerance=1pt]
\strand[line width=2pt, blue] (0.5,0,0) -- (1,0,0) -- (2,0,0) -- (2,1,0) -- (3,1,0) -- (4,1,0) -- (4,1,1) -- (4,0,1) -- (5,0,1) -- (6,0,1) -- (6,0,0) -- (6,1,0) -- (6,2,0) -- (5,2,0) -- (4,2,0) -- (3,2,0) -- (2,2,0) -- (1,2,0) -- (1,1,0) -- (1,1,1) -- (1,0,1) -- (2,0,1) -- (3,0,1) -- (3,0,0) -- (4,0,0) -- (5,0,0) -- (5,1,0) -- (5,1,1) -- (6,1,1) -- (6.5,1,1);
\strand[line width=2pt, red]
(6.5,2,1) -- (6,2,1) -- (5,2,1) -- (4,2,1) -- (3,2,1) -- (2,2,1) -- (1,2,1) -- (0.5,2,1); %
\flipcrossings{1,2,6,7,8}%
\end{knot}
\end{scope}
\end{tikzpicture}
}
\caption{}
\label{fig:polygon_2x1_2stringslocal}
\end{subfigure}
\caption{\subref{fig:polygon_2x1_2strings}~A smallest trefoil knot pattern in $\tube^*$; its span is 5.
This is a 2-filament pattern  because both strands of the pattern are needed to create the trefoil. \subref{fig:polygon_2x1_2stringslocal}~A smallest 1-filament knot pattern in $\tube^*$; its span is 6. Only one strand (the blue one) contains the trefoil.}
\label{fig2tubedef}
\end{figure}

Regarding different modes of occurrence for two-component links, we note that prime 4-plats have either one or two components and if they have two components, both components are unknots \cite{BZ13}. 
For a prime two-component link, we define two modes of link patterns: those where all vertices of one component (called here the \textit{one-polygon} mode) are contained in the link pattern and those where that is not the case (called here the \textit{two-polygon} mode)  
(see
Figure~\ref{fig:2modesHopfLink}).  Again, embeddings of each type will have the same scaling form, up to the constant term (Corollary~\ref{cor:corollarys7}).  It appears that for the Hopf link, the one-polygon mode link pattern is smaller in span than the two-polygon mode.  For embeddings of the Hopf link or the Hopf link connected sum with a single prime knot, we thus expect that the one-polygon mode will be more probable in $\tube^*$. However, if both components of the Hopf link are each connected-summed with a prime knot then the two-polygon mode of the Hopf link is the only option.
Similarly, suppose $L=K_1\# L_1\# K_2$ where $L_1$ is a 2-component 2-bridge link and $K_1$ and $K_2$ are 2-bridge knots added to different components of $L_1$. Then neither component of $L$  can  occur in a single-polygon mode. 

\begin{figure}
	\centering
	\begin{tikzpicture}
		\begin{knot}[consider self intersections=true, clip width=3, clip radius=3pt, end tolerance=1pt]
			\strand [very thick, blue] (3.5,1,0) -- (3,1,0) node [vblue] {} -- (2,1,0) node [vblue] {} -- (2,1,1) node [vblue] {} -- (2,0,1) node [vblue] {} -- (1,0,1) node [vblue] {} -- (0,0,1) node [vblue] {} -- (0,0,0) node [vblue] {} -- (0,1,0) node [vblue] {} -- (0,2,0) node [vblue] {} -- (1,2,0) node [vblue] {} -- (2,2,0) node [vblue] {} -- (3,2,0) node [vblue] {} -- (3.5,2,0);
			\strand [very thick, red] (-0.5,1,1) -- (0,1,1) node [vred] {} -- (1,1,1) node [vred] {} -- (1,1,0) node [vred] {} -- (1,0,0) node [vred] {} -- (2,0,0) node [vred] {} -- (3,0,0) node [vred] {} -- (3,0,1) node [vred] {} -- (3,1,1) node [vred] {} -- (3,2,1) node [vred] {} -- (2,2,1) node [vred] {} -- (1,2,1) node [vred] {} -- (0,2,1) node [vred] {} -- (-0.5,2,1);
			\flipcrossings{1,4,5}
		\end{knot}
		\begin{scope}[yshift=-3cm]
			\begin{knot}[consider self intersections=true, clip width=3, clip radius=3pt, end tolerance=1pt]
				\strand [very thick, blue] (-0.5,0,1) -- (0,0,1) node [vblue] {} -- (1,0,1) node [vblue] {} -- (1,1,1) node [vblue] {} -- (1,1,0) node [vblue] {} -- (2,1,0) node [vblue] {} -- (2.5,1,0);
				\strand [very thick, blue] (-0.5,2,0) -- (0,2,0) node [vblue] {} -- (1,2,0) node [vblue] {} -- (2,2,0) node [vblue] {} -- (2.5,2,0);
				\strand [very thick, red] (0,1,1) node [vred] {} -- (0,1,0) node [vred] {} -- (0,0,0) node [vred] {} -- (1,0,0) node [vred] {} -- (2,0,0) node [vred] {} -- (2,0,1) node [vred] {} -- (2,1,1) node [vred] {} -- (2,2,1) node [vred] {} -- (1,2,1) node [vred] {} -- (0,2,1) node [vred] {} -- cycle;
				\flipcrossings{3,4,5}
			\end{knot}
		\end{scope}
	\end{tikzpicture}
	\caption{Two modes of the Hopf link: the two-polygon mode (top) appears to require more edges than the one-polygon mode (bottom).
	}\label{fig:2modesHopfLink}
\end{figure}


In the following sections we give the detailed proofs of the results outlined above.
In particular, we prove how Theorem~\ref{thm-bounds} leads to the general pattern theorem for fixed-link-type embeddings (Corollary \ref{cor:density_of_connectsumpatterns}) and how this in turn yields conclusions about the size of the linked region. 

\subsection{Proof of a general pattern theorem for fixed-link-type embeddings} \label{ssec:generallinkpatternthm}



\begin{proof}[Proof of Corollary \ref{cor:density_of_connectsumpatterns}]
	Consider $P$ to be a 2-section or an unknot connected sum pattern.
	Suppose first that $L$ is prime. 
	Then any embedding of $L$ with at most $k$ translates of $P$ can be converted to an unknot polygon by the insertion of a braid block from $S_L$.  The braid block has span $s_L$ and hence this insertion creates at most $s_L$ additional translates of $P$. Thus there exists a number $b_L\geq 0$ such that 
	\begin{equation}
		p_{\tube^*,n}(L, P, \leq k) \leq    \binom{ n+b_L }{1} p_{\tube^*,n+b_L}(0_1,P, \leq k+s_L).
	\end{equation}
	More generally for $L$ with $f_L$ prime factors, there exists a number $B_L\geq 0$ such that 
	\begin{equation}
		p_{\tube^*,n}(L, P, \leq k) \leq    \binom{ n+B_L }{f_L} p_{\tube^*,n+B_L}(0_1,P, \leq k+f_Ls_L).
		\label{eqkversion}
	\end{equation}
	Fix an $\epsilon_P$ from Corollary~\ref{thm:density_of_connectsumpatterns} and consider $n$ sufficiently large so that $-\epsilon_P/2< (\epsilon_PB_L-f_Ls_L)/n < \epsilon_P/2$, and take $0<\epsilon_{L,P}<\epsilon_P/2$. Hence 
	$\epsilon_{L,P} < \epsilon_{P}+(\epsilon_PB_L-f_Ls_L)/n$ so that $\epsilon_{L,P}n+f_Ls_L < \epsilon_{P}(n+B_L)$. Then setting $k=\epsilon_{L,P}n$ in~\eqref{eqkversion}, taking logarithms and $\limsup$s yields the required result:
	\begin{align}
			\limsup_{n\to\infty} \frac1n \log p_{\tube^*,n}(L,P, \leq\! \epsilon_{L,P} n) &\leq \limsup_{n\to\infty} \frac1n \log p_{\tube^*,n+B_L}(0_1,P, \leq\! \epsilon_{P}(n+B_L)) \notag \\ & <\lim_{n\to\infty} \frac1n \log p_{\tube^*,n}(0_1) \notag \\ & =\lim_{n\to\infty} \frac1n \log p_{\tube^*,n}(L) = \log\mu_{\tube^*,0_1}= \kappa_{\tube^*}(0_1).
	\end{align}
\end{proof}

\subsection{Proofs regarding the size and mode of the linked region}\label{ssec:localglobalresults}

In this section we give more details of the arguments outlined in the introduction 
of this section
regarding the size and mode of the linked region.

We establish first that sets of embeddings of a non-split link $L$ in $\tube^*$ for which the size of the linked part of $L$ is $O(n)$ are exponentially rare amongst all embeddings of $L$.
\begin{cor}[Corollary of 
Corollary~\ref{cor:density_of_connectsumpatterns}
]\label{cor:corollarys4}
	Given any $\alpha \in(0,1]$ and even $n$, let $p_{\tube^*,n}(L,\alpha n)$  be the number of  embeddings of non-split $L$ in $\tube^*$ with at least one vertex in the plane $x=0$ and such that the size of the linked part is  $k_n=\lfloor \alpha n\rfloor$. Then
	\begin{equation}
		\limsup_{n\to\infty} \frac1n \log p_{\tube^*,n}(L,   \alpha n)  < \log\mu_{\tube^*,0_1}.
	\end{equation}
\end{cor}
\begin{proof}
	Fix any $\alpha \in(0,1]$ and, for each even $n$, consider the set of $n$-edge embeddings of $L$ where the size of the linked part is  $k_n=\lfloor \alpha n\rfloor$. Given an even $n$, when any embedding in the corresponding set is divided at its 2-sections, the number of edges in the (up to $f_L$) linked parts adds up to $k_n$. If there are $j\leq f_L$ linked parts, then the embedding can be thought of as a connected sum of $j$ link patterns with $j+1$ unknot patterns (where some unknot patterns may be empty). Each link pattern has a non-trivial link type and the number of ways that $j$ non-trivial links can give connected sum $L$ is bounded above by $f_L!\binom{f_L-1}{j-1}$. Further, each of the link patterns has no internal 2-sections. The number of ways to distribute $k_n$ edges to $j$ link patterns is  at most $\binom{k_n -1}{j-1}$ and the number of ways to distribute the remaining edges is $\binom{(n-k_n) +j}{j}$.  The $j$ link patterns can be concatenated to create an embedding of $L$ with at most $2j$ 2-sections and the unknot patterns can be concatenated to create an unknot polygon.  
	By taking $P$ to be 2-sections, it follows that
	{\small
	\begin{align}
		p_{\tube^*,n}&(L, \alpha n) 
  \\&\leq \sum_{j=1}^{f_L} \sum_{(L_l | L_1\#L_2...\#L_j=L)} \sum_{(m_l | \sum_{l=1}^{j+1} m_l=n-k_n)}\sum_{(n_l | \sum_{l=1}^j n_l =k_n)} \left(\prod_{l=1}^{j} p_{\tube^*,m_l}(0_1)p_{\tube^*,n_l}(L_l,P, \leq 2)\right) p_{\tube^*,m_{j+1}}(0_1) \nonumber \\
		&\leq   g(f_L) \binom{\lfloor \alpha n\rfloor -1}{f_L-1} \binom{\lfloor(1-\alpha)n\rfloor +2 +f_L}{f_L} p_{\tube^*,\lfloor(1-\alpha) n\rfloor+Q_L}(0_1)p_{\tube^*,\lfloor\alpha n\rfloor +C_L}(L,P, \leq 2f_L +c_L),
	\end{align}
}
	where $g(f_L)=(f_L!)2^{f_L-1}$, $p_{\tube^*,0}(0_1)\equiv 1$, and $Q_L,C_L,c_L$ are constants. We also assume that $n$ is sufficiently large that $\lfloor \alpha n\rfloor -1 \geq 2(f_L-1)$. 
	Taking logarithms, dividing by $n$ and taking the $\limsup$ as $n\to\infty$ then gives: 
	\begin{align}
		\limsup_{n\to\infty} \frac1n &\log p_{\tube^*,n}(L,   \alpha n)\notag \\ &\leq   (1-\alpha)\log\mu_{\tube^*,0_1} + \alpha\limsup_{n\to\infty} \frac{1}{\lfloor \alpha n\rfloor +C_L} \log p_{\tube^*,\lfloor \alpha n\rfloor +C_L}(L,P, \leq 2f_L +c_L) \notag \\  &< \log\mu_{\tube^*,0_1},
	\end{align}
	where the strict inequality comes from Corollary~\ref{cor:density_of_connectsumpatterns} (as soon as $n$ is large enough that $2f_L+c_L\leq \epsilon_Pn$).
\end{proof}

In contrast, sets of embeddings of non-split $L$ where the linked part is strongly localized are \emph{not} exponentially rare; instead the embedding counts follow the same scaling form as $p_{\tube^*,n}(L)$. Let $m_L$ be the minimum size of a link pattern of $L$. 
\begin{cor}\label{cor:corollarys5}
	Given any integer $N\geq m_L$, for any even integer $n\geq m_L$, let $p_{\tube^*,n}(L,\leq  N)$  be the number of  $n$-edge embeddings of $L$  in $\tube^*$ with at least one vertex in the plane $x=0$ and such that the size of the linked part is at most  $N$. Then there exist constants $C_{m_L},C_2$ such that for  sufficiently large $n$, 
	\begin{equation}
		C_{m_L}n^{f_L} p_{\tube^*,n}(0_1) \leq    p_{\tube^*,n}(L, \leq N) \leq C_2n^{f_L} p_{\tube^*,n}(0_1).  
		\label{finalformN}
	\end{equation}
\end{cor}
The proof is exactly the same as the proof  of Theorem~\ref{thm-bounds}{} where for the lower bound one uses a size $m_L$ link pattern to insert.      

Furthermore,  the conclusions of Theorem~\ref{thm-bounds}{} hold for any subset of embeddings of $L$ where, as in the proof of the lower bound, an element of the set can be obtained by inserting any fixed-size link patterns corresponding to the prime factors of $L$.
Thus, aside from the size of the linked part one can also restrict other geometric features of the linked part and get the  same scaling form. 
In particular, for the case of $L$ a prime knot, 
restricting to knot patterns in either one of the two modes, 1- or 2-filament, will yield the same types of bounds as in  Theorem~\ref{thm-bounds}{}.
Similarly, for 2 component non-split 4-plats the two modes of linking defined earlier  (one-polygon and two-polygon) follow the same scaling form as in Theorem~\ref{thm-bounds}{}. These results are summarized in the following two example corollaries.
\begin{cor}\label{cor:corollarys6}
	Given any one component $4$-plat 
	knot $K$, even integer $n$ and $i\in \{1,2\}$, let $p_{\tube^*,n}(K; i)$  be the number of  knot-type $K$ polygons in $\tube^*$ with at least one vertex in the plane $x=0$ and such that the polygon contains a link pattern of $K$ in the $i$-filament mode where $i$ is 1 or 2. Then there exist constants $C_{K,i},C_2$ such that for  sufficiently large $n$, 
	\begin{equation}
		C_{K,i}\,n\, p_{\tube^*,n}(0_1) \leq    p_{\tube^*,n}(K;i) \leq C_2 \,n\, p_{\tube^*,n}(0_1).  
		\label{finalformi}
	\end{equation}
\end{cor}

\begin{cor}\label{cor:corollarys7}
	Given any $4$-plat 
	2 component non-split link $L$, even integer $n$ and $i\in \{1,2\}$, let $p^{(2)}_{\tube^*,n}(L; i)$  be the number of  link-type $L$ embeddings in $\tube^*$ with at least one vertex in the plane $x=0$ and such that the embedding contains a link pattern of $L$ in the $i$-polygon mode. Then there exist constants $C_{L,i},C_2$ such that for  sufficiently large $n$, 
	\begin{equation}
		C_{L,i}\,n\, p_{\tube^*,n}(0_1) \leq    p^{(2)}_{\tube^*,n}(L;i) \leq C_2\,n\, p_{\tube^*,n}(0_1).  
		\label{finalformlinki}
	\end{equation}
\end{cor}

\section{Discussion}
\label{sec:discussion}
This paper is primarily devoted to proving Theorem~\ref{thm-bounds} that establishes the KE Conjecture for the case of polygons in the lattice tube $\tube^*$. For the upper bound, we proved new knot theory results about 4-plat diagrams (see Section \ref{sec:newupper}). For the lower bound, we proved a new pattern theorem for unknot polygons based on exact transfer-matrix results for polygons in $\tube^*$ (see Section \ref{sec:newlower}). Analogous results can be obtained if one takes the size of an embedding to be its \emph{span} rather than the number of edges.

Bounds of the form in Theorem~\ref{thm-bounds} Equation~\eqref{finalform} are expected to hold for any tube size and in the limit, as the tube dimensions go to infinity, i.e.\ for $\mathbb{Z}^3$ \cite{Bonato_2021, Orl98, Rechnitzer}. Equation~\eqref{unknottubegrowthconstantlimit} suggests that the proof of such bounds for arbitrary tubes could lead to results for the unconfined case. Here, we have established new strategies for proving such bounds. 
For the upper bound, instead of trying to unknot the embedding of $L$ with a series of deletions, we do it by inserting patterns.  For the lower bound,
we prove that unknot polygons have a density of 2-sections, where a 2-section is a place where the polygon can be divided into connected summands. Although these strategies do not immediately generalize to all tube sizes, they provide promising new directions. 

Our proofs identify some fundamental mathematical challenges towards proving the KE Conjecture more generally. For example, the upper bound arguments do not easily extend
to a $3\times 1$ tube
since the extra configurational freedom means that obtaining a 4-plat diagram from an embedding of a 4-plat link will require additional combinatorial analysis. 4-plat diagrams are needed to find the location to insert the braid block that unknots it.  

For the lower bound, the challenge for extending to larger tubes is two-fold. 
To extend Lemma~\ref{lem:tube*_no2sections_upperbound}, we need the transfer matrix for polygons in an $M\times N$ tube (these have been computed for tubes up to $5\times 1$ and $3\times 2$). For Lemma~\ref{lem:tube*_unknots_lowerbound}, we need an accurate count of unknot polygons up to a sufficiently large size. In general this involves generating polygons and checking their knot-type. Existing computational resources  are expected to yield a proof of the lower bound
for other small tube sizes. Proofs for large tube sizes are limited only by available computational resources.

Knots and links embedded in tubes, even tubes as narrow as $\tube^*$, are of independent interest to polymer scientists and molecular biologists. Some structures of interest in molecular biology are co-transcriptional R-loops, which have been the subject of many recent studies due to their biological relevance and potential impact to human health \cite{Chedin-survey,sanz_chedin_2016}. R-loops are prevalent 3-strand nucleic acid structures that form during transcription \cite{sanz_chedin_2016}. They consist of a DNA:RNA hybrid and a DNA single strand. A detailed topological characterization of the entanglement of R-loops is still lacking \cite{Chedin2020, Jonoska2020ModelingRH, Stolz2019, Liu_R-loops, Jonoska_R-loop_Grammar}. Modeling R-loops as chain embeddings in a tube is an ideal way to capture the properties of these braided configurations as 3D objects, and not just as symbolic objects such as those represented by braid words. The methods in this paper suggest ways to explore R-loop embeddings combinatorially and to characterize the breadth of possible topological configurations adopted by them.

Experimental results have sparked interest in computational modeling of DNA knot translocation through narrow channels (e.g.\ \cite{Suma17}; reviewed in \cite{Orlandini_2021}). Using nanopore sensors, scientists have successfully detected complex DNA knots ~\cite{Sharma19, Plesa_2016} at length scales one order of magnitude larger than traditional gel electrophoresis~\cite{Stasiak_1996, Trigueros_2001}. Using traditional DNA topology methods, detection is limited to DNA chains in the
$10^3$ base pair range \cite{Cebri_n_2014, Crisona_1999, Crisona_1994, Grainge_2007}. Tube $\tube^*$ sets a theoretical framework for studying knotting and linking in narrow channels. Here we provide a robust theoretical justification that the connected components of knots and links embedded in narrow tubular regions tend on average to be localized and separated from each other in $\tube^*$, somewhat like  pearls on a string. We show that knot factors occur as though they were distributed binomially along an unknot polygon in $\tube^*$. This  is consistent with the Poisson distribution observed for a variety of tube sizes in the off-lattice model studied in~\cite{micheletti_orlandini_2012}.

\section*{Acknowledgments}

NRB is grateful for support from the Australian Research Council, in particular grant DE170100186.
KI was supported by Japan Society for the Promotion of Science (JSPS) KAKENHI Grant Number 17K14190.
KS is partially supported by Japan Society for the Promotion of Science (JSPS) KAKENHI Grant Numbers 16H03928, 16K13751, 19K21827, 21H00978, 23K20791. CES acknowledges the support of the Natural Sciences and Engineering Research Council of Canada (NSERC) [funding reference number RGPIN-2020-06339], as well as Compute Canada resource allocations and past CRG support from the Pacific Institute for the Mathematical Sciences (PIMS). MV was supported by the National Science Foundation (NSF) Division of Mathematical Sciences (DMS) grants 1716987, 1817156 and 2054347.

\sloppy
\printbibliography

\end{document}